\documentclass[a4paper,leqno]{amsart}

\usepackage{latexsym}
\usepackage[english]{babel}
\usepackage{fancyhdr}
\usepackage[mathscr]{eucal}
\usepackage{amsmath}
\usepackage{mathrsfs}
\usepackage{amsthm}
\usepackage{amsfonts}
\usepackage{amssymb}
\usepackage{amscd}
\usepackage{bbm}
\usepackage{graphicx}
\usepackage{graphics}
\usepackage{latexsym}
\usepackage{color}
\usepackage{dsfont}

\newcommand{\ud}{\mathrm{d}}

\newcommand{\im}{\mathfrak{Im}}

\newcommand{\ii}{\mathrm{i}}

\newcommand{\C}{\mathbb C}

\newcommand{\R}{\mathbb R}

\theoremstyle{plain}
\newtheorem{theorem}{Theorem}[section]
\newtheorem{lemma}[theorem]{Lemma}

\newtheorem{corollary}[theorem]{Corollary}
\newtheorem{proposition}[theorem]{Proposition}

\theoremstyle{definition}

\newtheorem{remark}[theorem]{Remark}
\newtheorem*{remark*}{Remark}

\numberwithin{equation}{section}

\begin{document}

\title[Fractional powers of singular perturbations of the Laplacian]
{On fractional powers of singular perturbations of the Laplacian}
\author[V.~Georgiev]{Vladimir Georgiev}
\address[V.~Georgiev]{Department of Mathematics \\ University of Pisa \\ Largo Bruno Pontecorvo 5 \\ 56127 Pisa (Italy)  \\ and
\\ Faculty of Science and Engineering \\ Waseda University \\
 3-4-1, Okubo, Shinjuku-ku, Tokyo 169-8555 \\
Japan}
\email{georgiev@dm.unipi.it}
\author[A.~Michelangeli]{Alessandro Michelangeli}
\address[A.~Michelangeli]{International School for Advanced Studies -- SISSA \\ via Bonomea 265 \\ 34136 Trieste (Italy).}
\email{alemiche@sissa.it}
\author[R.~Scandone]{Raffaele Scandone}
\address[R.~Scandone]{International School for Advanced Studies -- SISSA \\ via Bonomea 265 \\ 34136 Trieste (Italy).}
\email{rscandone@sissa.it}

\begin{abstract}
We qualify a relevant range of fractional powers of the so-called Hamiltonian of point interaction in three dimensions, namely the singular perturbation of the negative Laplacian with a contact interaction supported at the origin. In particular we provide an explicit control of the domain of such a fractional operator and of its decomposition into regular and singular parts. We also qualify the norms of the resulting singular fractional Sobolev spaces and their mutual control with the corresponding classical Sobolev norms.
\end{abstract}

\date{\today}


\thanks{Partially supported by the 2014-2017 MIUR-FIR
grant ``\emph{Cond-Math: Condensed Matter and Mathematical Physics}'' code RBFR13WAET. The first author was supported in part by  INDAM, GNAMPA - Gruppo Nazionale per l'Analisi Matematica, la Probabilita e le loro Applicazion, by Institute of Mathematics and Informatics, Bulgarian Academy of Sciences and by Top Global University Project, Waseda University.}

\maketitle


\section{Introduction}

It is customary to refer to a \emph{singular perturbation of the $d$-dimensional Laplacian} as a self-adjoint operator on $L^2(\mathbb{R}^d)$ that acts as the Laplacian on sufficiently smooth functions with compact support in $\mathbb{R}^d\!\setminus\!\{x_0\}$ for a given $x_0\in\mathbb{R}^d$. Apart from the trivial case of the Laplacian defined on $H^2(\mathbb{R}^d)$, the constraint of self-adjointness induces a non-trivial action on a larger domain of functions that do not necessarily vanish around $x_0$, whence the terminology of singular perturbation at the point $x_0$. More generally one speaks of singular perturbations supported on a manifold $\Sigma\subset\mathbb{R}^d$, a special case of which is when $\Sigma$ is discrete and consists of finitely many or infinite points -- in fact, the higher the co-dimension of $\Sigma$, the more complicated the perturbation.

Singular perturbations arise naturally in the context of quantum systems of particles subject to inter-particle interactions, or to interactions with certain fixed points or surfaces, which have an extremely short range, virtually equal to zero. They correspond to formal Schr\"{o}dinger operators $-\Delta+\delta_\Sigma$, an idealisation of an ordinary Schr\"{o}dinger operator $-\Delta+V_\Sigma$ where the potential is very much peaked and shrunk around the manifold $\Sigma$. Remarkably, this turns out to be a very realistic model of various systems under suitable experimental conditions, such as ultra-cold gases with almost zero-range interactions or quantum particles in interaction with wires or membranes.

Mathematically, however, although the replacement of an actual (and possibly very complicated) potential $V_\Sigma$ with a formal $\delta_\Sigma$ results in a significant simplification of many formal computations, yet the general theory is more involved than the theory of ordinary Schr\"{o}dinger operators.

The object of this work is the prototypical case of one-point perturbations in $d=3$ dimensions, setting for concreteness $x_0=0$: thus, the formal operator of interest is $-\Delta+\delta(x)$, which can be thought of as the Hamiltonian of a two-particle system with a delta-like interaction in the relative variable $x=x_1-x_2$.

More rigorously, we deal with the self-adjoint extensions of the positive and densely defined symmetric operator $-\Delta|_{C^\infty_0(\mathbb{R}^3\setminus\{0\})}$ on $L^2(\mathbb{R}^3)$. This is today a well-known class of operators, since the first rigorous attempt \cite{Berezin-Faddeev-1961} by Berezin and Faddeev in 1961 and the seminal work \cite{AHK-1981-JOPTH} by Albeverio and H\o{}egh-Krohn in 1981. It is in fact a one-parameter family $\{-\Delta_\alpha\,|\,\alpha\in(-\infty,+\infty]\}$ of self-adjoint extensions of $-\Delta|_{C^\infty_0(\mathbb{R}^3\setminus\{0\})}$, the parameter $\alpha$ expressing, in suitable units, the inverse scattering length of the interaction supported at $x_0=0$. The special extension relative to $\alpha=\infty$ is the self-adjoint negative Laplacian on $L^2(\mathbb{R}^3)$, with domain $H^2(\mathbb{R}^3)$, all other extensions representing non-trivial operators of point interaction at $x_0=0$.
We cast in Section \ref{sec:3D-delta} a detailed summary of the main features and properties of the operator $-\Delta_\alpha$.


Each realisation $-\Delta_\alpha$ is semi-bounded from below, and positive for $\alpha\geqslant 0$. Thus, up to a non-essential shift we can restrict ourselves to the case of non-negative $-\Delta_\alpha$, namely, non-negative $\alpha$. In this respect, our concern for the present work is the qualification of the fractional powers of the operator $-\Delta_\alpha$, primarily the domain and the action of such powers. We therefore focus on the operators $(-\Delta_\alpha)^{s/2}$, $s\in\mathbb{R}$ thus denoting the number of `\emph{singular fractional derivatives}', aiming at covering the regime of main relevance, that is, $s\in(0,2)$ (the power $s=0$ corresponds to the identity operator, the power $s=2$ corresponds to the actual $-\Delta_\alpha$).

Among the motivations for the interest on $(-\Delta_\alpha)^{s/2}$, central is surely the observation that the domain of $(-\Delta_\alpha)^{s/2}$ provides a `singular-perturbed' version of the classical Sobolev space $H^s(\mathbb{R}^3)$ -- we shall denote it with $H_\alpha^s(\mathbb{R}^3)$ in our results. In turn, the knowledge of such singular Sobolev spaces, of their induced singular Sobolev norms, and of the mutual control between classical and singular Sobolev norms, constitutes a crucial tool for the study of the well-posedness of semi-linear `singular' Schr\"{o}dinger equations of the form
\[
 \ii\partial_t u\;=\;(-\Delta_\alpha)u+\mathcal{N}(u)
\]
with non-linearities of relevance such as $\mathcal{N}(u)=|u|^\gamma u$ or $\mathcal{N}(u)=|x|^{-\gamma}*|u|^2$, $\gamma>0$. These are non-linear PDE's that model, in a suitable regime, the presence of a localised impurity.
The analogue of the energy space would therefore be $H_\alpha^1(\mathbb{R}^3)$, and one would like to address also a higher or lower regularity theory, whence the importance of the understanding of the spaces $H_\alpha^s(\mathbb{R}^3)$. Dispersive and scattering properties of the linear propagator $t\mapsto e^{\ii t\Delta_\alpha}$ have been the object of past, as well as recent extensive studies \cite{Scarlatti-Teta-1990,Dancona-Pierfelice-Teta-2006,Iandoli-Scandone-2017,DMSY-2017}, that include Strichartz estimates and $L^p$-boundedness of the associated wave operators, which would be then natural to combine with a systematic information on the singular fractional Sobolev spaces.

On a more technical level, one of our main questions and of the crucial properties in applications, concerns the \emph{structure} of a generic function in the singular Sobolev space $H_\alpha^s(\mathbb{R}^3)$. It is indeed well known (and we review it in Section \ref{sec:3D-delta}) that the domain of $-\Delta_\alpha$ consists of functions that are decomposable uniquely into a `regular' $H^2$-part plus a `singular' part that is a multiple of the Green's function of the three-dimensional negative Laplacian, i.e., $(4\pi|x|)^{-1}\exp(-|x|)$, with a very special constant of proportionality that qualifies the link between regular and singular part, and in fact is the precise signature of the interaction supported at the origin.

In the first of our main results, Theorem \ref{thm:domain_s}, we determine the precise structure of the singular Sobolev space $H_\alpha^s(\mathbb{R}^3)$, identifying regular and singular part of a generic $g\in H_\alpha^s(\mathbb{R}^3)$ in all the regimes of $s$ for which such decomposition is meaningful. In our second main result, Theorem \ref{thm:equiv_of_norms}, we present a mutual control between classical and singular Sobolev norms, and in our third main result, Theorem \ref{thm:fract_pow_formulas}, we find an explicit formula for the computation of $(-\Delta_\alpha)^{s/2}u$.

These results and related remarks are stated in Section \ref{sec:main_results}. In particular, there arise three natural regimes of increasing regularity, $s\in(0,\frac{1}{2})$, $s\in(\frac{1}{2},\frac{3}{2})$, and $s\in(\frac{3}{2},2)$: the first is so low that no canonical decomposition between regular and singular part is possible; the second is large enough to produce indeed a decomposition, however with no constraint between regular and singular component; the third is so high as to induce a constraint between the two components, which is completely analogous to what was already known for the space $H_\alpha^2(\mathbb{R}^3)$, i.e., the domain of $-\Delta_\alpha$. The transition cases $s=\frac{1}{2}$ and $s=\frac{3}{2}$ are discussed separately in Section \ref{sec:main_results} and then in Propositions \ref{prop:domain_s_12} and \ref{prop:domain_s_32}.

In Sections \ref{sec:canonical-decomposition} through \ref{sec:fractional_maps} we develop an amount of preparatory material for the proof of our main results, which is then the object of our concluding Section \ref{sec:proofs}. In particular, in Section \ref{sec:canonical-decomposition}  we establish a spectral-theorem-based canonical decomposition of the domain of $(-\Delta_\alpha)^{s/2}$ and in Section \ref{sec:regularity} we study the regularity of each term of such a decomposition. This leads us to identify convenient subspaces of the fractional space $H_\alpha^2(\mathbb{R}^3)$  in Section \ref{sec:subspaces_of_Ds}, an information that we find convenient for the sake of clarity to re-cast in an operator-theoretic language in terms of suitable fractional maps, Section \ref{sec:fractional_maps}. A final Appendix contains the detail of a Schur-test bound that we used systematically for the estimate of the norm of a number of integral operators.

\section{Three-dimensional Laplacian with singular perturbation}\label{sec:3D-delta}

The concise review in this Section is based on \cite[Chapter I.1]{albeverio-solvable} and \cite[Section 3]{MO-2016}.

The class of self-adjoint extensions in $L^2(\mathbb{R}^3)$ of the positive and densely defined symmetric operator $-\Delta|_{C^\infty_0(\mathbb{R}^3\setminus\{0\})}$ is a one-parameter family of operators $-\Delta_\alpha$, $\alpha\in(-\infty,+\infty]$, defined by
\begin{equation}\label{eq:op_dom}
 \mathcal{D}(-\Delta_\alpha)\;=\;\Big\{g\in L^2(\mathbb{R}^3)\,\Big|\,g=f_\lambda+\frac{f_\lambda(0)}{\alpha+\frac{\sqrt{\lambda}}{4\pi}}\,G_\lambda\textrm{ with }f_\lambda\in H^2(\mathbb{R}^3)\Big\}
\end{equation}
and
\begin{equation}\label{eq:opaction}
 (-\Delta_\alpha+\lambda)\,g\;=\;(-\Delta+\lambda)\,f_\lambda\,,
\end{equation}
where $\lambda>0$ is an arbitrarily fixed constant and
\begin{equation}\label{eq:defGlambda}
G_\lambda(x)\;=\;\frac{e^{-\sqrt{\lambda}\,|x|}}{4\pi |x|}\,.
\end{equation}
The above decomposition of a generic $g\in\mathcal{D}(-\Delta_\alpha)$ is unique and holds true for every chosen $\lambda$. The same formulas are valid also for $\lambda=-z^2$ for $z\in\mathbb{C}$, $\im z>0$.

The extension $-\Delta_{\alpha=\infty}$ is the Friedrichs extension and is precisely the self-adjoint $-\Delta$ on $L^2(\mathbb{R}^3)$ with domain $H^2(\mathbb{R}^3)$.

Thanks to the continuity of $f_\lambda$, the boundary condition holding for a generic $g\in\mathcal{D}(-\Delta_\alpha)$ reads $g(x)\approx f_\lambda(0)(1+(4\pi\alpha+\sqrt{\lambda})^{-1}|x|^{-1})$ as $x\to 0$, and hence also, owing to the arbitrariness of $\lambda>0$,
\begin{equation}\label{eq:BPcontact}
g(x)\;=\;\mathrm{const}\cdot\Big(\frac{1}{|x|}-\frac{1}{-(4\pi\alpha)^{-1}}\Big)+o(1)\qquad\mathrm{as}\;\;x\to 0\,.
\end{equation}
The latter is the short-range asymptotics typical of the low-energy bound state of a potential with almost zero support and $s$-wave scattering length $a=-(4\pi\alpha)^{-1}$, as was first recognised by Bethe and Peierls \cite{Bethe_Peierls-1935} (originally in the form $\partial_r(rg)-4\pi\alpha rg\to 0$), whence the name of \emph{Bethe-Peierls contact condition}.

Clear consequences of \eqref{eq:op_dom} or \eqref{eq:BPcontact} above are: on $H^2$-functions vanishing at $x=0$ the operator $-\Delta_\alpha$ acts precisely as $-\Delta$; moreover, the only singularity that the elements of $\mathcal{D}(-\Delta_\alpha)$ may display at $x=0$ is of the form $|x|^{-1}$. This suggests that $-\Delta g$ fails to be in $L^2(\mathbb{R}^3)$ by a distributional contribution removing which yields $-\Delta_\alpha g$. This is precisely what can be proved:
\begin{equation}\label{eq:renorm}
-\Delta_\alpha g\;=\;-\Delta g-\Big(\lim_{x\to 0}|x|g(x)\Big)\,\delta_0\,,\qquad g\in\mathcal{D}(-\Delta_\alpha)\,.
\end{equation}
Identity \eqref{eq:renorm} indicates that $-\Delta_\alpha g$ may be thought of a suitable \emph{renormalisation} of $-\Delta g$: in fact, in the r.h.s.~there is a difference of two distributions which gives eventually a $L^2$-function.

Another relevant form of the boundary condition for $g\in\mathcal{D}(-\Delta_\alpha)$ is available in  Fourier transform. The following limit is finite
\begin{equation}\label{eq:xi-finite}
\xi\;=\;\lim_{R\to+\infty}\frac{1}{4\pi R}\int_{\substack{\,p\in\mathbb{R}^3 \\ \! |p|<R}}\,{\widehat g(p) \,\ud p}
\end{equation}
and is customarily referred to as the \emph{charge} of $g$, in terms of which one has the asymptotics
\begin{equation}\label{eq:TMS_cond_asymptotics_1}
\int_{\substack{\,p\in\mathbb{R}^3 \\ \! |p|<R}}\,{\widehat g(p) \,\ud p}\;=\;4\pi\xi \,(R+2\pi^2\alpha)+o(1) \qquad\textrm{as}\qquad R\to +\infty\,.
\end{equation}
The latter is the so-called \emph{Ter-Martyrosyan--Skornyakov condition}, originally identified by Ter-Martyrosyan and Skornyakov \cite{TMS-1956}, and is in practice the Fourier counterpart of \eqref{eq:BPcontact}. One can show that imposing the Ter-Martyrosyan--Skornyakov condition at given $\alpha$ to the functions in the domain of the adjoint of $-\Delta|_{C^\infty_0(\mathbb{R}^3\setminus\{0\})}$ selects precisely $\mathcal{D}(-\Delta_\alpha)$.
The action of $-\Delta_\alpha$ in Fourier transform reads
\begin{equation}
\widehat{(-\Delta_\alpha g)}(p)\;=\;p^2 \widehat{g}(p)-\xi\;=\;p^2 \widehat{g}(p)-\lim_{R\to+\infty}\frac{1}{4\pi R}\int_{\substack{\,p\in\mathbb{R}^3 \\ \! |p|<R}}\,{\widehat g(p) \,\ud p}\,,
\end{equation}
which is the Fourier counterpart of \eqref{eq:renorm}.

Arbitrarily close to each $-\Delta_\alpha$, in the sense of resolvents, one finds an ordinary Schr\"{o}dinger operator $-\Delta+V$, with a potential $V$ suitably peaked and shrunk around $x=0$. More precisely, it can be proved that if $V:\mathbb{R}^3\to\mathbb{R}$ is measurable, $\langle x\rangle V\in L^1(\mathbb{R}^3)$, $V\in$ Rollnik, and $-\Delta+V$ is \emph{zero-energy resonant}, then setting
\begin{equation}
 \qquad V_\varepsilon(x)\;:=\;\frac{\eta(\varepsilon)}{\varepsilon^{2}}\,V(x/\varepsilon)\,,\qquad\eta\textrm{ smooth and }\eta(0)=1\,,
\end{equation}
one has
\begin{equation}\label{eq:NRlimit}
 (-\Delta+V_\varepsilon+\lambda\mathbbm{1})^{-1}\;\xrightarrow[]{\quad\varepsilon\downarrow 0\quad}\;(-\Delta_\alpha+\lambda\mathbbm{1})^{-1}
\end{equation}
in the norm operator sense,
for a value of $\alpha$ uniquely determined by the chosen $V$ and $\eta$. (\emph{Without} the zero-energy resonance the limit is the resolvent of the free (negative) Laplacian.)
Explicitly, the zero-energy resonance condition and the other assumptions above on $V$ imply the existence of $\phi\in L^2(\mathbb{R}^3)$ with
\[
\mathrm{sgn}(V)\sqrt{|V|}(-\Delta)^{-1}\sqrt{|V|}\,\phi\;=\;-\phi
\]
and
\[
\psi\;:=\;(-\Delta)^{-1}\sqrt{|V|}\,\phi\;\in\;L^2_\mathrm{loc}(\mathbb{R}^3)\!\setminus\! L^2(\mathbb{R}^3)\,,\qquad (-\Delta+V)\psi=0\quad\mathrm{in}\quad\mathcal{D}'(\mathbb{R}^3)\,,
\]
and in terms of such a resonance function the constant $\alpha$ emerging in the limit \eqref{eq:NRlimit} is given by
\[
\alpha\;=\;-\eta'(0)\,\Big|\int_{\mathbb{R}^3}V\psi\,\ud x\Big|^{-2}\;=\;-\eta'(0)\,\Big|\int_{\mathbb{R}^3}\sqrt{|V|}\,\phi\,\ud x\Big|^{-2}\,.
\]

The resolvent limit \eqref{eq:NRlimit} is intimately connected with the following resolvent identity:
\begin{equation}\label{eq:res_formula}
(-\Delta_\alpha+\lambda\mathbbm{1})^{-1}\;=\;(-\Delta+\lambda\mathbbm{1})^{-1}+\frac{1}{\alpha+\frac{\sqrt{\lambda}\,}{4\pi}}\,|G_\lambda\rangle\langle\overline{G_\lambda}|\,.
\end{equation}
It says that the resolvent of $-\Delta_\alpha$ is a rank-one perturbation of the free resolvent.

Furthermore, the following equivalent characterisation of $-\Delta_\alpha$ has the virtue of showing explicitly that the two operators $-\Delta_\alpha$ and $-\Delta$ only differ on the subspace of spherically symmetric functions.
The canonical decomposition
\begin{equation}\label{eq:L2_ell_decomposition}
L^2(\mathbb{R}^3)\;\cong\;\bigoplus_{\ell=0}^\infty L^2(\mathbb{R}^+,r^2\,\ud r)\otimes\mathrm{span}\{Y_{\ell,-\ell},\dots,Y_{\ell,\ell}\}\;\equiv\;\bigoplus_{\ell=0}^\infty L^2_\ell(\mathbb{R}^3)
\end{equation}
(where the $Y_{\ell,m}$'s are the spherical harmonics on $\mathbb{S}^2$) reduces $-\Delta_\alpha$ and for each $\ell\geqslant 1$ one has $-\Delta_\alpha|_{L^2_\ell}=-\Delta|_{L^2_\ell}$. On the sector $\ell=0$, namely the Hilbert space
\begin{equation}
L^2_{\ell=0}(\mathbb{R}^3)\;=\;U^{-1} L^2(\mathbb{R}^+\,\ud r)\otimes\mathrm{span}\Big\{\frac{1}{4\pi}\Big\}\,,
\end{equation}
where $U:L^2(\mathbb{R}^+,r^2\,\ud r)\xrightarrow[]{\cong}L^2(\mathbb{R}^+,\ud r)$, $(Uf)(r)=rf(r)$, one has
\begin{equation}
-\Delta_\alpha|_{L^2_{\ell=0}}\;=\;(U^{-1}h_{0,\alpha}\,U)\otimes\mathbbm{1}\,,
\end{equation}
and $h_{0,\alpha}$ is self-adjoint on $L^2(\mathbb{R}^+\,\ud r)$ with
\begin{equation}
\begin{split}
h_{0,\alpha}\;&=\;-\frac{\ud^2}{\ud r^2} \\
\mathcal{D}(h_{0,\alpha})\;&=\;\left\{g\in L^2(0,+\infty)\left|\!
\begin{array}{c}
g,g'\in AC_{\mathrm{loc}}((0,+\infty)) \\ g''\in L^2((0,+\infty)) \\
-4\pi\alpha \,g(0^+)+g'(0^+)=0
\end{array}\right.\!\!\right\}.
\end{split}
\end{equation}

From the above characterisation of $-\Delta_\alpha$ it is possible to deduce the spectral properties
\begin{equation}
\sigma_{\mathrm{ess}}(-\Delta_\alpha)\;=\;\sigma_{\mathrm{ac}}(-\Delta_\alpha)\;=\;[0,+\infty)\,,\qquad \sigma_{\mathrm{sc}}(-\Delta_\alpha)\;=\;\emptyset\,,
\end{equation}
and
\begin{equation}
\sigma_{\mathrm{p}}(-\Delta_\alpha)\;=\;
\begin{cases}
\qquad \emptyset & \textrm{if }\alpha\in[0,+\infty] \\
\{-(4\pi\alpha)^2\} & \textrm{if }\alpha\in(-\infty,0)\,.
\end{cases}
\end{equation}
The negative eigenvalue $-(4\pi\alpha)^2$, when it exists, is simple and the corresponding eigenfunction is $|x|^{-1}e^{4\pi\alpha|x|}$.

Last, we come to the quadratic form of the operator $-\Delta_\alpha$.
For each fixed $\lambda>0$ the form domain is the space
\begin{equation}\label{eq:form_dom}
\begin{split}
\mathcal{D}[-\Delta_\alpha]\;&=\;H^1(\mathbb{R}^3)\dotplus\mathrm{span}\{ G_\lambda\} \\
&=\;\Big\{g=f+c\,G_\lambda\,\Big|\,f\in H^1(\mathbb{R}^3)\,,\;c\in\mathbb{C}\Big\}
\end{split}
\end{equation}
and the quadratic form is given by
\begin{equation}\label{eq:form_action}
\begin{split}
(-\Delta_\alpha)[&f+c\,G_\lambda]+\lambda\|f+c\,G_\lambda\|_2^2 \;= \\
&=\;\|\nabla f\|_2^2+\lambda\|f\|_2^2+\Big(\alpha+\frac{\sqrt{\lambda}\,}{4\pi}\Big)\,|c|^2\,.
\end{split}
\end{equation}
Analogously to the operator domain, also for the functions in the form domain the highest local singularity is $|x|^{-1}$, since $G_\lambda\;\in\;H^{\frac{1}{2}-}(\mathbb{R}^3)$ and $f_\lambda\in H^1(\mathbb{R}^3)$. Instead, as typical when passing from the domain of a self-adjoint operator to its (larger) form domain, the characteristic boundary condition of $\mathcal{D}(-\Delta_\alpha)$ is lost in $\mathcal{D}[-\Delta_\alpha]$ and no constraint between regular and singular component remains (actually functions in $\mathcal{D}(-\Delta_\alpha)$ are not necessarily continuous).


\section{The fractional singular Laplacian $(-\Delta_\alpha)^{s/2}$: main results}\label{sec:main_results}

For $\alpha\geqslant 0$, the singular perturbed Laplacian $-\Delta_\alpha$ is a positive self-adjoint operator on $L^2(\mathbb{R}^3)$ and the spectral theorem provides an unambiguous definition of its fractional powers $(-\Delta_\alpha)^{s/2}$. Special cases are $s=0$, yielding the identity operator on $L^2(\mathbb{R}^3)$, and $s=2$, yielding the operator $-\Delta_\alpha$ itself, whereas $s=1$ (the square root) corresponds to an operator whose domain is the form domain of $-\Delta_\alpha$.

For general $s\in(0,2)$ we are able to provide the following amount of information.

Our first result concerns the `\emph{fractional domains}', namely the domains of the fractional powers of $-\Delta_\alpha$. We find that for small $s$ the fractional domain is the Sobolev space of order $s$, whereas when $s>\frac{1}{2}$ for each element of $\mathcal{D}((-\Delta_\alpha)^{s/2})$ we retrieve a notion of a \emph{regular part} in $H^s(\mathbb{R}^3)$ and a \emph{singular part} proportional to the Green's function $G_\lambda$, thus carrying a local $|x|^{-1}$ singularity. This is in complete analogy to what happens with the operator domain $\mathcal{D}(-\Delta_\alpha)$ and the form domain $\mathcal{D}[-\Delta_\alpha]$ -- see \eqref{eq:op_dom} and \eqref{eq:form_dom} above. In particular, when $s>\frac{3}{2}$ the singular part is also continuous, and its evaluation at $x=0$ provides the proportionality constant in front of the singular part, the very same kind of boundary condition displayed by the elements of $\mathcal{D}(-\Delta_\alpha)$.

\begin{theorem}\label{thm:domain_s}
Let $\alpha\geqslant 0$, $\lambda>0$, and $s\in(0,2)$. The following holds.
\begin{itemize}
 \item[(i)] If $s\in(0,\frac{1}{2})$, then
 \begin{equation}\label{eq:Ds_s0-1/2}
  \mathcal{D}((-\Delta_\alpha)^{s/2})\;=\;H^s(\mathbb{R}^3)\,.
 \end{equation}
 \item[(ii)] If $s\in(\frac{1}{2},\frac{3}{2})$, then
 \begin{equation}\label{eq:Ds_s1/2-3/2}
 \mathcal{D}((-\Delta_\alpha)^{s/2})\;=\;H^s(\mathbb{R}^3)\dotplus\mathrm{span}\{G_\lambda\}\,,
 \end{equation}
 where $G_\lambda$ is the function \eqref{eq:defGlambda}.
 \item[(iii)] If $s\in(\frac{3}{2},2)$, then
 \begin{equation}\label{eq:Ds_s3/2-2}
 \begin{split}
 \mathcal{D}&((-\Delta_\alpha)^{s/2})\;= \\
 &=\;\Big\{g\in L^2(\mathbb{R}^3)\,\Big|\,g=F_\lambda+\frac{F_\lambda(0)}{\,\alpha+\frac{\sqrt{\lambda}}{4\pi}\,}\,G_\lambda\textrm{ with }F_\lambda\in H^s(\mathbb{R}^3)\Big\}.
 \end{split}
 \end{equation}
\end{itemize}
\end{theorem}

Separating the three regimes above, two different transitions occur. When $s$ decreases from larger values, the first transition arises at $s=\frac{3}{2}$, namely the level of $H^s$-regularity at which continuity is lost. Correspondingly, the elements in $\mathcal{D}((-\Delta_\alpha)^{3/4})$ still decompose into a regular $H^{\frac{3}{2}}$-part plus a multiple of $G_\lambda$ (singular part), and the decomposition is still of the form $F_\lambda+c_{F_\lambda}G_\lambda$, except that now $F_\lambda$ cannot be arbitrary in $H^{\frac{3}{2}}(\mathbb{R}^3)$: indeed, $F_\lambda$ has additional properties, among which the fact that its Fourier transform is integrable (a fact that is false for generic $H^{\frac{3}{2}}$-functions), and for such $F_\lambda$'s the constant $c_{F_\lambda}$ has a form that is completely analogous to the constant in \eqref{eq:Ds_s3/2-2}, that is,
\[
c_{F_\lambda}\;=\;\frac{1}{\,\alpha+\frac{\sqrt{\lambda}}{4\pi}\,}\,\frac{1}{\,(2\pi)^{\frac{3}{2}}}\int_{\mathbb{R}^3}\!\ud p\,\widehat{F_\lambda}(p)
\]
(see \eqref{eq:Ds_for_s32_Prop} below).
Then, for $s<\frac{3}{2}$, the link between the two components disappears completely.

Decreasing $s$ further, the next transition occurs at $s=\frac{1}{2}$, namely the level of $H^s$-regularity below which the Green's function itself belongs to $H^s(\mathbb{R}^3)$ and it does not necessarily carry the leading singularity any longer. At the transition $s=\frac{1}{2}$, the elements in $\mathcal{D}((-\Delta_\alpha)^{1/4})$ still exhibit a decomposition into a regular $H^{\frac{1}{2}}$-part plus a more singular ${H^{\frac{1}{2}}}^-$-part, except that ${H^{\frac{1}{2}}}^-$-singularity is not explicitly expressed in terms of the Green's function $G_\lambda$. Then, for $s<\frac{1}{2}$, only $H^s$-functions form the fractional domain.

We shall discuss these transition points in Propositions \ref{prop:domain_s_12} and \ref{prop:domain_s_32}.

%
%

Our next result concerns the `singular' Sobolev norm induced by each fractional power $(-\Delta_\alpha)^{s/2}$ on its domain, in comparison with the corresponding ordinary Sobolev norm of the same order. Recall that  $(-\Delta_\alpha+\lambda\mathbbm{1})^{s/2}\geqslant \mathcal \lambda^{s/2}\mathbbm{1}$ and hence $g\mapsto\|(-\Delta_\alpha+\lambda\mathbbm{1})^{s/2}g\|_2$ defines a norm on $\mathcal{D}((-\Delta_\alpha)^{s/2})$, with respect to which the fractional domain is complete.

\begin{theorem}\label{thm:equiv_of_norms}
Let $\alpha\geqslant 0$, $\lambda>0$, and $s\in(0,2)$. Denote by $H^s_\alpha(\mathbb{R}^3)$, the `singular Sobolev space' of fractional order $s$, the Hilbert space $\mathcal{D}((-\Delta_\alpha)^{s/2})$ equipped with the `fractional singular Sobolev norm'
\begin{equation}
\|g\|_{H^s_\alpha}\;:=\;\|(-\Delta_\alpha+\lambda\mathbbm{1})^{s/2}g\|_2\,,\qquad g\in\mathcal{D}(-\Delta_\alpha)^{s/2}\,.
\end{equation}
The following holds.
\begin{itemize}
 \item[(i)] If $s\in(0,\frac{1}{2})$, then
 \begin{equation}\label{eq:equiv_of_norms_s012}
  \|g\|_{H^s_\alpha}\;\approx\;\|g\|_{H^s}\qquad \forall g\in\mathcal{D}(-\Delta_\alpha)^{s/2}=H^s(\mathbb{R}^3)
 \end{equation}
 in the sense of equivalence of norms. The constant in \eqref{eq:equiv_of_norms_s012} is bounded, and bounded away from zero, uniformly in $\alpha$.
 \item[(ii)] If $s\in(\frac{1}{2},\frac{3}{2})$ and $g=F+c\,G_\lambda$ is a generic element in  $H^s_\alpha(\mathbb{R}^3)$ according to the decomposition \eqref{eq:Ds_s1/2-3/2}, then
 \begin{equation}\label{eq:equiv_of_norms_s1232}
 \|F+c\,G_\lambda\|_{H^s_\alpha}\;\approx\;\|F\|_{H^s}+(1+\alpha)|c|\,.
 \end{equation}
 \item[(iii)] If $s\in(\frac{3}{2},2)$ and $g=F_\lambda+\frac{F_\lambda(0)}{\,\alpha+\frac{\sqrt{\lambda}}{4\pi}\,}\,G_\lambda$ is a generic element in  $H^s_\alpha(\mathbb{R}^3)$ according to the decomposition \eqref{eq:Ds_s3/2-2}, then
 \begin{equation}\label{eq:equiv_of_norms_s322}
 \big\|F_\lambda+{\textstyle\frac{F_\lambda(0)}{\,\alpha+\frac{\sqrt{\lambda}}{4\pi}\,}}\,G_\lambda\big\|_{H^s_\alpha}\;\approx\;\|F_\lambda\|_{H^s}\,.
 \end{equation}
 The constant in \eqref{eq:equiv_of_norms_s322} is bounded, and bounded away from zero, uniformly in $\alpha$.
\end{itemize}
\end{theorem}

It is worth remarking that in the limit $\alpha\to+\infty$ (recall that $\Delta_{\alpha=\infty}$ is the self-adjoint Laplacian on $L^2(\mathbb{R}^3)$) the equivalence of norms \eqref{eq:equiv_of_norms_s1232} tends to be lost, consistently with the fact that the function $G_\lambda$ does not belong to $H^s(\mathbb{R}^3)$. Instead, the norm equivalences \eqref{eq:equiv_of_norms_s012} and \eqref{eq:equiv_of_norms_s322} remain valid in the limit $\alpha\to+\infty$, which is also consistent with the structure of the space $H^s_\alpha(\mathbb{R}^3)$ in those two cases.

Last, we examine the action of $-\Delta_\alpha$ on generic functions of its domain and in particular, when applicable, on the function $G_\lambda$. We prove a computationally useful expression of $(-\Delta_\alpha+\lambda\mathbbm{1})^{s/2}\varphi$ in terms of the classical fractional derivative $(-\Delta+\lambda\mathbbm{1})^{s/2}\varphi$.

\begin{theorem}\label{thm:fract_pow_formulas}
Let $\alpha\geqslant 0$, $\lambda>0$, and $s\in(0,2)$.
\begin{itemize}
 \item[(i)] For each $\varphi\in L^2(\mathbb{R}^3)$ 
one has the distributional identity
\begin{equation}\label{eq:fract_pow_formula}
\begin{split}
 (-\Delta_\alpha+&\lambda\mathbbm{1})^{s/2}\varphi\;= \\
 &=\;(-\Delta+\lambda\mathbbm{1})^{s/2}\varphi-\,4\sin{\textstyle\frac{s \pi}{2}}\int_0^{+\infty}\!\ud t\; \frac{t^{s/2}\,\kappa_\varphi(t)}{\,4\pi\alpha+\sqrt{\lambda+t}\,}\,\frac{\,e^{-\sqrt{\lambda+t}\,|x|}}{4\pi|x|}\,,
\end{split}
\end{equation}
where
\begin{equation}\label{eq:kappa_phi}
\kappa_\varphi(t)\;:=\;\int_{\mathbb{R}^3}\ud y\,\frac{\,e^{-\sqrt{\lambda+t}\,|y|}}{4\pi|y|}\,\varphi(y)\,.
\end{equation}
 When $\varphi\in\mathcal{D}((-\Delta_\alpha)^{s/2})\cap H^s(\mathbb{R}^3)$ \eqref{eq:fract_pow_formula} is understood as an identity between $L^2$-functions, whereas when $\varphi\in\mathcal{D}((-\Delta_\alpha)^{s/2})\!\setminus\! H^s(\mathbb{R}^3)$ the r.h.s.~in the $L^2$-identity \eqref{eq:fract_pow_formula} is understood as the difference of two distributional contributions.
 \item[(ii)] The function $G_\lambda$ defined in \eqref{eq:defGlambda} belongs to $\mathcal{D}((-\Delta_\alpha)^{s/2})$ if and only if $s\in(0,\frac{3}{2})$, in which case
\begin{equation}\label{eq:gain_of_reg_DalphGlambd-THM}
(-\Delta_\alpha+\lambda\mathbbm{1})^{s/2}G_\lambda\in H^{\sigma-}(\mathbb{R}^3)\,,\quad \sigma:=\min\{{\textstyle\frac{3}{2}-s,\frac{1}{2}}\}\,,\quad s\in(0,{\textstyle\frac{3}{2}})\,.
\end{equation}
 Explicitly,
 \begin{equation}\label{eq:Dalph_on_Glam-THM}
(-\Delta_\alpha+\lambda\mathbbm{1})^{s/2}G_\lambda\;=\;J_\lambda\,,
\end{equation}
where $J_\lambda$ is the $L^2$-function given by
\begin{equation}\label{eq:defJ}
\widehat{J_\lambda}(p)\;:=\;\frac{\sin\frac{s\pi}{2}}{\pi(2\pi)^{\frac{3}{2}}}\int_0^{+\infty}\!\!\ud t\,\frac{t^{\frac{s}{2}-1}\,\phi(t)}{\,p^2+\lambda+t\,}\,,\qquad p\in\mathbb{R}^3\,,
\end{equation}
and
\begin{equation}\label{eq:defJ-phi}
\phi(t)\;:=\;\frac{4\pi\alpha+\sqrt{\lambda}}{\,4\pi\alpha+\sqrt{\lambda+t\,}\,}\,,\qquad t\geqslant 0\,.
\end{equation}
\end{itemize}
\end{theorem}

As mentioned already, the proofs of Theorems \ref{thm:domain_s}, \ref{thm:equiv_of_norms}, and \ref{thm:fract_pow_formulas} are deferred to Section \ref{sec:proofs}, after developing the preparatory material in Sections \ref{sec:canonical-decomposition}-\ref{sec:fractional_maps}; the only exception is the integral formula \eqref{eq:fract_pow_formula}, that for its technical relevance in our discussion will be proved in advance, at the end of Section \ref{sec:canonical-decomposition}.

\section{Canonical decomposition of the domain of $(-\Delta_\alpha)^{s/2}$}\label{sec:canonical-decomposition}

In this Section we present an intermediate technical lemma that is crucial for our analysis and gives a canonical decomposition of the domain of $(-\Delta_\alpha)^{s/2}$ for powers $s\in(0,2)$.

Based on the same argument, we then prove the integral formula \eqref{eq:fract_pow_formula} and hence part (i) of Theorem \ref{thm:fract_pow_formulas}.

\begin{proposition}\label{prop:canonical_decomp}
Fix $\alpha\geqslant 0$ and $\lambda>0$. Let $s\in(0,2)$ and $g\in\mathcal{D}((-\Delta_\alpha)^{s/2})$. Then
\begin{equation}\label{eq:gdecomp}
g\;=\;f_g+h_g
\end{equation}
where $f_g\in H^s(\mathbb{R}^3)$ is given by
\begin{equation}\label{eq:reg_comp_fg}
f_g\;:=\;(-\Delta+\lambda\mathbbm{1})^{-s/2}(-\Delta_\alpha+\lambda\mathbbm{1})^{s/2}g
\end{equation}
and $h_g\in L^2(\mathbb{R}^3)$ is given by
\begin{equation}\label{eq:sing_comp_hg}
\begin{split}
h_g(x)\;:&=\;4\sin\!\!\!\begin{array}{l}\frac{s\pi}{2}\end{array}\!\!\!\int_0^{+\infty}\!\ud t\; \frac{t^{-s/2}\,c_g(t)}{4\pi\alpha+\sqrt{\lambda+t}\,}\,\frac{e^{-\sqrt{\lambda+t}\,|x|}}{4\pi|x|}\,,
\end{split}
\end{equation}
having set
\begin{equation}\label{eq:cgrep}
\begin{split}
c_g(t)\;&:=\;\int_{\mathbb{R}^3}\ud y\;\frac{e^{-\sqrt{\lambda+t}\,|y|}}{4\pi|y|}\,((-\Delta_\alpha+\lambda\mathbbm{1})^{s/2}g)(y)\,.
\end{split}
\end{equation}
When $g$ runs in $\mathcal{D}((-\Delta_\alpha)^{s/2})$ then the corresponding component $f_g$ in the decomposition \eqref{eq:gdecomp} spans the whole $H^s(\mathbb{R}^3)$.
In terms of this decomposition,
\begin{equation}\label{eq:action_Dalpha}
(-\Delta_\alpha+\lambda\mathbbm{1})^{s/2}g\;=\;(-\Delta+\lambda\mathbbm{1})^{s/2}f_g\,.
\end{equation}
\end{proposition}

\begin{proof}
\eqref{eq:action_Dalpha} follows from \eqref{eq:reg_comp_fg}, so the proof consists of showing that \eqref{eq:reg_comp_fg} and \eqref{eq:sing_comp_hg} give \eqref{eq:gdecomp}. Our argument is based on the identity
\begin{equation}\label{eq:Ds-1}
\mathcal{D}((-\Delta_\alpha)^{s/2})\;=\;\mathcal{D}((-\Delta_\alpha+\lambda\mathbbm{1})^{s/2})\;=\;(-\Delta_\alpha+\lambda\mathbbm{1})^{-s/2} L^2(\mathbb{R}^3)\,,
\end{equation}
which follows from the spectral theorem, owing to $-\Delta_\alpha\geqslant\mathbb{O}$,  and on the integral identity
\begin{equation}\label{eq:x-integral_id}
x^{s/2}\;=\;\frac{\sin s\frac{\pi}{2}}{\pi}\int_0^{+\infty}\!\ud t\; t^{s/2-1}\,\frac{x}{t+x}\,,\qquad x\geqslant 0\,,\quad s\in(0,2)\,.
\end{equation}
By the functional calculus of $-\Delta_\alpha$,  \eqref{eq:x-integral_id} gives
\begin{equation*}
\begin{split}
(-\Delta_\alpha+&\lambda\mathbbm{1})^{-s/2}\; \\
& =\;\frac{\sin s\frac{\pi}{2}}{\pi}\int_0^{+\infty}\!\ud t\; t^{s/2-1}\,(-\Delta_\alpha+\lambda\mathbbm{1})^{-1}(t+(-\Delta_\alpha+\lambda\mathbbm{1})^{-1})^{-1} \\
& =\;\frac{\sin s\frac{\pi}{2}}{\pi}\int_0^{+\infty}\!\ud t\; t^{s/2-2}\,(-\Delta_\alpha+(\lambda+t^{-1})\mathbbm{1})^{-1}
\end{split}
\end{equation*}
and by means of the resolvent formula \eqref{eq:res_formula} and of \eqref{eq:x-integral_id} again one finds
\begin{equation}\label{eq:ref_f_1}
\begin{split}
(-\Delta_\alpha&+\lambda\mathbbm{1})^{-s/2}\;=\;\frac{\sin s\frac{\pi}{2}}{\pi}\int_0^{+\infty}\!\ud t\; t^{s/2-2}\,(-\Delta+(\lambda+t^{-1})\mathbbm{1})^{-1} \\
& \qquad+\frac{\sin s\frac{\pi}{2}}{\pi}\int_0^{+\infty}\!\ud t\; t^{s/2-2}\,\Big( \alpha+\frac{\sqrt{\lambda+t^{-1}}\,}{4\pi}\Big)^{\!-1}\,|G_{\lambda+t^{-1}}\rangle\langle\overline{G_{\lambda+t^{-1}}}| \\
&=\;(-\Delta+\lambda\mathbbm{1})^{-s/2} \;+\\
& \qquad +\frac{\sin s\frac{\pi}{2}}{\pi}\int_0^{+\infty}\!\ud t\; t^{s/2-2}\,\Big( \alpha+\frac{\sqrt{\lambda+t^{-1}}\,}{4\pi}\Big)^{\!-1}\,|G_{\lambda+t^{-1}}\rangle\langle\overline{G_{\lambda+t^{-1}}}|\,.
\end{split}
\end{equation}

Let now $g\in \mathcal{D}((-\Delta_\alpha)^{s/2})$: applying the operator identity \eqref{eq:ref_f_1} to the $L^2$-function $(-\Delta_\alpha+\lambda\mathbbm{1})^{s/2}g$ gives $g$ itself in the l.h.s.~and two summands on the r.h.s., the first of which is precisely $f_g$ defined in \eqref{eq:reg_comp_fg}, whereas the second is
\[
\begin{split}
&\frac{\sin s\frac{\pi}{2}}{\pi}\int_0^{+\infty}\!\ud t\; t^{s/2-2}\,\Big( \alpha+\frac{\sqrt{\lambda+t^{-1}}\,}{4\pi}\Big)^{\!-1}\,G_{\lambda+t^{-1}}\,\langle\,\overline{G_{\lambda+t^{-1}}}\,, (-\Delta_\alpha+\lambda\mathbbm{1})^{s/2}g\,\rangle \\
&=\;4\sin{\textstyle\frac{s \pi}{2}}\int_0^{+\infty}\!\ud t\; \frac{t^{-s/2}}{4\pi\alpha+\sqrt{\lambda+t}\,}\:G_{\lambda+t}\,\langle\,\overline{G_{\lambda+t}}\,, (-\Delta_\alpha+\lambda\mathbbm{1})^{s/2}g\,\rangle \;=\;h_g
\end{split}
\]
defined in \eqref{eq:sing_comp_hg}-\eqref{eq:cgrep}. This proves that $h_g=g-f_g\in L^2(\mathbb{R}^3)$ and yields \eqref{eq:gdecomp}.

Not only is $f_g\in H^s(\mathbb{R}^3)$ for given $g\in \mathcal{D}((-\Delta_\alpha)^{s/2})$, but also, conversely, given an arbitrary $f\in H^s(\mathbb{R}^3)$ the function $(-\Delta_\alpha+\lambda\mathbbm{1})^{-s/2}(-\Delta+\lambda\mathbbm{1})^{s/2}f$ clearly belongs to $\mathcal{D}((-\Delta_\alpha)^{s/2})$ and its component $f_g$ is precisely $f$. Thus, $f_g$ does span $H^s(\mathbb{R}^3)$ when $g$ runs in $\mathcal{D}((-\Delta_\alpha)^{s/2})$.
\end{proof}
%
%
%
%
%


\begin{proof}[Proof of Theorem \ref{thm:fract_pow_formulas}(i)]
 We follow the same line of reasoning that has led to Proposition \eqref{prop:canonical_decomp}. By \eqref{eq:x-integral_id} and the functional calculus of $-\Delta_\alpha$,
\begin{equation*}
(-\Delta_\alpha+\lambda\mathbbm{1})^{s/2}\varphi\;=\;\;\frac{\sin s\frac{\pi}{2}}{\pi}\int_0^{+\infty}\!\ud t\; t^{s/2-1}\,(-\Delta_\alpha+\lambda\mathbbm{1})(-\Delta_\alpha+(\lambda+t)\mathbbm{1})^{-1})\varphi\,.
\end{equation*}
Taking the difference between the identity above for generic $\alpha$ and for $\alpha=\infty$ (namely for the operator $-\Delta$ instead of $-\Delta_\alpha$), together with the resolvent formula \eqref{eq:res_formula}, yields
\begin{equation*}
\begin{split}
 &(-\Delta_\alpha+\lambda\mathbbm{1})^{s/2}\varphi-(-\Delta+\lambda\mathbbm{1})^{s/2}\varphi\;= \\
 &=\;-\frac{\sin s\frac{\pi}{2}}{\pi}\int_0^{+\infty}\!\ud t\; t^{s/2}\,\Big((-\Delta_\alpha+(\lambda+t)\mathbbm{1})^{-1})\varphi-(-\Delta+(\lambda+t)\mathbbm{1})^{-1})\varphi\Big) \\
 &=\;-\frac{\sin s\frac{\pi}{2}}{\pi}\int_0^{+\infty}\!\ud t\; t^{s/2}\,\Big(\alpha+\frac{\sqrt{\lambda+t}}{4\pi}\,\Big)^{\!-1}\,\frac{\,e^{-\sqrt{\lambda+t}\,|x|}}{4\pi|x|}\int_{\mathbb{R}^3}\ud y\,\frac{\,e^{-\sqrt{\lambda+t}\,|y|}}{4\pi|y|}\,\varphi(y)\,,
\end{split}
\end{equation*}
which leads to \eqref{eq:fract_pow_formula}, by means of the definition \eqref{eq:kappa_phi}.
\end{proof}

\section{Regularity properties 
}\label{sec:regularity}

In this Section we discuss 
the regularity and asymptotic properties of functions of the form $h_g$ that emerge in the the canonical decomposition of Proposition \ref{prop:canonical_decomp}.

For given $\alpha\geqslant 0$, $\lambda>0$, $s\in(0,2)$, and $f\in H^s(\mathbb{R}^3)$, we define
\begin{equation}\label{eq:cgeneral}
\begin{split}
c_f(t)\;&:=\;\int_{\mathbb{R}^3}\ud y\;\frac{e^{-\sqrt{\lambda+t}\,|y|}}{4\pi|y|}\,((-\Delta+\lambda\mathbbm{1})^{s/2}f)(y)
\end{split}
\end{equation}
and
\begin{equation}\label{eq:hgeneral}
h_f(x)\;:=\; 4\sin\!\!\!\begin{array}{l}\frac{s\pi}{2}\end{array}\!\!\!\int_0^{+\infty}\!\ud t\; \frac{t^{-s/2}\,c(t)}{4\pi\alpha+\sqrt{\lambda+t}\,}\,\frac{e^{-\sqrt{\lambda+t}\,|x|}}{4\pi|x|}\,.
\end{equation}
Equivalently, in Fourier transform,
\begin{equation}\label{eq:cgeneral-Fourier}
c_f(t)\;:=\;\frac{1}{\,(2\pi)^{3/2}}\int_{\mathbb{R}^3}\ud p\:\frac{\:(p^2+\lambda)^{s/2}}{p^2+\lambda+t\,}\:\widehat{f}(p)
\end{equation}
and
\begin{equation}\label{eq:hgeneral-Fourier}
\widehat{h_f}(p)\;:=\;\frac{4\sin\frac{s\pi}{2}}{(2\pi)^{3/2}}\int_0^{+\infty}\!\ud t\; \frac{t^{-s/2}\,c(t)}{4\pi\alpha+\sqrt{\lambda+t}\,}\,\frac{1}{p^2+\lambda+t\,}\,.
\end{equation}

It is also convenient to introduce the function $w_f$ whose Fourier transform is
\begin{equation}\label{eq:w}
\widehat{w_f}(p)\;:=\; -\frac{1}{\,p^2+\lambda\,}\,\frac{4\sin\frac{s\pi}{2}}{(2\pi)^{\frac{3}{2}}}\int_0^{+\infty}\!\ud t\,\frac{t^{1-\frac{s}{2}}\,c(t)}{4\pi\alpha+\sqrt{\lambda+t}\,}\,\frac{1}{\,p^2+\lambda+t\,}\,.
\end{equation}
Formally,
\begin{equation}
h_f\;=\;q_{f}\,G_\lambda+w_f\,
\end{equation}
where $G_\lambda$ is the function \eqref{eq:defGlambda} and
\begin{equation}\label{eq:qs}
q_f\;:=\;4\!\!\begin{array}{l}\sin\frac{s\pi}{2}\end{array}\!\!\!\int_0^{+\infty}\!\ud t\,\frac{t^{-\frac{s}{2}}\,c_f(t)}{4\pi\alpha+\sqrt{\lambda+t}\,}\,.
\end{equation}

\begin{lemma}\label{lem:ct_L2weighted}
For given $\alpha\geqslant 0$, $\lambda>0$, $s\in(0,2)$, and $f\in H^s(\mathbb{R}^3)$, the function $c(t)$ defined in \eqref{eq:cgeneral} is continuous in $t\in[0,+\infty)$ and satisfies the bounds
\begin{equation}\label{eq:ct_Linf}
|c_f(t)|\;\lesssim\;\|f\|_{H^s}(1+t)^{-\frac{1}{4}}
\end{equation}
and
\begin{equation}\label{eq:ct_L2weighted}
\int_0^{+\infty}\!\ud t\,t^{-\frac{1}{2}}|c_f(t)|^2\;\leqslant\; {\textstyle\frac{1}{2}}\,\|(p^2+\lambda)^{\frac{s}{2}}\widehat{f}\|_2^2\;\approx\;\|f\|^2_{H^s}\,.
\end{equation}
\end{lemma}

\begin{proof}
The continuity of $t\mapsto c_f(t)$ is immediately checked by re-writing \eqref{eq:cgeneral} as $c_f(t)=\langle G_{\lambda+t},(-\Delta+\lambda\mathbbm{1})^{s/2}f\rangle$. From
\begin{equation*}
\|G_{\lambda+t}\|_2\;=\;(8\pi\sqrt{\lambda+t})^{-\frac{1}{2}}\;\leqslant\;(8\pi\sqrt{\lambda})^{-\frac{1}{2}}\,,
\end{equation*}
a Schwarz inequality yields
\begin{equation*}
|c_f(t)|\;\leqslant\;\|G_{\lambda+t}\|_2\:\|(-\Delta+\lambda\mathbbm{1})^{s/2}f\|_2\;\lesssim\;\|f\|_{H^s}
\end{equation*}
and
\[
|c_f(t)|\;\lesssim\;t^{-1/4}\|f\|_{H^s}\,,
\]
whence \eqref{eq:ct_Linf}. Next, we consider the function
\[
\eta_\omega(\varrho)\;:=\;\varrho(\varrho^2+\lambda)^\frac{s}{2}\widehat{f}(\varrho,\omega)\,,\qquad\varrho\in\mathbb{R}^+\,,\,\omega\in \mathbb{S}^2\,,
\]
where we wrote $\widehat{f}(p)=\widehat{f}(\rho,\omega)$ in polar coordinates $p\equiv(\varrho,\omega)$, $\varrho:=|p|$, $\omega\in\mathbb{S}^2$. Clearly,
\[
\begin{split}
\int_{\mathbb{S}^2}\!\ud\omega\|\eta_\omega\|^2_{L^2(\mathbb{R}^+,\ud\varrho)}\;&=\;\int_{\mathbb{S}^2}\!\ud\omega\int_0^{+\infty}\!\ud\varrho\,\varrho^2|(\varrho^2+\lambda)^\frac{s}{2}\widehat{f}(\varrho,\omega)|^2 \\
&=\;\|(p^2+\lambda)^{\frac{s}{2}}\widehat{f}\|_2^2   \;\approx\;\|f\|^2_{H^s(\mathbb{R}^3)}\,,
\end{split}
\]
and we estimate
\[
\begin{split}
\int_0^{+\infty}\!\ud t\,t^{-\frac{1}{2}}|c_f(t)|^2\;&=\; 2\int_0^{+\infty}\!\ud t\,|c(t^2)|^2\;=\;\frac{1}{\,4\pi^{3}}\int_0^{+\infty}\!\ud t\,\Big|\!\int_{\mathbb{R}^3}\ud p\:\frac{\:(p^2+\lambda)^{\frac{s}{2}}\widehat{f}(p)}{p^2+\lambda+t^2}\,\Big|^2 \\
&=\;\frac{1}{\,4\pi^{3}}\int_0^{+\infty}\!\ud t\,\Big|\int_{\mathbb{S}^2}\!\ud\omega\int_0^{+\infty}\!\ud\varrho\,\frac{\varrho\,\eta_\omega(\varrho)}{\,\varrho^2+\lambda+t^2}\Big|^2 \\
&\leqslant\;\frac{1}{\,\pi^{2}}\int_0^{+\infty}\!\ud t\int_{\mathbb{S}^2}\!\ud\omega\,|(Q\eta_\omega)(t)|^2\;=\;\frac{1}{\,\pi^{2}}\int_{\mathbb{S}^2}\!\ud\omega\,\|Q\eta_\omega\|^2_{L^2(\mathbb{R}^+,\ud t)}\,,
\end{split}
\]
where $\eta\mapsto Q\eta$ is the integral operator on functions on $\mathbb{R}^+$ defined by
\[
(Q\eta)(t)\;:=\;\int_0^{+\infty}\!\!\!Q(t,\varrho)\,\eta(\varrho)\,\ud\varrho\,,\qquad Q(t,\varrho)\;:=\;\frac{\varrho}{\varrho^2+t^2}\,.
\]
We observe that $Q$ has precisely the form of the operator $Q_{\beta,\gamma,\delta}$ defined in \eqref{eq:defQkernelSchur}-\eqref{eq:defQopSchur} of Corollary \ref{cor:schur} with $\beta=\frac{1}{4}$, $\gamma=\delta=2$. Then the Schur bound \eqref{eq:Q_boundedness} yields
\[
\|Q\eta\|_2\;\leqslant\;\frac{\pi}{\sqrt{2}\,}\|\eta\|_2\qquad\forall\eta\in L^2(\mathbb{R}^+,\ud\varrho)\,.
\]
Therefore,
\[
\begin{split}
\int_0^{+\infty}\!\ud t\,t^{-\frac{1}{2}}|c_f(t)|^2\;&\leqslant\; \frac{1}{\,\pi^{2}}\int_{\mathbb{S}^2}\!\ud\omega\,\|Q\eta_\omega\|^2_{L^2(\mathbb{R}^+,\ud t)}  \leqslant\;\frac{1}{2}\int_{\mathbb{S}^2}\!\ud\omega\|\eta_\omega\|^2_{L^2(\mathbb{R}^+,\ud\varrho)} \\
&=\;{\textstyle\frac{1}{2}}\,\|(p^2+\lambda)^{\frac{s}{2}}\widehat{f}\|_2^2\;\approx\;\|f\|^2_{H^s}\,,
\end{split}
\]
which gives \eqref{eq:ct_L2weighted}.
%
\end{proof}

Let us now exploit the above information on the behaviour of $c_f(t)$
in order to obtain information about the regularity of the functions $h$ and $w$ defined, respectively, in \eqref{eq:hgeneral} and \eqref{eq:w}.
To this aim, we shall make often use of the identity (see \eqref{eq:useful_integral_repr})
\begin{equation}\label{eq:useful_int}
\int_0^{+\infty}\!\ud t\,\frac{\,t^{a-1}}{\:R+t\,}\;=\;\frac{\pi}{\sin a\pi}\,\frac{1}{\,R^{1-a}}\,,\qquad a\in(0,1)\,,\quad R> 0\,,
\end{equation}
whence also the useful limit
\begin{equation}\label{eq:useful_int-limit}
 \lim_{R\to+\infty}\:\Big(\frac{\pi}{\sin a\pi}\,\frac{1}{\,R^{1-a}}\Big)^{\!-1}\!\int_1^{+\infty}\!\ud t\,\frac{\,t^{a-1}}{\:R+t\,}\;=\;1\,,\qquad a\in(0,1)\,.
\end{equation}

We start with the function $h$ in the regime of small $s$.

\begin{proposition}\label{pro:h_in_Hs}
For given $\alpha\geqslant 0$, $\lambda>0$, $s\in(0,\frac{1}{2}]$, and $f\in H^s(\mathbb{R}^3)$, let $h_f$ be the function defined in \eqref{eq:cgeneral}-\eqref{eq:hgeneral}.
\begin{itemize}
 \item[(i)] If $s\in(0,\frac{1}{2})$, then $h_f\in H^s(\mathbb{R}^3)$ with
 \begin{equation}\label{eq:h_in_Hs}
 \qquad\|h_f\|_{H^s}\;\lesssim\;\|f\|_{H^s}\,,\qquad\qquad s\in(0,\textstyle\frac{1}{2})\,.
 \end{equation}
 \item[(ii)] If $s=\frac{1}{2}$, then
 $h_f\in H^{\frac{1}{2}^-}(\mathbb{R}^3)$, but in general $h_f\notin H^{1/2}(\mathbb{R}^3)$.
\end{itemize}
\end{proposition}

\begin{proof}
(i) Using \eqref{eq:hgeneral-Fourier} and setting $\mu_f(t):=t^{-\frac{1}{4}}c_f(t)$, we observe that
\[
\begin{split}
\|h_f&\|_{H^s}^2\;\approx\;\int_{\mathbb{R}^3}\ud p\,|(p^2+\lambda)^\frac{s}{2}\widehat{h_f}(p)|^2 \\
&
\approx\;\int_{\mathbb{R}^3}\ud p\,\Big|\int_0^{+\infty}\!\!\ud t\,\frac{t^{-\frac{s}{2}}\,c_f(t)}{\,4\pi\alpha+\sqrt{\lambda+t}\,}\,\frac{(p^2+\lambda)^{\frac{s}{2}}}{\,p^2+\lambda+t\,}\Big|^2 \\
&\lesssim\int_0^{+\infty}\!\!\ud \varrho\;\Big|\int_0^{+\infty}\!\!\ud t\,\frac{1}{\,t^{\frac{1}{4}+\frac{s}{2}}}\,\frac{\varrho(\varrho^2+\lambda)^{\frac{s}{2}}}{\varrho^2+\lambda+t}\;\mu_f(t)\Big|^2 \\
&\lesssim\int_0^{1}\ud \varrho\,\varrho^2(\varrho^2+\lambda)^{s}\Big|\!\int_0^{+\infty}\!\!\ud t\,\frac{\mu_f(t)}{\,t^{\frac{1}{4}+\frac{s}{2}}(\lambda+t)}\Big|^2 +\int_1^{+\infty}\!\!\ud \varrho\;\Big|\int_1^{+\infty}\!\!\ud t\,\frac{\varrho^{1+s}
}{\,t^{\frac{1}{4}+\frac{s}{2}}\,(\varrho^2+t)\,}\mu_f(t)\Big|^2 \\
&\lesssim\;\|\mu_f\|_{L^2(\mathbb{R}^+,\ud t)}^2+\|Q\mu_f\|^2_{L^2(\mathbb{R}^+,\ud\varrho)}\,,
\end{split}
\]
the last step following by a Schwarz inequalities and by setting
\[
(Q\mu_f)(\varrho)\;:=\;\int_0^{+\infty}\!\!\!\ud t\,Q(\varrho,t)\,\mu_f(t)\,,\qquad Q(\varrho,t)\;:=\;\frac{\varrho^{1+s}
}{\,t^{\frac{1}{4}+\frac{s}{2}}\,(\varrho^2+t)\,}\,.
\]
In fact, this defines an integral operator $Q$ on functions on $\mathbb{R}^+$ which has precisely the form of the operator $Q_{\beta,\gamma,\delta}$ defined in \eqref{eq:defQkernelSchur}-\eqref{eq:defQopSchur} of Corollary \ref{cor:schur} with $\beta=-\frac{1}{4}-\frac{s}{2}$, $\gamma=2$, $\delta=1$. Then the Schur bound \eqref{eq:Q_boundedness} yields
\[
\|Q\mu_f\|^2_{L^2(\mathbb{R}^+,\ud\varrho)}\;\leqslant\;\frac{\pi}{\sqrt{2}\,\cos(\frac{\pi}{4}+\frac{s\pi}{2})}\,\|\mu_f\|_{L^2(\mathbb{R}^+,\ud t)}^2\,.
\]
%
This, together with the bound \eqref{eq:ct_L2weighted}, gives
\[
\|h_f\|_{H^s}^2\;\lesssim\;\|\mu_f\|_{L^2(\mathbb{R}^+,\ud t)}^2+\|Q\mu_f\|^2_{L^2(\mathbb{R}^+,\ud\varrho)}\;\lesssim\;\|\mu_f\|^2_{L^2(\mathbb{R}^+,\ud t)}\;\lesssim\;\|f\|^2_{H^s(\mathbb{R}^3)}\,,
\]
which completes the proof of \eqref{eq:h_in_Hs} and of part (i).

(ii) When $s=\frac{1}{2}$, \eqref{eq:hgeneral-Fourier} reads
\[
\widehat{h_f}(p)\;=\;\frac{1}{\;\pi^{\frac{3}{2}}}\int_0^{+\infty}\!\!\ud t\,\frac{t^{-\frac{1}{4}}\,c_f(t)}{\,4\pi\alpha+\sqrt{\lambda+t}\,}\,\frac{1}{\,p^2+\lambda+t\,}\,.
\]
We consider the non-empty case of a non-zero $f\in H^{1/2}(\mathbb{R}^3)$ with positive Fourier transform and hence with non-zero $c_f(t)\geqslant 0$, due to \eqref{eq:cgeneral-Fourier}.
Owing to \eqref{eq:ct_Linf} and dominated convergence,
\[
\frac{1}{\;\pi^{\frac{3}{2}}}\int_0^{1}\!\ud t\,\frac{t^{-\frac{1}{4}}\,c_f(t)}{\,4\pi\alpha+\sqrt{\lambda+t}\,}\,\frac{1}{\,p^2+\lambda+t\,}\;\approx\;C_1\frac{1}{\,p^2+\lambda\,}\qquad\textrm{ as }|p|\to+\infty
\]
with constant
\[
C_1\;:=\;\int_0^{1}\!\ud t\,\frac{\pi^{-\frac{3}{2}}\,t^{-\frac{1}{4}}\,c_f(t)}{\,4\pi\alpha+\sqrt{\lambda+t}\,}\;\in\;(0,+\infty)\,,
\]
namely a contribution to $h_f$ that is a $H^{\frac{1}{2}^-}\!\!$--\,function not belonging to $H^{\frac{1}{2}}(\mathbb{R}^3)$. The remaining contribution to $h_f$ is given by the integration over $t\in[1,+\infty)$, and it is again a positive function of $p$, which therefore cannot compensate the singularity of the first contribution, i.e., it cannot make $h_f$ more regular than $H^{\frac{1}{2}^-}\!(\mathbb{R}^3)$.
%
%
%
\end{proof}


Next we show that for given $f\in H^s(\mathbb{R}^3)$ with $s\in(\frac{1}{2},2)$ the corresponding $h_f$ is a $H^{\frac{1}{2}^-}\!\!$--\,function given by the sum of the $H^{\frac{1}{2}^-}\!\!$--\,function $q_fG_\lambda$, that carries the leading singularity of $h_f$, and the more regular function $w_f\in H^{s}(\mathbb{R}^3)$. This is seen first discussing $q_f$ and then $w_f$.

For given $\alpha\geqslant 0$, $\lambda>0$, $s\in(\frac{1}{2},2)$, we introduce the function $\mathcal{G}_\lambda$ whose Fourier transform is given by
\begin{equation}\label{eq:GGlambda}
\widehat{\mathcal{G}_\lambda}(p)\;:=\;\frac{4\sin\frac{s\pi}{2}}{(2\pi)^{3/2}}\int_0^{+\infty}\!\!\!\ud t\,\frac{t^{-s/2}}{\,(4\pi\alpha+\sqrt{\lambda+t\,})(p^2+\lambda+t)}\,.
\end{equation}
$\mathcal{G}_\lambda$ has positive and bounded Fourier transform with asymptotics
\begin{equation}\label{eq:tG_asymptotics}
\widehat{\mathcal{G}_\lambda}(p)\;=\;\Big(\int_0^{+\infty}\!\!\!\ud t\,\frac{4\sin\frac{s\pi}{2}\;t^{-s/2}}{\,4\pi\alpha+\sqrt{\lambda+t\,}}\Big)\,\widehat{G_\lambda}(p)\,(1+o(1))\qquad\textrm{as }|p|\to +\infty\,,
\end{equation}
as follows immediately by dominated convergence.

\begin{lemma}\label{lem:qf_scalar_product}
For given $\alpha\geqslant 0$, $\lambda>0$, $s\in(\frac{1}{2},2)$, and $f\in H^s(\mathbb{R}^3)$, the corresponding constant $q_f$ defined in \eqref{eq:qs} satisfies
\begin{equation}\label{eq:qf_scalar_product}
 q_f\;=\;\langle\mathcal{G}_\lambda,(-\Delta+\lambda\mathbbm{1})^{s/2}f\rangle\,,
\end{equation}
where $\mathcal{G}_{\lambda}$ is the function \eqref{eq:GGlambda}. In particular,
\begin{equation}\label{eq:qf_fnormHs}
 |q_f|\;\lesssim\;\frac{1}{1+\alpha}\,\|f\|_{H^s}
 \end{equation}
 and
\begin{equation}\label{eq:vanishing_qf}
q_f=0\qquad\Leftrightarrow\qquad(-\Delta+\lambda\mathbbm{1})^{s/2}f\;\perp\;\mathcal{G}_\lambda
\end{equation}
in the sense of $L^2$-orthogonality.
\end{lemma}

\begin{proof}
Because of \eqref{eq:cgeneral-Fourier} and \eqref{eq:qs},
\[
\begin{split}
q_f\;&=\;\frac{4\sin\frac{s\pi}{2}}{(2\pi)^{3/2}}\int_0^{+\infty}\!\!\!\ud t\,\frac{t^{-s/2}}{\,4\pi\alpha+\sqrt{\lambda+t\,}}\int_{\mathbb{R}^3}\!\ud p\,\frac{\,(p^2+\lambda)^{\frac{s}{2}}\widehat{f}(p)\,}{\,p^2+\lambda+t\,} \\
&=\;\int_{\mathbb{R}^3}\!\ud p\;\widehat{\mathcal{G}_\lambda}(p)\,(p^2+\lambda)^{\frac{s}{2}}\widehat{f}(p)\,,
\end{split}
\]
whence \eqref{eq:qf_scalar_product}. A Schwarz inequality in \eqref{eq:qf_scalar_product}, together with the bound (see \eqref{eq:tG_asymptotics})
\[
|\widehat{\mathcal{G}_\lambda}(p)|\;\lesssim\;\Big(\int_0^{+\infty}\!\!\!\ud t\,\frac{4\sin\frac{s\pi}{2}\;t^{-s/2}}{\,4\pi\alpha+\sqrt{\lambda+t\,}}\Big)\,|\widehat{G_\lambda}(p)|\;\lesssim\;\frac{|\widehat{G_\lambda}(p)|}{1+\alpha}
\]
yields eventually \eqref{eq:qf_fnormHs}. 
\end{proof}

\begin{proposition}\label{lem:w_in_Hs}
For given $\alpha\geqslant 0$, $\lambda>0$, $s\in(\frac{1}{2},2)$, and $f\in H^s(\mathbb{R}^3)$, the functions $h_f$ and $w_f$ and the constant $q_f$ defined, respectively, in \eqref{eq:hgeneral}, \eqref{eq:w}, and \eqref{eq:qs}, satisfy the
identity
 \begin{equation}\label{eq:decomp_h=G+w}
  h_f\;=\;q_f\,G_\lambda+w_f\,,
 \end{equation}
 where $G_\lambda$ is the function \eqref{eq:defGlambda}.
 Moreover, $w_f$ belongs to $H^s(\mathbb{R}^3)$ and
 \begin{equation}\label{eq:Hsl_norm_w}
 \|(p^2+\lambda)^{\frac{s}{2}}\widehat{w_f}\|_2\;\leqslant\;\frac{\sqrt{2}\,\sin\frac{s\pi}{2}}{\:\sin(\frac{s\pi}{2}-\frac{\pi}{4})}\,\|(p^2+\lambda)^{\frac{s}{2}}\widehat{f}\|_2\,,
 \end{equation}
 whence also
 \begin{equation}\label{eq:bounds_qs_wp}
  \|w_f\|_{H^s}\;\lesssim\;\|f\|_{H^s}\,.
 \end{equation}
\end{proposition}

\begin{proof} The decomposition \eqref{eq:decomp_h=G+w} is an immediate consequence of the finiteness of $q_f$, namely of the bound \eqref{eq:qf_fnormHs}.
Using
\eqref{eq:w} and setting $\mu_f(t):=t^{-\frac{1}{4}}c_f(t)$, we observe that
\[
\begin{split}
\|(p^2+&\lambda)^{\frac{s}{2}}\widehat{w_f}\|_2^2\;= \\
&=\;\frac{\,2\sin^2\!\frac{s\pi}{2}}{\pi^3}\int_{\mathbb{R}^3}\ud p\,\Big|\int_0^{+\infty}\!\!\ud t\,\frac{t^{1-\frac{s}{2}}\,c_f(t)}{\,4\pi\alpha+\sqrt{\lambda+t}\,}\,\frac{1}{\,(p^2+\lambda)^{1-\frac{s}{2}}\,(p^2+\lambda+t)\,}\Big|^2 \\
&\leqslant\;\frac{\,8\sin^2\!\frac{s\pi}{2}}{\pi^2}\int_0^{+\infty}\!\!\ud \varrho\,\Big|\int_0^{+\infty}\!\!\ud t\,\frac{\varrho
\,t^{\frac{3}{4}-\frac{s}{2}}\,\mu_f(t)}{\,(\varrho^2+\lambda)^{1-\frac{s}{2}}\,(\varrho^2+\lambda+t)\,}\Big|^2 \\
&\leqslant\;\frac{\,8\sin^2\!\frac{s\pi}{2}}{\pi^2}\:\|Q\mu_f\|^2_{L^2(\mathbb{R}^+,\ud\varrho)}\,,
\end{split}
\]
where for convenience we wrote
\[
(Q\mu_f)(\varrho)\;:=\;\int_0^{+\infty}\!\!\!\ud t\,Q(\varrho,t)\,\mu_f(t)\,,\qquad Q(\varrho,t)\;:=\;\frac{\varrho^{s-1}\,t^{\frac{3}{4}-\frac{s}{2}}}{\,\varrho^2+t\,}\,.
\]
In fact this defines an integral operator $Q$ on functions on $\mathbb{R}^+$ which has precisely the form of the operator $Q_{\beta,\gamma,\delta}$ defined in \eqref{eq:defQkernelSchur}-\eqref{eq:defQopSchur} of Corollary \ref{cor:schur} with $\beta=\frac{3}{4}-\frac{s}{2}$, $\gamma=2$, $\delta=1$. Then the Schur bound \eqref{eq:Q_boundedness} yields
\[
\|Q\|_{L^2(\mathbb{R}^+,\ud t)\to L^2(\mathbb{R}^+,\ud\varrho)}\;\leqslant\;\frac{\pi}{\,\sqrt{2}\,\sin(\frac{s\pi}{2}-\frac{\pi}{4})}\,.
\]
Combining the estimates above with \eqref{eq:ct_L2weighted} then yields
\[
\begin{split}
\|(p^2+\lambda)^{\frac{s}{2}}\widehat{w_f}\|_2^2\;&\leqslant\;\frac{\,8\sin^2\!\frac{s\pi}{2}}{\pi^2}\:\|Q\mu_f\|^2_{L^2(\mathbb{R}^+,\ud\varrho)}\;\leqslant\;\frac{4\sin^2\!\frac{s\pi}{2}}{\:\sin^2(\frac{s\pi}{2}-\frac{\pi}{4})}\,\|\mu_f\|^2_{L^2(\mathbb{R}^+,\ud t)} \\
&\leqslant\;\frac{2\sin^2\!\frac{s\pi}{2}}{\:\sin^2(\frac{s\pi}{2}-\frac{\pi}{4})}\,\|(p^2+\lambda)^{\frac{s}{2}}\widehat{f}\|_2^2
\end{split}
\]
which is precisely \eqref{eq:Hsl_norm_w}.
\end{proof}

For the last noticeable property we want to discuss in this Section, as well as for later purposes, it is useful to highlight a few features, whose proof is elementary and will be omitted, of the function $t\mapsto\phi(t)$, $t\geqslant 0$, introduced in \eqref{eq:defJ-phi}.

\begin{lemma}\label{lem:phi}
For given $\alpha\geqslant 0$ and $\lambda>0$, \eqref{eq:defJ-phi} defines a function $\phi\in C^\infty([0,+\infty))$ with
\begin{equation}\label{eq:phi_double_form}
\begin{split}
\phi(t)\;&=\frac{4\pi\alpha+\sqrt{\lambda}}{\,4\pi\alpha+\sqrt{\lambda+t\,}\,} \;=\;1-\frac{t}{\,(4\pi\alpha+\sqrt{\lambda+ t\,}\,)(\sqrt{\lambda+t\,}+\sqrt{\lambda})}\,,
\end{split}
\end{equation}
\begin{equation}\label{eq:phi_positive_small1}
0\;<\;\phi(t)\;\leqslant\;\phi(0)\;=\;1\,,
\end{equation}
and
\begin{equation}\label{eq:bound_on_phi}
\phi(t)\;\lesssim\;(1+t)^{-1/2}\,.
\end{equation}
$\phi$ is strictly monotone decreasing and decays as $t\to +\infty$ with asymptotics
\begin{equation}\label{eq:asymptotics_phi}
\phi(t)\;=\;\frac{4\pi\alpha+\sqrt{\lambda}}{\sqrt{t\,}}-\frac{4\pi\alpha(4\pi\alpha+\sqrt{\lambda})}{t}+O(t^{-\frac{3}{2}})\qquad\textrm{as }t\to +\infty\,.
\end{equation}
\end{lemma}

We turn now to the discussion of a relevant connection between the constant $q_f$ defined in \eqref{eq:qs} and the function
\begin{equation}\label{eq:F}
\qquad F_f\;:=\;f+w_f\,.
\end{equation}
In fact, owing to Proposition \ref{lem:w_in_Hs}, when $f\in H^s(\mathbb{R}^3)$ so is $w_f$, and hence $F_f$ too. When $s>\frac{3}{2}$, a standard Sobolev lemma implies that $F_f$ is continuous. We shall now see that, in this regime of $s$, $F_f(0)$ is a multiple of $q_f$. 
Significantly, an analogous property survives when $s=\frac{3}{2}$ (see Proposition \ref{prop:high_s_in_Ds}(ii) in the next Section).

\begin{lemma}\label{lem:qf_controlledby_F0}
For given $\alpha\geqslant 0$, $\lambda>0$, $s\in(\frac{3}{2},2)$, and $f\in H^s(\mathbb{R}^3)$, let $w_f$ and $q_f$ be, respectively, the function and the constant defined in \eqref{eq:w} and \eqref{eq:qs}, and let $F_f$ be the function \eqref{eq:F}. 
Then $F_f$ is continuous and
\begin{equation}\label{eq:F0_qf}
F(0)\;=\;(\alpha+{\textstyle\frac{\sqrt{\lambda}}{4\pi}})\,q_f\,.
\end{equation}
\end{lemma}

\begin{remark}\label{rem:qf_s2}
It is worth noticing that \eqref{eq:F0_qf} is consistent also when $s\to 2$. Indeed, when $s=2$ and $f\in H^2(\mathbb{R}^3)$, then $w_f\equiv 0$, owing to \eqref{eq:w}, whence $F_f(0)=f(0)$. On the r.h.s.~of \eqref{eq:F0_qf}, we re-write $q_f$ given by \eqref{eq:qs} as
\begin{equation*}
q_f\;=\;\frac{\sin\frac{s\pi}{2}}{\pi}\int_0^{+\infty}\!\ud t\,\frac{t^{-\frac{s}{2}}\,c_f(t)}{\,\alpha+\frac{\sqrt{\lambda+t}}{4\pi}\,}\,.
\end{equation*}
As $s\to 2$ the pre-factor in front of the integral vanishes asymptotically as $(1-\frac{s}{2})$, whereas the integral diverges: indeed when $s=2$ we see from \eqref{eq:cgeneral-Fourier} that $c_f(t)\to f(0)$ as $t\to 0$, therefore when $s\to 2$ the leading (i.e., divergent) part of the integral is given by the integration around $t=0$, i.e.,
\[
\begin{split}
\int_0^{+\infty}\!\ud t\,\frac{t^{-\frac{s}{2}}\,c_f(t)}{\,\alpha+\frac{\sqrt{\lambda+t}}{4\pi}\,}\;&\approx\;(\alpha+{\textstyle\frac{\sqrt{\lambda}}{4\pi}})^{-1}f(0)\int_0^1  \ud t\,t^{-s/2} \\
&=\;(\alpha+{\textstyle\frac{\sqrt{\lambda}}{4\pi}})^{-1}{\textstyle(1-\frac{s}{2})^{-1}}f(0)\qquad\textrm{ as }s\to 2\,.
\end{split}
\]
Thus, $(\alpha+{\textstyle\frac{\sqrt{\lambda}}{4\pi}})\,q_f\to F_f(0)$ as $s\to 2$.
\end{remark}

\begin{proof}[Proof of Lemma \ref{lem:qf_controlledby_F0}]
We have already argued before stating the Lemma that $F_f$ is continuous.

Since $f\in H^s(\mathbb{R}^3)$ for $s>\frac{3}{2}$, then $\widehat{f}\in L^1(\mathbb{R}^3)$ and
\[
\begin{split}
f(0)\;&=\;\frac{1}{\:(2\pi)^{\frac{3}{2}}}\int_{\mathbb{R}^3}\!\ud p\,\widehat{f}(p)\;=\;\frac{1}{\:(2\pi)^{\frac{3}{2}}}\int_{\mathbb{R}^3}\!\ud p\,\widehat{f}(p)\;\frac{\sin\frac{s\pi}{2}}{\pi}\!\int_0^{+\infty}\!\!\!\ud t\,\frac{t^{-\frac{s}{2}}(p^2+\lambda)^{\frac{s}{2}}}{p^2+\lambda+t} \\
&=\;\frac{\sin\frac{s\pi}{2}}{\pi}\!\int_0^{+\infty}\!\!\!\ud t\,t^{-\frac{s}{2}}\,c_f(t)\,,
\end{split}
\]
having used \eqref{eq:useful_int} in the second identity and \eqref{eq:cgeneral-Fourier} in the third one.

Also $w_f\in H^s(\mathbb{R}^3)$ for $s>\frac{3}{2}$, owing to Proposition \ref{lem:w_in_Hs}, and hence $\widehat{w_f}\in L^1(\mathbb{R}^3)$; from this fact and from \eqref{eq:w} one obtains
\[
\begin{split}
w_f(0)\;&=\;\frac{1}{\:(2\pi)^{\frac{3}{2}}}\int_{\mathbb{R}^3}\!\ud p\:\widehat{w_f}(p) \\
&=\;-\frac{\,4\sin\frac{s\pi}{2}}{\:(2\pi)^3}\int_0^{+\infty}\!\!\!\ud t\,\frac{t^{1-\frac{s}{2}}\,c_f(t)}{\,4\pi\alpha+\sqrt{\lambda+t}\,}\int_{\mathbb{R}^3}\!\frac{\ud p}{\,(p^2+\lambda+t)(p^2+\lambda)\,} \\
&=\;-\frac{\,\sin\frac{s\pi}{2}}{\pi}\int_0^{+\infty}\!\!\!\ud t\,\frac{t^{1-\frac{s}{2}}\,c_f(t)}{\,(4\pi\alpha+\sqrt{\lambda+t}\,)(\sqrt{\lambda+t}+\sqrt{\lambda})\,} \\
&=\;-\frac{\sin\frac{s\pi}{2}}{\pi}\!\int_0^{+\infty}\!\!\!\ud t\,t^{-\frac{s}{2}}\,c_f(t)+\frac{\sin\frac{s\pi}{2}}{\pi}\!\int_0^{+\infty}\!\!\!\ud t\,t^{-\frac{s}{2}}\,c_f(t)\,\phi(t)\,,
\end{split}
\]
where we used \eqref{eq:phi_double_form} for  $\phi(t)$.

Combining the last two equations, and using \eqref{eq:phi_double_form} and \eqref{eq:qs}, one obtains
\[
\begin{split}
F_f(0)\;&=\;f(0)+w_f(0)\;=\;(4\pi\alpha+\sqrt{\lambda}\,)\,\frac{\sin{\textstyle \frac{s\pi}{2}}}{\pi}\int_0^{+\infty}\!\!\!\ud t\,\frac{t^{-\frac{s}{2}}\,c_f(t)}{\,4\pi\alpha+\sqrt{\lambda+t}\,} \\
&=\;(\alpha+{\textstyle\frac{\sqrt{\lambda}}{4\pi}})\,q_f\,,
\end{split}
\]
thus proving \eqref{eq:F0_qf}.
\end{proof}

\section{Subspaces of $\mathcal{D}((-\Delta_\alpha)^{s/2})$}\label{sec:subspaces_of_Ds}

In this Section we show that in the regime $s\in(0,\frac{3}{2})$ the domain of the fractional operator $(-\Delta_\alpha)^{s/2}$ contains two noticeable subspaces: the one-dimensional span of the Green function $G_\lambda$ defined in \eqref{eq:defGlambda} and the Sobolev space $H^s(\mathbb{R}^3)$. We also show that in the remaining regime $s\in[\frac{3}{2},2)$ none of these spaces is entirely contained in $\mathcal{D}((-\Delta_\alpha)^{s/2})$ -- however, there is a proper subspace of $H^s(\mathbb{R}^3)\dotplus\mathrm{span}\{G_\lambda\}$ which is instead part of $\mathcal{D}((-\Delta_\alpha)^{s/2})$.

As a consequence, recalling that $G_\lambda\in H^{\frac{1}{2}-}(\mathbb{R}^3)$, we will conclude that
\begin{equation}\label{eq:Ds_supset_Hs_Gl}
\mathcal{D}((-\Delta_\alpha)^{s/2})\;\supset\;H^s(\mathbb{R}^3)\dotplus\mathrm{span}\{G_\lambda\}\,,\qquad s\in[{\textstyle\frac{1}{2}},{\textstyle\frac{3}{2}})\,,
\end{equation}
and
\begin{equation}
\mathcal{D}((-\Delta_\alpha)^{s/2})\;\supset\;H^s(\mathbb{R}^3)\,,\qquad s\in(0,{\textstyle\frac{1}{2}})\,.
\end{equation}

%

The first two main results of this Section are formulated as follows.

\begin{proposition}\label{prop:eq:Dalph_on_Glam}
For given $\alpha\geqslant 0$, $\lambda>0$, and $s\in(0,2)$, one has
\begin{equation}\label{eq:Dalph_on_Glam}
(-\Delta_\alpha+\lambda\mathbbm{1})^{s/2}G_\lambda\;=\;J_\lambda
\end{equation}
in the distributional sense, where $G_\lambda$ is the function defined in \eqref{eq:defGlambda} and $J_\lambda$ is the function defined by \eqref{eq:defJ}-\eqref{eq:defJ-phi}. In particular,
\begin{equation}\label{eq:Dalph_on_Glam2}
G_\lambda\in\mathcal{D}((-\Delta_\alpha)^{s/2})\qquad\Leftrightarrow\qquad s\in(0,\textstyle\frac{3}{2})\,,
\end{equation}
in which case
\begin{equation}\label{eq:Gl_L2norm}
\|(-\Delta_\alpha+\lambda\mathbbm{1})^{s/2}G_\lambda\|_2\;\lesssim\;1+\alpha\,.
\end{equation}
\end{proposition}

\begin{proposition}\label{prop:Hs_in_domDalph}
For given $\alpha\geqslant 0$,
\begin{itemize}
 \item[(i)] if $s\in(0,\frac{3}{2})$, then $H^s(\mathbb{R}^3)$ is a subspace of $\mathcal{D}((-\Delta_\alpha)^{s/2})$ and for every $\lambda>0$ and $F\in H^s(\mathbb{R}^3)$ one has
\begin{equation}\label{eq:equiv_Hs_norm}
\|(-\Delta_\alpha+\lambda\mathbbm{1})^{s/2}F\|_{L^2}\;\lesssim\|F\|_{H^s}\,;
\end{equation}
\item[(ii)] if $s\in[\frac{3}{2},2)$, then $H^s(\mathbb{R}^3)$ is not a subspace of $\mathcal{D}((-\Delta_\alpha)^{s/2})$.
\end{itemize}
\end{proposition}

The third main result of this Section will be discussed later, see Proposition \ref{prop:high_s_in_Ds} below. In order to prove Proposition \ref{prop:eq:Dalph_on_Glam} we establish the following properties.

\begin{lemma}\label{lem:Jlambda}
For given $\alpha\geqslant 0$, $\lambda>0$, and $s\in(0,2)$, the function $J_\lambda$ defined by \eqref{eq:defJ}-\eqref{eq:defJ-phi} has real and bounded Fourier transform that satisfies
\begin{align}
\widehat{J_\lambda}(p)&\;=\;\frac{\kappa_s}{\:(p^2+\lambda)\,}\,(1+o(1))\,, &0<s<1\,, \label{eq:J-}\\
\widehat{J_\lambda}(p)&\;=\;\kappa_1\,\frac{\:\ln(p^2+\lambda+1)\:}{\:(p^2+\lambda)\,}\,(1+o(1))\,, &s=1\,, \label{eq:J1}\\
\widehat{J_\lambda}(p)&\;=\;\frac{\kappa_s}{\:(p^2+\lambda)^{\frac{3}{2}-\frac{s}{2}}}\,(1+o(1))\,, &1<s<2\,, \label{eq:J+}
\end{align}
as $|p|\to +\infty$, where $\kappa_s>0$ depends only on $s$ (as well as on $\alpha$ and $\lambda$). As a consequence, $J_\lambda$ belongs to $L^2(\mathbb{R}^3)$ if and only if $s\in(0,\frac{3}{2})$. When this is the case,
\begin{equation}\label{eq:Jl_L2norm}
\|J_\lambda\|_2\;\lesssim\;1+\alpha\,,
\end{equation}
and moreover
\begin{equation}\label{eq:gain_of_reg_DalphGlambd}
J_\lambda\in H^{\sigma-}(\mathbb{R}^3)\,,\qquad \sigma:=\min\{{\textstyle\frac{3}{2}-s,\frac{1}{2}}\}\,,\qquad s\in(0,{\textstyle\frac{3}{2}})\,.
\end{equation}
\end{lemma}

\begin{proof}
In the case $s\in(0,1)$, owing to \eqref{eq:phi_positive_small1}-\eqref{eq:bound_on_phi},
\[
\kappa_s\;:=\;\frac{\sin\frac{s\pi}{2}}{\pi(2\pi)^{\frac{3}{2}}}\int_0^{+\infty}\!\!\ud t\,t^{\frac{s}{2}-1}\,\phi(t)\;\lesssim\;\int_0^{+\infty}\!\!\ud t\,\frac{\,t^{\frac{s}{2}-1}\,}{(1+t)^{\frac{1}{2}}}\;<\;+\infty\,,
\]
whence
\[
(p^2+\lambda)\,\widehat{J_\lambda}(p)\;=\;\frac{\sin\frac{s\pi}{2}}{\pi(2\pi)^{\frac{3}{2}}}\int_0^{+\infty}\!\!\ud t\,t^{\frac{s}{2}-1}\,\phi(t)\,\frac{p^2+\lambda}{\,p^2+\lambda+t\,}\;\xrightarrow[]{\,|p|\to +\infty\,}\;\kappa_s
\]
by dominated convergence, which proves \eqref{eq:J-}.

In the case $s=1$,
\[
\widehat{J_\lambda}(p)\;=\;\frac{1}{\pi(2\pi)^{\frac{3}{2}}}\Big(\int_0^{1}\!\ud t\,\frac{t^{-\frac{1}{2}}\,\phi(t)}{\,p^2+\lambda+t\,}+\int_1^{+\infty}\!\!\ud t\,\frac{t^{-\frac{1}{2}}\,\phi(t)}{\,p^2+\lambda+t\,}\Big)\,.
\]
As $|p|\to +\infty$,
\[
 \int_0^{1}\!\ud t\,\frac{t^{-\frac{1}{2}}\,\phi(t)}{\,p^2+\lambda+t\,}\;\approx\;\frac{\mathrm{const.}}{\,p^2+\lambda\,}
\]
by \eqref{eq:phi_positive_small1} and dominated convergence, and
\[
\int_1^{+\infty}\!\!\ud t\,\frac{t^{-\frac{1}{2}}\,\phi(t)}{\,p^2+\lambda+t\,}\;\approx\;(4\pi\alpha+\sqrt{\lambda}\,)\int_1^{+\infty}\!\!\ud t\,\frac{t^{-1}\,}{\,p^2+\lambda+t\,}\;=\;(4\pi\alpha+\sqrt{\lambda}\,)\,\frac{\ln(p^2+\lambda+1)}{p^2+\lambda}
\]
by \eqref{eq:asymptotics_phi} and dominated convergence, which proves \eqref{eq:J1} with $\kappa_1:=\frac{4\pi\alpha+\sqrt{\lambda}}{\,\pi(2\pi)^{3/2}}$\,.

In the case $s\in(1,2)$,
\[
\widehat{J_\lambda}(p)\;=\;\frac{\sin\frac{s\pi}{2}}{\pi(2\pi)^{\frac{3}{2}}}\Big(\int_0^{1}\!\ud t\,\frac{t^{\frac{s}{2}-1}\,\phi(t)}{\,p^2+\lambda+t\,}+\int_1^{+\infty}\!\!\ud t\,\frac{t^{\frac{s}{2}-1}\,\phi(t)}{\,p^2+\lambda+t\,}\Big)\,.
\]
As $|p|\to +\infty$,
\[
 \int_0^{1}\!\ud t\,\frac{t^{-\frac{1}{2}}\,\phi(t)}{\,p^2+\lambda+t\,}\;\approx\;\frac{\mathrm{const.}}{\,p^2+\lambda\,}
\]
by \eqref{eq:phi_positive_small1} and dominated convergence, and
\[
\begin{split}
\frac{\sin\frac{s\pi}{2}}{\pi(2\pi)^{\frac{3}{2}}}\int_1^{+\infty}\!\!\ud t\,\frac{t^{\frac{s}{2}-1}\,\phi(t)}{\,p^2+\lambda+t\,}\;&\approx\;\frac{(4\pi\alpha+\sqrt{\lambda}\,)\sin\frac{s\pi}{2}}{\pi(2\pi)^{\frac{3}{2}}}\int_1^{+\infty}\!\!\ud t\,\frac{t^{\frac{s}{2}-\frac{3}{2}}\,}{\,p^2+\lambda+t\,} \\
&\approx\;\frac{(4\pi\alpha+\sqrt{\lambda}\,)\sin\frac{s\pi}{2}}{\pi(2\pi)^{\frac{3}{2}}}\int_0^{+\infty}\!\!\ud t\,\frac{t^{\frac{s}{2}-\frac{3}{2}}\,}{\,p^2+\lambda+t\,} \\
&=\;-\frac{(4\pi\alpha+\sqrt{\lambda}\,)\tan\frac{s\pi}{2}}{\,(2\pi)^{\frac{3}{2}}}\:\frac{1}{\,(p^2+\lambda)^{\frac{3}{2}-\frac{s}{2}}}
\end{split}
\]
by \eqref{eq:asymptotics_phi}, \eqref{eq:useful_int-limit}, and dominated convergence, which proves \eqref{eq:J+} with
\[
\kappa_s\;:=\;-(2\pi)^{-\frac{3}{2}}(4\pi\alpha+\sqrt{\lambda}\,)\tan\frac{s\pi}{2}>0\,.
\]

It is clear from the above arguments that in all cases $\widehat{J_\lambda}(p)$ is positive and uniformly bounded. Immediate consequences of the asymptotics \eqref{eq:J-}-\eqref{eq:J1}-\eqref{eq:J+} are the fact that $J_\lambda\in L^2(\mathbb{R}^3)$ if and only if $s\in(0,\frac{3}{2}$) and the gain of regularity \eqref{eq:gain_of_reg_DalphGlambd}. Then the point-wise bound
\begin{equation}
|\widehat{J_\lambda}(p)|\;\lesssim\;(1+\alpha) \:\frac{\sin\frac{s\pi}{2}}{\pi(2\pi)^{\frac{3}{2}}}\int_0^{+\infty}\!\!\ud t\,\frac{t^{\frac{s}{2}-1}}{\,(p^2+1+t)\sqrt{1+t}\,}
\end{equation}
yields immediately \eqref{eq:Jl_L2norm}.
\end{proof}

We can now prove Proposition \ref{prop:eq:Dalph_on_Glam}.

\begin{proof}[Proof of Proposition \ref{prop:eq:Dalph_on_Glam}]
By formula \eqref{eq:fract_pow_formula} of Theorem \ref{thm:fract_pow_formulas}(i), re-written in Fourier transform, we have
\[
((-\Delta_\alpha+\lambda\mathbbm{1})^{\frac{s}{2}}G_\lambda)\,\widehat{\;}\,(p)\;=\;((-\Delta+\lambda\mathbbm{1})^{\frac{s}{2}}G_\lambda)\,\widehat{\;}\,(p)\:+\:\widehat{\mathcal{I}_\lambda}(p)\,,
\]
where for convenience we set
\[
\widehat{\mathcal{I}_\lambda}(p)\;:=\;-\frac{\,4\sin{\textstyle\frac{s \pi}{2}}}{(2\pi)^{\frac{3}{2}}}\int_0^{+\infty}\!\ud t\; \frac{t^{\frac{s}{2}}\,\kappa_{G_\lambda}(t)}{\,4\pi\alpha+\sqrt{\lambda+t\,}\,}\,\frac{1}{\,p^2+\lambda+t\,}\,,
\]
and
$\kappa_{G_\lambda}$, given by \eqref{eq:kappa_phi}, is now computed as
\begin{equation*}
\kappa_{G_\lambda}(t)\;=\;\frac{1}{(2\pi)^3}\int_{\mathbb{R}^3}\ud p\,\frac{1}{(p^2+\lambda+t)(p^2+\lambda)}\;=\;\frac{1}{4\pi}\,\frac{1}{\sqrt{\lambda+t\,}+\sqrt{\lambda}}\,.
\end{equation*}
(Formula \eqref{eq:fract_pow_formula} is indeed usable here, because it has been already demonstrated, in the end of Section \ref{sec:canonical-decomposition}.) Thus,
\[
\begin{split}
\widehat{\mathcal{I}_\lambda}(p)\;&=\;-\frac{\,\sin{\textstyle\frac{s \pi}{2}}}{\pi(2\pi)^{\frac{3}{2}}}\int_0^{+\infty}\!\ud t\; \frac{t^{\frac{s}{2}-1}\,}{\,p^2+\lambda+t\,}+\frac{\sin\frac{s\pi}{2}}{\pi(2\pi)^{\frac{3}{2}}}\int_0^{+\infty}\!\!\ud t\,\frac{t^{\frac{s}{2}-1}\,\phi(t)}{\,p^2+\lambda+t\,}\,,
\end{split}
\]
where $\phi(t)$ is the function already introduced in \eqref{eq:defJ-phi} and \eqref{eq:phi_double_form}. Owing to \eqref{eq:useful_int},
\[
\frac{\,\sin{\textstyle\frac{s \pi}{2}}}{\pi(2\pi)^{\frac{3}{2}}}\int_0^{+\infty}\!\ud t\; \frac{t^{\frac{s}{2}-1}\,}{\,p^2+\lambda+t\,}\;=\;\frac{1}{\,(2\pi)^{\frac{3}{2}}}\,\frac{1}{\,(p^2+\lambda)^{1-\frac{s}{2}}\,}\;=\;((-\Delta+\lambda\mathbbm{1})^{\frac{s}{2}}G_\lambda)\,\widehat{\;}\,(p)\,,
\]
whereas, according to our definition \eqref{eq:defJ},
\[
\frac{\sin\frac{s\pi}{2}}{\pi(2\pi)^{\frac{3}{2}}}\int_0^{+\infty}\!\!\ud t\,\frac{t^{\frac{s}{2}-1}\,\phi(t)}{\,p^2+\lambda+t\,}\;=\;\widehat{J_\lambda}(p)\,.
\]
Therefore, $\widehat{\mathcal{I}_\lambda}(p)=-((-\Delta+\lambda\mathbbm{1})^{\frac{s}{2}}G_\lambda)\,\widehat{\;}\,(p)+\widehat{J_\lambda}(p)$, whence
\begin{equation*}
(-\Delta_\alpha+\lambda\mathbbm{1})^{s/2}G_\lambda\;=\;J_\lambda\,,
\end{equation*}
that is, the identity \eqref{eq:Dalph_on_Glam}.
As proved in Lemma \ref{lem:Jlambda}, $J_\lambda\in L^2(\mathbb{R}^3)$ $\Leftrightarrow$ $s\in(0,\frac{3}{2})$: thus,
$G_\lambda\in\mathcal{D}((-\Delta_\alpha)^{s/2})$ $\Leftrightarrow$ $s\in(0,\frac{3}{2})$, and \eqref{eq:Dalph_on_Glam2} follows.  \eqref{eq:Gl_L2norm} is then an immediate consequence of \eqref{eq:Jl_L2norm}.
\end{proof}

Let us now pass to the proof of Proposition \ref{prop:Hs_in_domDalph}.
First, we establish the following property.

\begin{lemma}
For given $\lambda>0$, $s\in(0,\frac{3}{2})$, and $F\in H^s(\mathbb{R}^3)$, let $\kappa_F(t)$ be the function defined in \eqref{eq:kappa_phi}, namely
\begin{equation}
\kappa_F(t)\;=\;\frac{1}{\,(2\pi)^\frac{3}{2}}\int_{\mathbb{R}^3}\!\ud p\,\frac{\widehat{F}(p)}{\,p^2+\lambda+t\,}\,.
\end{equation}
Then
\begin{equation}\label{eq:kF_L2weighted}
\int_0^{+\infty}\!\!\ud t\,\frac{|\kappa_F(t)|^2}{\,(t+\lambda)^{\frac{1}{2}-s}}\;\lesssim\;\|F\|_{H^s}^2\,.
\end{equation}
\end{lemma}

\begin{proof}
Passing to polar coordinates $p\equiv(\varrho,\omega)$, $\varrho:=|p|$, $\omega\in\mathbb{S}^2$, $\widehat{F}(p)=\widehat{F}(\rho,\omega)$, we see that the function $\eta_\omega(\varrho):=\varrho(\varrho^2+\lambda)^\frac{s}{2}\widehat{F}(\varrho,\omega)$ belongs to $L^2(\mathbb{R}^+,\ud\varrho)$ with
\[
\int_{\mathbb{S}^2}\!\ud\omega\|\eta_\omega\|^2_{L^2(\mathbb{R}^+,\ud\varrho)}\;=\;\int_{\mathbb{S}^2}\!\ud\omega\int_0^{+\infty}\!\ud\varrho\,\varrho^2|(\varrho^2+\lambda)^\frac{s}{2}\widehat{F}(\varrho,\omega)|^2\;\approx\;\|F\|^2_{H^s(\mathbb{R}^3)}\,.
\]
Moreover,
\[\tag{*}
 \int_0^{+\infty}\!\!\ud t\,\frac{|\kappa_F(t)|^2}{\,(t+\lambda)^{\frac{1}{2}-s}}\;\leqslant\;\int_0^{+\infty}\!\ud t\int_{\mathbb{S}^2}\!\ud\omega\:\Big|\int_0^{+\infty}\!\ud\varrho\,\frac{\,t^{\frac{1}{2}}\varrho^{1-s}\,\eta_\omega(\varrho)}{\,(t^2+\lambda)^{\frac{1}{4}-\frac{s}{2}}(\varrho^2+\lambda+t^2)\,}\Big|^2\,,
\]
because
\[
\begin{split}
\int_0^{+\infty}\!\!\ud t\,&\frac{|\kappa_F(t)|^2}{\,(t+\lambda)^{\frac{1}{2}-s}}\;=\;2\int_0^{+\infty}\!\!\ud t\,\frac{t\,|\kappa_F(t^2)|^2}{\,(t^2+\lambda)^{\frac{1}{2}-s}} \\
&=\;\frac{1}{\:4\pi^3}\int_0^{+\infty}\!\!\ud t\,\frac{t}{\,(t^2+\lambda)^{\frac{1}{2}-s}}\:\Big|\int_{\mathbb{R}^3}\ud p\:\frac{(p^2+\lambda)^{\frac{s}{2}} \widehat{F}(p)}{\,(p^2+\lambda+t^2)(p^2+\lambda)^{\frac{s}{2}}}\Big|^2 \\
&=\;\frac{1}{\,4\pi^{3}}\int_0^{+\infty}\!\ud t\:\Big|\int_{\mathbb{S}^2}\!\ud\omega\int_0^{+\infty}\!\ud\varrho\,\frac{\,t^{\frac{1}{2}}\varrho^2\,(\varrho^2+\lambda)^\frac{s}{2}\widehat{F}(\varrho,\omega)}{\,(t^2+\lambda)^{\frac{1}{4}-\frac{s}{2}}(\varrho^2+\lambda+t^2)\,(\varrho^2+\lambda)^\frac{s}{2}}\Big|^2 \\
&\leqslant\;\frac{1}{\,\pi^{2}}\int_0^{+\infty}\!\ud t\int_{\mathbb{S}^2}\!\ud\omega\:\Big|\int_0^{+\infty}\!\ud\varrho\,\frac{\,t^{\frac{1}{2}}\varrho^{1-s}\,\eta_\omega(\varrho)}{\,(t^2+\lambda)^{\frac{1}{4}-\frac{s}{2}}(\varrho^2+\lambda+t^2)\,}\Big|^2\,. 
\end{split}
\]
There are two possible cases: $s\in[0,\frac{1}{2})$ and $s\in[\frac{1}{2},\frac{3}{2})$. In the first case one has $\frac{1}{4}-\frac{s}{2}\in(0,\frac{1}{4}]$, and (*) yields
\[
 \begin{split}
  \int_0^{+\infty}\!\!\ud t\,\frac{|\kappa_F(t)|^2}{\,(t+\lambda)^{\frac{1}{2}-s}}\;&\leqslant\;\int_0^{+\infty}\!\ud t\int_{\mathbb{S}^2}\!\ud\omega\:\Big|\int_0^{+\infty}\!\ud\varrho\,\frac{\,t^s\varrho^{1-s}\,}{\,\varrho^2+\lambda+t^2\,}\,\eta_\omega(\varrho)\Big|^2 \\
  &\leqslant\;\int_0^{+\infty}\!\ud t\int_{\mathbb{S}^2}\!\ud\omega\,\big|(Q\eta_\omega)(t)\big|^2\;=\;\int_{\mathbb{S}^2}\!\ud\omega\,\|Q\eta_\omega\|^2_{L^2(\mathbb{R}^+,\ud t)}\,,
 \end{split}
\]
where $Q$ is the integral operator on functions on $\mathbb{R}^+$ defined by the kernel
\[
Q(\varrho,t)\;:=\;\frac{\,t^s\,\varrho^{1-s}}{\,\varrho^2+t^2\,}\,.
\]
In fact, $Q$ has precisely the form of the operator $Q_{\beta,\gamma,\delta}$ defined in \eqref{eq:defQkernelSchur}-\eqref{eq:defQopSchur} of Corollary \ref{cor:schur} with $\beta=\frac{1}{4}-\frac{s}{2}$, $\gamma=\delta=2$, where in this case $\beta\in(0,\frac{1}{4}]$ and hence it is admissible (the admissibility condition in Corollary \ref{cor:schur} is $\beta\in(-\frac{1}{2},\frac{1}{2})$): then the Schur bound \eqref{eq:Q_boundedness} yields
\[
\|Q\eta_\omega\|_{L^2(\mathbb{R}^+,\ud t)} \;\leqslant\;\frac{\pi}{\sqrt{2}\,\cos(\frac{\pi}{4}-\frac{s\pi}{2})}\,\|\eta_\omega\|_{L^2(\mathbb{R}^+,\ud\varrho)}\,.
\]
Therefore,
\[
\begin{split}
\int_0^{+\infty}\!\!\ud t\,\frac{|\kappa_F(t)|^2}{\,(t+\lambda)^{\frac{1}{2}-s}}\;&\leqslant\;\int_{\mathbb{S}^2}\!\ud\omega\,\|Q\eta_\omega\|^2_{L^2(\mathbb{R}^+,\ud t)} \\
&\lesssim\;\int_{\mathbb{S}^2}\!\ud\omega\|\eta_\omega\|^2_{L^2(\mathbb{R}^+,\ud\varrho)}\;\approx\;\|F\|^2_{H^s(\mathbb{R}^3)}\,,
\end{split}
\]
which proves \eqref{eq:kF_L2weighted} in the case $s\in[0,\frac{1}{2})$. In the second case, namely $s\in[\frac{1}{2},\frac{3}{2})$, one has $\frac{s}{2}-\frac{1}{4}\in[0,\frac{1}{2})$, and (*) yields
\[
 \begin{split}
  \int_0^{+\infty}\!\!\ud t&\,\frac{|\kappa_F(t)|^2}{\,(t+\lambda)^{\frac{1}{2}-s}}\;\leqslant\;\int_0^{+\infty}\!\ud t\int_{\mathbb{S}^2}\!\ud\omega\:\Big|\int_0^{+\infty}\!\ud\varrho\,\frac{\,t^{\frac{1}{2}}(t^2+\lambda)^{\frac{s}{2}-\frac{1}{4}}\varrho^{1-s}\,\eta_\omega(\varrho)}{\,\varrho^2+\lambda+t^2\,}\Big|^2 \\
  &\leqslant\;\int_0^{+\infty}\!\ud t\int_{\mathbb{S}^2}\!\ud\omega\:\Big|\int_0^{+\infty}\!\ud\varrho\,\frac{\,(t^2+\lambda)^{\frac{s}{2}}\varrho^{1-s}\,\eta_\omega(\varrho)}{\,\varrho^2+\lambda+t^2\,}\Big|^2 \\
  &\lesssim\;\int_0^{1}\ud t\,(t^2+\lambda)^s\int_{\mathbb{S}^2}\!\ud\omega\:\Big|\int_0^{+\infty}\!\ud\varrho\,\frac{\,\varrho^{1-s}\,}{\,\varrho^2+\lambda\,}\,\eta_\omega(\varrho)\Big|^2 \\
  &\qquad\qquad +\int_1^{+\infty}\!\ud t\int_{\mathbb{S}^2}\!\ud\omega\:\Big|\int_0^{+\infty}\!\ud\varrho\,\frac{\,t^s\varrho^{1-s}\,}{\,\varrho^2+t^2\,}\,\eta_\omega(\varrho)\Big|^2  \\
  &\lesssim\;
  \int_{\mathbb{S}^2}\!\ud\omega\,\|\eta_\omega\|^2_{L^2(\mathbb{R}^+,\ud\varrho)}+\int_{\mathbb{S}^2}\!\ud\omega\,\|Q\eta_\omega\|^2_{L^2(\mathbb{R}^+,\ud t)}\,,
 \end{split}
\]
the integral operator $Q$ being defined as in the first case. Here $Q$ is of the form $Q_{\beta,\gamma,\delta}$ of\eqref{eq:defQkernelSchur}-\eqref{eq:defQopSchur} with $\beta=\frac{1}{4}-\frac{s}{2}$, $\gamma=\delta=2$, where in this case $\beta\in(-\frac{1}{2},0]$ and hence $\beta$ is again admissible: the above inequality and the Schur bound \eqref{eq:Q_boundedness} then yield
\[
\begin{split}
\int_0^{+\infty}\!\!\ud t\,\frac{|\kappa_F(t)|^2}{\,(t+\lambda)^{\frac{1}{2}-s}}\;&\lesssim\;\int_{\mathbb{S}^2}\!\ud\omega\,\|\eta_\omega\|^2_{L^2(\mathbb{R}^+,\ud\varrho)}+\int_{\mathbb{S}^2}\!\ud\omega\,\|Q\eta_\omega\|^2_{L^2(\mathbb{R}^+,\ud t)} \\
&\lesssim\;\int_{\mathbb{S}^2}\!\ud\omega\|\eta_\omega\|^2_{L^2(\mathbb{R}^+,\ud\varrho)}\;\approx\;\|F\|^2_{H^s(\mathbb{R}^3)}\,,
\end{split}
\]
which proves \eqref{eq:kF_L2weighted} also in the case $s\in[\frac{1}{2},\frac{3}{2})$.
\end{proof}

We can now prove Proposition \ref{prop:Hs_in_domDalph}. To this aim, it is convenient to introduce the function $I_F$ whose Fourier transform is given by
\begin{equation}\label{eq:IF}
\widehat{I_F}(p)\;:=\;-\frac{\,4\sin{\textstyle\frac{s \pi}{2}}}{(2\pi)^{\frac{3}{2}}}\int_0^{+\infty}\!\ud t\; \frac{t^{\frac{s}{2}}\,\kappa_{F}(t)}{\,4\pi\alpha+\sqrt{\lambda+t\,}\,}\,\frac{1}{\,p^2+\lambda+t\,}\,,
\end{equation}
where
\begin{equation}\label{eq:IF_kF}
\kappa_F(t)\;:=\;\frac{1}{\,(2\pi)^\frac{3}{2}}\int_{\mathbb{R}^3}\!\ud p\,\frac{\widehat{F}(p)}{\,p^2+\lambda+t\,}\,.
\end{equation}

\begin{proof}[Proof of Proposition \ref{prop:Hs_in_domDalph}]
(i) By formulas \eqref{eq:fract_pow_formula}-\eqref{eq:kappa_phi} of Theorem \ref{thm:fract_pow_formulas}(i), re-written in Fourier transform, we have
\begin{equation}\label{eq:DasF}
((-\Delta_\alpha+\lambda\mathbbm{1})^{\frac{s}{2}}F)\,\widehat{\;}\,(p)\;=\;((-\Delta+\lambda\mathbbm{1})^{\frac{s}{2}}F)\,\widehat{\;}\,(p)\:+\:\widehat{I_F}(p)\,,
\end{equation}
where the function $I_F$ is given by \eqref{eq:IF}-\eqref{eq:IF_kF}.
By assumption, $(-\Delta+\lambda\mathbbm{1})^{\frac{s}{2}}F\in L^2(\mathbb{R}^3)$; therefore, the fact that $F\in\mathcal{D}((-\Delta_\alpha)^{s/2})$ with $\|(-\Delta_\alpha+\lambda\mathbbm{1})^{\frac{s}{2}}F\|_2\lesssim\|F\|_{H^s}$
follows at once from \eqref{eq:DasF} if one proves that $I_F\in L^2(\mathbb{R}^3)$ with $\|I_F\|_{L^2}\lesssim\|F\|_{H^s}$.
To this aim, setting $\mu(t):=(t+\lambda)^{-\frac{1}{4}+\frac{s}{2}}\kappa_F(t)$, we observe that
\[
\begin{split}
\|I_F\|_{2}^2\;&\lesssim\;\int_{\mathbb{R}^3}\ud p\,\Big|\int_0^{+\infty}\!\!\ud t\,\frac{t^{\frac{s}{2}}\,\kappa_F(t)}{\,4\pi\alpha+\sqrt{\lambda+t}\,}\,\frac{1}{\,p^2+\lambda+t\,}\Big|^2 \\
&
\lesssim\;\int_0^{+\infty}\!\!\ud \varrho\,\Big|\int_0^{+\infty}\!\!\ud t\;\frac{
\,t^{\frac{s}{2}}\,}{\,(\lambda+t)^{\frac{1}{4}+\frac{s}{2}}\,}\,\frac{\varrho}{\,\varrho^2+\lambda+t\,}\,\mu(t)\Big|^2\;\leqslant\;\|Q\mu\|^2_{L^2(\mathbb{R}^+,\ud\varrho)}\,,
\end{split}
\]
where for convenience we wrote
\[
(Q\mu)(\varrho)\;:=\;\int_0^{+\infty}\!\!\!\ud t\,Q(\varrho,t)\,\mu(t)\,,\qquad Q(\varrho,t)\;:=\;\frac{
\,t^{-\frac{1}{4}}\,\varrho\,}{\,\varrho^2+t\,}\,.
\]
In fact, this defines an integral operator $Q$ on functions on $\mathbb{R}^+$ which has precisely the form of the operator $Q_{\beta,\gamma,\delta}$ defined in \eqref{eq:defQkernelSchur}-\eqref{eq:defQopSchur} of Corollary \ref{cor:schur} with $\beta=-\frac{1}{4}$, $\gamma=2$, $\delta=1$. Then the Schur bound \eqref{eq:Q_boundedness} yields
\[
|Q\mu\|_{L^2(\mathbb{R}^+,\ud\varrho)}\;\leqslant\;\pi\,\|\mu\|_{L^2(\mathbb{R}^+,\ud t)}\,.
\]
%
%
Combining the estimates above with \eqref{eq:kF_L2weighted} yields
\[
\begin{split}
\|I_F\|_{2}\;\lesssim\;\|Q\mu\|_{L^2(\mathbb{R}^+,\ud\varrho)}\;\lesssim\;\|\mu\|_{L^2(\mathbb{R}^+,\ud t)}\;\lesssim\;\|F\|_{H^s}\,,
\end{split}
\]
which completes the proof of part (i).

As for part (ii), if for contradiction $H^s(\mathbb{R}^3)$ was a subspace of $\mathcal{D}((-\Delta_\alpha)^{s/2})$, then the canonical decomposition \eqref{eq:gdecomp}/\eqref{eq:decomp_h=G+w}  $g=f_g+c_{f_g} G_\lambda+w_{f_g}$ of a generic element $g\in\mathcal{D}((-\Delta_\alpha)^{s/2})$ for suitable functions $f_g,w_{f_g}\in H^s(\mathbb{R}^3)$ would imply that $c_{f_g}G_\lambda =g-{f_g}-w_{f_g}\in \mathcal{D}((-\Delta_\alpha)^{s/2})$. For those $g$'s with non-zero coefficient $c_g$ this would yield the contradiction that $G_\lambda$ too belongs to $\mathcal{D}((-\Delta_\alpha)^{s/2})$, which was proved to be false in Proposition \ref{prop:eq:Dalph_on_Glam}.
\end{proof}

We move now to the third main result of this Section. It is formulated for $s\in(\frac{1}{2},2)$, but it is relevant for us in the regime of large $s$, namely $s\in[\frac{3}{2},2)$ (it provides no new information for lower $s$). As seen previously, in this regime neither $H^s(\mathbb{R}^3)$ nor $\mathrm{span}\{G_\lambda\}$ are contained in $\mathcal{D}((-\Delta_\alpha)^{s/2})$. Nevertheless, we can identify a suitable proper subspace of $H^s(\mathbb{R}^3)\dotplus\mathrm{span}\{G_\lambda\}$ which is still contained in $\mathcal{D}((-\Delta_\alpha)^{s/2})$, as we shall now show.

To this aim,  given $\alpha\geqslant 0$, $\lambda>0$, and $s\in(\frac{1}{2},2)$, we introduce the subspace $\mathcal{D}_0^{(s)}\subset H^s(\mathbb{R}^3)$ defined by
\begin{equation}\label{eq:def_subspace_Ds}
 \mathcal{D}_0^{(s)}\;:=\;\left\{ F\in H^s(\mathbb{R}^3)\left|\!\!
 \begin{array}{c}
  F^{(0)}:=\displaystyle\frac{1}{\;(2\pi)^{\frac{3}{2}}}\int_{\mathbb{R}^3}\!\ud p\,\widehat{F}(p)<+\infty \\
  I_F+\frac{F^{(0)}}{\,\alpha+\frac{\sqrt{\lambda}}{4\pi}\,}\,J_\lambda\in L^2(\mathbb{R}^3)
 \end{array}
 \!\!\!\right.\right\},
\end{equation}
where $I_F$ is the function defined by \eqref{eq:IF}-\eqref{eq:IF_kF} for given $F$, and $J_\lambda$ is the function defined by \eqref{eq:defJ}-\eqref{eq:defJ-phi}.

\begin{proposition}\label{prop:high_s_in_Ds} Let $\alpha\geqslant 0$ and $\lambda>0$.
\begin{itemize}
 \item[(i)] For $s\in(\frac{1}{2},2)$ one has
  \begin{equation}\label{eq:Ds_contains_specialclass}
 \mathcal{D}((-\Delta_\alpha)^{s/2})\;\supset\;\Big\{F+\frac{F^{(0)}}{\,\alpha+\frac{\sqrt{\lambda}}{4\pi}\,}\,G_\lambda\,\Big|\,F\in \mathcal{D}_0^{(s)}\Big\}\,,
 \end{equation}
 the space $\mathcal{D}_0^{(s)}\subset H^s(\mathbb{R}^3)$ being defined in \eqref{eq:def_subspace_Ds}.  In particular, $\mathcal{D}_0^{(s)}$ contains the Schwarz class $\mathcal{S}(\mathbb{R}^3)$, and
  \begin{equation}\label{eq:Ds_contains_specialclass2}
 \mathcal{D}((-\Delta_\alpha)^{s/2})\;\supset\;\Big\{F+\frac{F(0)}{\,\alpha+\frac{\sqrt{\lambda}}{4\pi}\,}\,G_\lambda\,\Big|\,F\in\mathcal{S}(\mathbb{R}^3)\Big\}\,.
 \end{equation}
 \item[(ii)] For $s=\frac{3}{2}$ one has
  \begin{equation}\label{eq:Ds_for_s32}
 \mathcal{D}((-\Delta_\alpha)^{3/4})\;=\;\Big\{F+\frac{F^{(0)}}{\,\alpha+\frac{\sqrt{\lambda}}{4\pi}\,}\,G_\lambda\,\Big|\,F\in \mathcal{D}_0^{(3/2)}\Big\}\,.
 \end{equation}
\end{itemize}
\end{proposition}

\begin{remark}
 Formula \eqref{eq:Ds_for_s32} qualifies the fractional domain in the transition case $s=\frac{3}{2}$ and implies the following interesting corollary: the only linear combinations $F+q\,G_{\lambda}$ that it is possible to find in $\mathcal{D}((-\Delta_\alpha)^{3/4})$ for some $H^{\frac{3}{2}}$-function $F$ must satisfy $\int_{\mathbb{R}^3}\widehat{F}(p)\,\ud p<+\infty$; as such, $F$ cannot be a generic function in $H^{\frac{3}{2}}(\mathbb{R}^3)$. Such a \emph{loss of genericity} of the $H^{\frac{3}{2}}$-regular component in $\mathcal{D}((-\Delta_\alpha)^{3/4})$ is the distinctive feature of the transition at $s=\frac{3}{2}$, since both below and above this threshold the regular part of an element in the fractional domain $\mathcal{D}((-\Delta_\alpha)^{s/2})$ is indeed a generic $H^s$-function. In fact, we can prove this remarkable feature of the transition $s=\frac{3}{2}$ independently of the present proof of Proposition \ref{prop:high_s_in_Ds}: in order not to break the flow of our discussion, we find it instructive to cast an alternative proof in Appendix \ref{app:alternative_proofs}.
\end{remark}




\begin{proof}[Proof of Proposition \ref{prop:high_s_in_Ds}]
(i) Let $F\in\mathcal{D}_0^{(s)}$. In particular, $F\in H^s(\mathbb{R}^3)$ and $F^{(0)}$ is finite.

In order to prove \eqref{eq:Ds_contains_specialclass} one needs to show that $(-\Delta_\alpha+\lambda\mathbbm{1})^{s/2}(F+\frac{F^{(0)}}{\,\alpha+\frac{\sqrt{\lambda}}{4\pi}\,}\,G_\lambda)$ is square integrable. In fact, owing to \eqref{eq:Dalph_on_Glam} and \eqref{eq:DasF},
\begin{equation}\label{eq:ation_on_F+F0Gl}
(-\Delta_\alpha+\lambda\mathbbm{1})^{s/2}\Big(F+\frac{F^{(0)}}{\,\alpha+\frac{\sqrt{\lambda}}{4\pi}\,}\,G_\lambda\Big)\;=\;(-\Delta+\lambda\mathbbm{1})^{s/2}F+ I_F+\frac{F^{(0)}}{\,\alpha+\frac{\sqrt{\lambda}}{4\pi}\,}\,J_\lambda\,,
\end{equation}
which indeed belongs to $L^2(\mathbb{R}^3)$ because so do $(-\Delta+\lambda\mathbbm{1})^{s/2}F$ and $I_F+\frac{F^{(0)}}{\,\alpha+\frac{\sqrt{\lambda}}{4\pi}\,}\,J_\lambda$, as a consequence of the fact that $F$ belongs to the space $\mathcal{D}_0^{(s)}$.

Next, in order to prove \eqref{eq:Ds_contains_specialclass2} we combine \eqref{eq:defJ}-\eqref{eq:defJ-phi} and \eqref{eq:IF}-\eqref{eq:IF_kF} so as to get
\[
\begin{split}
\widehat{I_F}(p)+\frac{F^{(0)}}{\,\alpha+\frac{\sqrt{\lambda}}{4\pi}\,}\,\widehat{J_\lambda}(p)\;=\;\frac{4\sin\frac{s\pi}{2}}{(2\pi)^{\frac{3}{2}}}\int_0^{+\infty}\!\!\!\ud t\,\frac{\,t^{\frac{s}{2}-1}(F^{(0)}-t\,\kappa_F(t))\,}{\,4\pi\alpha+\sqrt{\lambda+t}\,}\,\frac{1}{\,p^2+\lambda+t\,}\,.
\end{split}
\]
When $F\in\mathcal{S}(\mathbb{R}^3)$ the finiteness of $F^{(0)}=F(0)$ is obvious, and
\[
\begin{split}
 |F^{(0)}-t\,\kappa_F(t)|\;&=\;\Big|\frac{1}{\,(2\pi)^{\frac{3}{2}}}\int_{\mathbb{R}^3}\ud p\,\widehat{F}(p)\,\frac{p^2+\lambda}{\,p^2+\lambda+t\,}\Big|\;\lesssim\;\frac{\;\mathrm{const}(F)\;}{1+t}\,,
\end{split}
\]
whence
\[
\Big|\,\widehat{I_F}(p)+\frac{F^{(0)}}{\,\alpha+\frac{\sqrt{\lambda}}{4\pi}\,}\,\widehat{J_\lambda}(p)\,\Big|\;\lesssim\;\frac{\;\mathrm{const}(F)\;}{p^2+\lambda}\,.
\]
This shows that $I_F+\frac{F^{(0)}}{\,\alpha+\frac{\sqrt{\lambda}}{4\pi}\,}\,J_\lambda\in L^2(\mathbb{R}^3)$ whenever $F\in\mathcal{S}(\mathbb{R}^3)$, thus concluding that $\mathcal{S}(\mathbb{R}^3)\subset\mathcal{D}_0^{(s)}$.

\noindent (ii) One has to prove the opposite inclusion than \eqref{eq:Ds_contains_specialclass} in the special case $s=\frac{3}{2}$. Let $g\in\mathcal{D}((-\Delta_\alpha)^{3/4})$. Necessarily $g=F_{f_g}+q_{f_g}\, G_\lambda$ for functions $f_g,w_{f_g}, F_{f_g}\in  H^{\frac{3}{2}}(\mathbb{R}^3)$ with $F_{f_g}=f_g+w_{f_g}$ and for a constant $q_{f_g}\in \mathbb{C}$, as prescribed by the canonical decomposition \eqref{eq:gdecomp}/\eqref{eq:decomp_h=G+w}. Let us suppress the index `$g$' in the following.

%
%
Now, we claim that
\[\tag{i}
 F_f^{(0)}\;=\;\displaystyle\frac{1}{\;(2\pi)^{3/2}}\int_{\mathbb{R}^3}\!\ud p\,\widehat{F_f}(p)<+\infty\qquad\textrm{and}\qquad q_f\;=\;\frac{F_f^{(0)}}{\,\alpha+\frac{\sqrt{\lambda}}{4\pi}\,}\,.
\]
From this claim we deduce that $F_f+\frac{F_f^{(0)}}{\,\alpha+\frac{\sqrt{\lambda}}{4\pi}\,}\,G_\lambda=g\in\mathcal{D}((-\Delta_\alpha)^{3/4})$; as a consequence, \eqref{eq:ation_on_F+F0Gl} implies that $I_{F_f}+\frac{F_f^{(0)}}{\,\alpha+\frac{\sqrt{\lambda}}{4\pi}\,}\,J_\lambda\in L^2(\mathbb{R}^3)$. This completes the proof, because the finiteness of $F_f^{(0)}$ and the square-integrability of $I_{F_f}+\frac{F_f^{(0)}}{\,\alpha+\frac{\sqrt{\lambda}}{4\pi}\,}\,J_\lambda$ amount to $F_f\in\mathcal{D}^{(3/2)}_0$, and $g$ has the form $F_f+\frac{F_f^{(0)}}{\,\alpha+\frac{\sqrt{\lambda}}{4\pi}\,}\,G_\lambda$.

Let us therefore establish (i). To this aim, we mimic the proof of Lemma \ref{lem:qf_controlledby_F0}: in that case we had $s>\frac{3}{2}$, which made the manipulation of all the indefinite integrals harmless; now, instead, $s=\frac{3}{2}$ and a truncation scheme is needed.
Moreover, thanks to the linearity, let us assume, non restrictively, that $\widehat{f}(p)>0$, and hence also $c_f(t)>0$ and $-\widehat{w_f}(p)>0$, as follows from \eqref{eq:cgeneral-Fourier} and \eqref{eq:w}.

First of all,
\[\tag{ii}
 \begin{split}
  F_f^{(0)}\;&=\;\lim_{R\to +\infty}\int_{|p|<R}\widehat{F_f}(p)\,\ud p \\
  &=\;\lim_{R\to +\infty}\Big(\frac{1}{\,(2\pi)^{\frac{3}{2}}}\int_{|p|<R}\widehat{f}(p)\,\ud p+\frac{1}{\,(2\pi)^{\frac{3}{2}}}\int_{|p|<R}\widehat{w_f}(p)\,\ud p\Big)\,.
 \end{split}
\]
In general, each integral in the r.h.s.~above is in divergent as $R\to +\infty$, and we want to show that a compensation among them cancels this possible divergence.

By inserting into the first integrand in the r.h.s.~of (ii) the quantity
\[
 1\;=\;\frac{1}{\pi\sqrt{2}}\int_0^{+\infty}\!\!\ud t\,\frac{\,t^{-\frac{3}{4}}(p^2+\lambda)^{\frac{3}{4}}}{p^2+\lambda+t}
\]
(see \eqref{eq:useful_int}), it is immediately checked that dominated convergence and exchange of the truncated integration over $t$ and $p$ apply, so one has
\[\tag{iii}
 \begin{split}
  \frac{1}{\,(2\pi)^{\frac{3}{2}}}&\int_{|p|<R}\widehat{f}(p)\,\ud p\;= \\
  &=\;\frac{1}{\,\pi\sqrt{2}\,(2\pi)^{\frac{3}{2}}}\lim_{T\to +\infty}\int_{|p|<R}\ud p \,\widehat{f}(p)\,\int_0^T\ud t\, \frac{\,t^{-\frac{3}{4}}(p^2+\lambda)^{\frac{3}{4}}}{p^2+\lambda+t} \\
  &=\;\frac{1}{\,\pi\sqrt{2}\,}\lim_{T\to +\infty}\int_0^T\ud t\,t^{-\frac{3}{4}}\frac{1}{\,(2\pi)^{\frac{3}{2}}}\int_{|p|<R}\ud p \,\frac{\,(p^2+\lambda)^{\frac{3}{4}}\widehat{f}(p)}{p^2+\lambda+t} \\
   &=\;\frac{1}{\,\pi\sqrt{2}\,}\lim_{T\to +\infty}\int_0^T\ud t\,t^{-\frac{3}{4}}\,c_{R,f}(t)\,,
 \end{split}
\]
where for convenience we denoted by
\[
 c_{R,f}(t)\;:=\;\frac{1}{\,(2\pi)^{\frac{3}{2}}}\int_{|p|<R}\ud p \,\frac{\,(p^2+\lambda)^{\frac{3}{4}}\widehat{f}(p)}{p^2+\lambda+t}
\]
the finite-momentum truncation of the function $c_f(t)$ defined in \eqref{eq:cgeneral-Fourier}.

An analogous use of dominated convergence and exchange of integration, using \eqref{eq:w}, yields
\[
 \begin{split}
  \frac{1}{\,(2\pi)^{\frac{3}{2}}}&\int_{|p|<R}\widehat{w_f}(p)\,\ud p\;= \\
  &=\;-\frac{2\sqrt{2}}{\,(2\pi)^3}\int_{|p|<R}\ud p \,\int_0^{+\infty}\!\!\ud t\,\frac{\,t^{\frac{1}{4}}c_f(t)}{\,4\pi\alpha+\sqrt{\lambda+t}\,}\,\frac{1}{\,(p^2+\lambda)(p^2+\lambda+t)\,} \\
  &=\;-\frac{2\sqrt{2}}{\,(2\pi)^3}\lim_{T\to +\infty}\int_{|p|<R}\ud p \,\int_0^{T}\ud t\,\frac{\,t^{\frac{1}{4}}c_f(t)}{\,4\pi\alpha+\sqrt{\lambda+t}\,}\,\frac{1}{\,(p^2+\lambda)(p^2+\lambda+t)\,} \\
  &=\;-\frac{2\sqrt{2}}{\,(2\pi)^3}\lim_{T\to +\infty}\int_0^{T}\ud t\,\frac{\,t^{\frac{1}{4}}c_f(t)}{\,4\pi\alpha+\sqrt{\lambda+t}\,}\int_{|p|<R}\frac{\ud p}{\,(p^2+\lambda)(p^2+\lambda+t)\,}\,.
 \end{split}
\]
It is convenient to re-arrange the r.h.s.~above as
\[\tag{iv}
 \begin{split}
  &\frac{1}{\,(2\pi)^{\frac{3}{2}}}\int_{|p|<R}\widehat{w_f}(p)\,\ud p\;= \\
  &\;\;=\;\frac{1}{\pi\sqrt{2}\,}\lim_{T\to +\infty}\int_0^{T}\ud t\,t^{-\frac{3}{4}}\,c_f(t)\,\phi(t)\big({\textstyle \frac{2}{\pi}\arctan\frac{R}{\sqrt{\lambda+t}}}\big) \\
  &\;\;\qquad-\frac{1}{\pi\sqrt{2}\,}\lim_{T\to +\infty}\int_0^{T}\ud t\,t^{-\frac{3}{4}}\,c_f(t)\big({\textstyle \frac{2}{\pi}\arctan\frac{R}{\sqrt{\lambda+t}}}\big) \\
  &\;\;\qquad-\frac{2\sqrt{2}}{\,(2\pi)^3}\lim_{T\to +\infty}\int_0^{T}\ud t\,\frac{\,t^{\frac{1}{4}}c_f(t)}{\,4\pi\alpha+\sqrt{\lambda+t}\,}\;\times \\
  &\;\;\qquad\qquad\times\textstyle\Big(\int_{|p|<R}\frac{\ud p}{\,(p^2+\lambda)(p^2+\lambda+t)\,}-\frac{4\pi}{\sqrt{\lambda+t}+\sqrt{\lambda}}\,\arctan\frac{R}{\sqrt{\lambda+t}}\Big)\,,
 \end{split}
\]
where we inserted the function $\phi(t)$ defined in \eqref{eq:defJ-phi}/\eqref{eq:phi_double_form}.

Plugging (iii) and (iv) into (ii),
\[\tag{v}
 \begin{split}
  F_f^{(0)}\;&=\;\lim_{R\to +\infty}\Big\{\frac{1}{\pi\sqrt{2}\,}\lim_{T\to +\infty}\int_0^{T}\ud t\,t^{-\frac{3}{4}}\,c_f(t)\,\phi(t)\big({\textstyle \frac{2}{\pi}\arctan\frac{R}{\sqrt{\lambda+t}}}\big) \\
  &\;\;\;\;\qquad\qquad +\frac{1}{\pi\sqrt{2}\,}\lim_{T\to +\infty}\int_0^{T}\ud t\,t^{-\frac{3}{4}}\,\Big(c_{R,f}(t)-c_f(t)\big({\textstyle \frac{2}{\pi}\arctan\frac{R}{\sqrt{\lambda+t}}}\big)\Big)\\
  &\;\;\;\;\qquad\qquad-\frac{2\sqrt{2}}{\,(2\pi)^3}\lim_{T\to +\infty}\int_0^{T}\ud t\,\frac{\,t^{\frac{1}{4}}c_f(t)}{\,4\pi\alpha+\sqrt{\lambda+t}\,}\;\times \\
  &\;\;\;\;\;\;\qquad\qquad\qquad\times\textstyle\Big(\int_{|p|<R}\frac{\ud p}{\,(p^2+\lambda)(p^2+\lambda+t)\,}-\frac{4\pi}{\sqrt{\lambda+t}+\sqrt{\lambda}}\,\arctan\frac{R}{\sqrt{\lambda+t}}\Big)\Big\}\,.
 \end{split}
\]

The first term in the r.h.s.~of (v) can be thought of as an integration over $t\in\mathbb{R}$ of the function $t\mapsto\mathbf{1}_{\{t\in[0,T]\}}t^{-\frac{3}{4}}\,c_f(t)\,\phi(t)({\textstyle \frac{2}{\pi}\arctan\frac{R}{\sqrt{\lambda+t}}})$. Recalling that $c_f(t)\lesssim(1+t)^{-\frac{1}{4}}$ (see \eqref{eq:ct_Linf}), $\phi(t)\lesssim(1+t)^{-\frac{1}{2}}$ (see \eqref{eq:bound_on_phi}), and $\frac{2}{\pi}\arctan\frac{R}{\sqrt{\lambda+t}}<1$, we see that dominated convergence applies twice and
\[\tag{vi}
\begin{split}
 \frac{1}{\pi\sqrt{2}\,}&\int_0^{T}\ud t\,t^{-\frac{3}{4}}\,c_f(t)\,\phi(t)\big({\textstyle \frac{2}{\pi}\arctan\frac{R}{\sqrt{\lambda+t}}}\big) \\
 &\qquad \xrightarrow[]{\;T\to +\infty\;}\;\;\frac{1}{\pi\sqrt{2}\,}\int_0^{+\infty}\!\!\ud t\,t^{-\frac{3}{4}}\,c_f(t)\,\phi(t)\big({\textstyle \frac{2}{\pi}\arctan\frac{R}{\sqrt{\lambda+t}}}\big) \\
 &\qquad \xrightarrow[]{\;R\to +\infty\;}\;\;\frac{1}{\pi\sqrt{2}\,}\int_0^{+\infty}\!\!\ud t\,t^{-\frac{3}{4}}\,c_f(t)\,\phi(t)\;=\; \\
 &\qquad\qquad\qquad\qquad=\;\frac{1}{\pi\sqrt{2}\,}\int_0^{+\infty}\!\!\ud t\,t^{-\frac{3}{4}}\,c_f(t)\,\frac{4\pi\alpha+\sqrt{\lambda}}{\,4\pi\alpha+\sqrt{\lambda+t}\,} \\
 &\qquad\qquad\qquad\qquad=\;(\alpha+{\textstyle\frac{\sqrt{\lambda}}{4\pi}})\,q_f\,,
\end{split}
\]
having used \eqref{eq:phi_double_form} and \eqref{eq:qs} in the last two steps.

From (v) and (vi) we find
\[
 \begin{split}
  F_f^{(0)}\;&=\;q_f\,(\alpha+{\textstyle\frac{\sqrt{\lambda}}{4\pi}})+ \\
  & \qquad +\frac{1}{\pi\sqrt{2}\,}\lim_{R\to +\infty}\lim_{T\to +\infty}\int_0^{T}\ud t\,t^{-\frac{3}{4}}\,\Big(c_{R,f}(t)-c_f(t)\big({\textstyle \frac{2}{\pi}\arctan\frac{R}{\sqrt{\lambda+t}}}\big)\Big) \\
  & \qquad -\frac{2\sqrt{2}}{\,(2\pi)^3}\lim_{R\to +\infty}\lim_{T\to +\infty}\int_0^{T}\ud t\,\frac{\,t^{\frac{1}{4}}c_f(t)}{\,4\pi\alpha+\sqrt{\lambda+t}\,}\;\times \\
  & \qquad\qquad\qquad\qquad\times\textstyle\Big(\int_{|p|<R}\frac{\ud p}{\,(p^2+\lambda)(p^2+\lambda+t)\,}-\frac{4\pi}{\sqrt{\lambda+t}+\sqrt{\lambda}}\,\arctan\frac{R}{\sqrt{\lambda+t}}\Big)\Big\}\,,
 \end{split}
\]
which implies (i) as long as one proves that
\[\tag{vii}
 \lim_{R\to +\infty}\lim_{T\to +\infty}\int_0^{T}\ud t\,t^{-\frac{3}{4}}\,\Big(c_f(t)\big({\textstyle \frac{2}{\pi}\arctan\frac{R}{\sqrt{\lambda+t}}}\big)-c_{R,f}(t)\Big)\;=\;0
\]
and
\[\tag{viii}
\begin{split}
  \lim_{R\to +\infty}&\lim_{T\to +\infty}\int_0^{T}\ud t\,\frac{\,t^{\frac{1}{4}}c_f(t)}{\,4\pi\alpha+\sqrt{\lambda+t}\,}\;\times \\
  &\qquad \times \textstyle\Big(\int_{|p|<R}\frac{\ud p}{\,(p^2+\lambda)(p^2+\lambda+t)\,}-\frac{4\pi}{\sqrt{\lambda+t}+\sqrt{\lambda}}\,\arctan\frac{R}{\sqrt{\lambda+t}}\Big)\Big\}\;=\;0\,.
\end{split}
\]

Last, let us establish (vii) and (viii), thus completing the proof. One has
\[
 \begin{split}
  &\int_0^{T}\ud t\,t^{-\frac{3}{4}}\,c_{R,f}(t)\;=\;\frac{1}{\,(2\pi)^{\frac{3}{2}}}\int_0^{T}\ud t\,t^{-\frac{3}{4}}\int_{|p|<R}\ud p\,\frac{\,(p^2+\lambda)^{\frac{3}{4}}\widehat{f}(p)}{p^2+\lambda+t} \\
  &\;\;\xrightarrow[]{\;T\to +\infty\;}\;\frac{1}{\,(2\pi)^{\frac{3}{2}}}\int_0^{+\infty}\!\!\ud t\,t^{-\frac{3}{4}}\int_{|p|<R}\ud p\,\frac{\,(p^2+\lambda)^{\frac{3}{4}}\widehat{f}(p)}{p^2+\lambda+t}\;=\;\int_0^{+\infty}\!\!\ud t\,t^{-\frac{3}{4}}\,c_{R,f}(t)
 \end{split}
\]
by dominated convergence, thanks to the uniform-in-$T$ summable majorant function $t\mapsto \mathrm{const}(R)\cdot t^{-\frac{3}{4}}(\lambda+t)^{-1}$. One also has
\[
 \begin{split}
  &\int_0^{T}\ud t\,t^{-\frac{3}{4}}\,c_f(t)\big({\textstyle \frac{2}{\pi}\arctan\frac{R}{\sqrt{\lambda+t}}}\big)\;\xrightarrow[]{\;T\to +\infty\;}\;\int_0^{+\infty}\!\!\ud t\,t^{-\frac{3}{4}}\,c_f(t)\big({\textstyle \frac{2}{\pi}\arctan\frac{R}{\sqrt{\lambda+t}}}\big)
 \end{split}
\]
by dominated convergence, thanks to the bound $\arctan(\frac{R}{\sqrt{\lambda+t}})\leqslant\frac{R}{\sqrt{\lambda+t}}$ and hence to the uniform-in-$T$ summable majorant function $t\mapsto R\, t^{-\frac{3}{4}}(\lambda+t)^{-\frac{1}{2}}$. Thus,
\[
 \begin{split}
  & \lim_{R\to +\infty}\lim_{T\to +\infty}\int_0^{T}\ud t\,t^{-\frac{3}{4}}\,\Big(c_f(t)\big({\textstyle \frac{2}{\pi}\arctan\frac{R}{\sqrt{\lambda+t}}}\big)-c_{R,f}(t)\Big)\;= \\
  &\qquad\qquad =\; \lim_{R\to +\infty}\int_0^{+\infty}\!\!\ud t\,t^{-\frac{3}{4}}\,\Big(c_f(t)\big({\textstyle \frac{2}{\pi}\arctan\frac{R}{\sqrt{\lambda+t}}}\big)-c_{R,f}(t)\Big)
 \end{split}
\]
Now, since $c_{R,f}(t)\nearrow c_f(t)$ and $\frac{2}{\pi}\arctan\frac{R}{\sqrt{\lambda+t}}\nearrow 1$ as $R\to +\infty$, the functions
\[
 t\mapsto t^{-\frac{3}{4}}\,\Big(c_f(t)\big({\textstyle \frac{2}{\pi}\arctan\frac{R}{\sqrt{\lambda+t}}}\big)-c_{R,f}(t)\Big)
\]
form a decreasing-in-$R$ net of summable functions, whose point-wise limit as $R\to +\infty$ is the null function. Therefore, by monotone convergence,
\[
 \lim_{R\to +\infty}\int_0^{+\infty}\!\!\ud t\,t^{-\frac{3}{4}}\,\Big(c_f(t)\big({\textstyle \frac{2}{\pi}\arctan\frac{R}{\sqrt{\lambda+t}}}\big)-c_{R,f}(t)\Big)\;=\;0
\]
and (vii) is proved.

Concerning (viii), with analogous bounds as above one  takes the limit $T\to +\infty$ based on dominated convergence. In order to take the limit $R\to +\infty$ in the resulting quantity
\[
 \int_0^{+\infty}\!\!\ud t\,\frac{\,t^{\frac{1}{4}}c_f(t)}{\,4\pi\alpha+\sqrt{\lambda+t}\,}\textstyle\Big(\int_{|p|<R}\frac{\ud p}{\,(p^2+\lambda)(p^2+\lambda+t)\,}-\frac{4\pi}{\sqrt{\lambda+t}+\sqrt{\lambda}}\,\arctan\frac{R}{\sqrt{\lambda+t}}\Big)
\]
one observes that
\[
 \begin{split}
  &\textstyle\Big(\frac{4\pi}{\sqrt{\lambda+t}+\sqrt{\lambda}}\,\arctan\frac{R}{\sqrt{\lambda+t}}-\int_{|p|<R}\frac{\ud p}{\,(p^2+\lambda)(p^2+\lambda+t)\,}\Big)\;= \\
  &\qquad = \;\textstyle\frac{4\pi}{\sqrt{\lambda+t}+\sqrt{\lambda}}\,\arctan\frac{R}{\sqrt{\lambda+t}}-\frac{1}{t} \Big(\sqrt{\lambda+t}\,\arctan\frac{R}{\sqrt{\lambda+t}}-\sqrt{\lambda}\,\arctan\frac{R}{\sqrt{\lambda}}\Big) \\
   &\qquad=\;\frac{\sqrt{\lambda}}{t}{\textstyle\Big(\arctan\frac{R}{\sqrt{\lambda}}- \arctan\frac{R}{\sqrt{\lambda+t}}\Big)}\;\leqslant\;\frac{\pi\sqrt{\lambda}}{t}\,,
 \end{split}
\]
which shows that the integrand vanishes point-wise in $t$ as $R\to +\infty$ and is bounded by a uniformly-in-$R$ integrable function: then dominated convergence applies and (viii) is also proved.
\end{proof}

\section{Fractional maps}\label{sec:fractional_maps}


In this Section we revisit part of the results of Sections \ref{sec:canonical-decomposition}-\ref{sec:subspaces_of_Ds} relative to the regime $s\in(\frac{1}{2},2)$ in terms of certain linear maps which it is very natural to introduce and which provide a more compact formulation.

For $s\in(\frac{1}{2},2)$, we define the linear maps
\begin{equation}
\mathcal{R}_s:H^s(\mathbb{R}^3)\to H^s(\mathbb{R}^3)\,,\qquad \mathcal{R}_s f\;:=\;f+w_f
\end{equation}
and
\begin{equation}
\mathcal{Q}_s:H^s(\mathbb{R}^3)\to \mathbb{C}\,,\qquad \mathcal{Q}_s f\;:=\;q_f\,,
\end{equation}
where $w_f$ is the function defined in \eqref{eq:w} and $q_f$ is the constant defined in \eqref{eq:qs}, for given $\alpha\geqslant 0$ and $\lambda>0$. Owing to Lemma \ref{lem:qf_scalar_product} and Proposition \ref{lem:w_in_Hs}, both maps are bounded:
\begin{equation}\label{eq:fractional_maps_are_bounded}
\|\mathcal{R}_s f\|_{H^s}\;\lesssim\;\|f\|_{H^s}\,,\qquad |\mathcal{Q}_s f|\;\lesssim\;\frac{1}{\,1+\alpha\,}\|f\|_{H^s}\,.
\end{equation}

As a consequence of Propositions \ref{prop:canonical_decomp} and \ref{lem:w_in_Hs},
\begin{equation}\label{eq:doms_RQ}
\mathcal{D}((-\Delta_\alpha)^{s/2})\;=\;\{\mathcal{R}_sf+(\mathcal{Q}_s f)G_\lambda\,|\,f\in H^s(\mathbb{R}^3)\}\,,
\end{equation}
that is, when $f$ spans $H^s(\mathbb{R}^3)$, $\mathcal{R}_sf$ spans all possible regular components and $(\mathcal{Q}_sf)G_\lambda$ spans all possible singular components of the elements of $\mathcal{D}((-\Delta_\alpha)^{s/2})$.

It is also convenient to write
\begin{equation}
\mathcal{R}_s\;=\;\mathbbm{1}-\mathcal{W}_s\,,\qquad \mathcal{W}_sf\;:=\;-w_f\,.
\end{equation}
The linear map $\mathcal{W}_s:H^s(\mathbb{R}^3)\to H^s(\mathbb{R}^3)$ is bounded, because of Proposition \ref{lem:w_in_Hs}.

%

\begin{proposition}\label{prop:RsQs}~
\begin{itemize}
 \item[(i)] When $s\in(\frac{1}{2},\frac{3}{2})$, the maps $\mathcal{R}_s$ and $\mathcal{Q}_s$ are surjective and not injective; moreover, there are functions in $\ker \mathcal{R}_s$ that do not belong to $\ker \mathcal{Q}_s$ and vice versa.
 \item[(ii)] Explicitly, when $s\in(\frac{1}{2},\frac{3}{2})$, the non-zero $H^s$-function
 \begin{equation}
 f_\star\;:=\;(-\Delta+\lambda\mathbbm{1})^{-s/2}J_\lambda\,,
 \end{equation}
 where $J_\lambda$ is the function defined in \eqref{eq:defJ-phi},
 satisfies
 \begin{equation}\label{eq:f*results}
 \mathcal{R}_s f_\star\;=\;0\,,\qquad\mathrm{and}\qquad\mathcal{Q}_s f_\star\;=\;1\,.
 \end{equation}
 \item[(iii)] For any $s\in(\frac{1}{2},2)$,
 \begin{equation}\label{eq:kernelQs}
  \ker\mathcal{Q}_s\;=\;(-\Delta+\lambda\mathbbm{1})^{-s/2}\big(\,\{\mathcal{G}_\lambda\}^{\perp}\big)
 \end{equation}
 in the sense of $L^2$-orthogonality, where $\mathcal{G}_\lambda$ is the function defined in \eqref{eq:GGlambda}.
 \item[(iv)] When $s=\frac{3}{2}$, $\mathcal{R}_{3/2}$ is injective and not surjective, whereas $\mathcal{Q}_s$ is surjective and not injective.
 \item[(v)] When $s\in(\frac{3}{2},2)$, $\mathcal{R}_s$ is surjective and injective, hence a bijection in $H^s(\mathbb{R}^3)$, whereas $\mathcal{Q}_s$ is surjective and not injective.
\end{itemize}
\end{proposition}

\begin{proof} (i) From \eqref{eq:doms_RQ} and from the fact that $H^s(\mathbb{R}^3)\subset \mathcal{D}((-\Delta_\alpha)^{s/2})$ (Proposition \ref{prop:Hs_in_domDalph}(i)) it follows that $\mathcal{R}_s$ is surjective and $\mathcal{Q}_s$ is not injective, and that there exist $f$'s in $H^s(\mathbb{R}^3)$ for which $\mathcal{R}_sf\neq 0$ whereas $\mathcal{Q}_sf=0$. From \eqref{eq:doms_RQ} again and from the fact that $\mathrm{span}\{G_\lambda\}\subset \mathcal{D}((-\Delta_\alpha)^{s/2})$ it follows that $\mathcal{Q}_s$ is surjective and $\mathcal{R}_s$ is not injective, and that there exist $f$'s in $H^s(\mathbb{R}^3)$ for which $\mathcal{Q}_sf\neq 0$ whereas $\mathcal{R}_sf=0$.

(ii) Owing to \eqref{eq:action_Dalpha},
 \[
 (-\Delta_\alpha+\lambda\mathbbm{1})^{-s/2}J_\lambda\;=\;(-\Delta_\alpha+\lambda\mathbbm{1})^{-s/2}(-\Delta+\lambda\mathbbm{1})^{s/2}f_\star\;\in\;\mathcal{D}((-\Delta_\alpha)^{s/2})\,,
 \]
 whence also, owing to \eqref{eq:Dalph_on_Glam}, as well as to \eqref{eq:gdecomp}, \eqref{eq:reg_comp_fg}, and \eqref{eq:action_Dalpha},
 \[
 G_\lambda\;=\;(-\Delta_\alpha+\lambda\mathbbm{1})^{-s/2}J_\lambda\;=\;\mathcal{R}_sf_\star+(\mathcal{Q}_s f_\star)G_\lambda\,,
 \]
 from which \eqref{eq:f*results} follows.

(iii) The identity \eqref{eq:kernelQs} is precisely equation \eqref{eq:vanishing_qf} proved in Lemma \ref{lem:qf_scalar_product}.

(iv)-(v) The surjectivity of $\mathcal{Q}_s$ is obvious, and its non-injectivity is proved in general in part (iii) above.

For the injectivity of $\mathcal{R}_s$ when $s\in[\frac{3}{2},2)$ we exploit the fact, encoded in \eqref{eq:doms_RQ}, that if $f\in H^s(\mathbb{R}^3)$, then $g:=\mathcal{R}_sf+(\mathcal{Q}_s f)G_\lambda$ is an element of $\mathcal{D}((-\Delta_\alpha)^{s/2})$ and \eqref{eq:action_Dalpha} implies that $f=(-\Delta+\lambda\mathbbm{1})^{-s/2}(-\Delta_\alpha+\lambda\mathbbm{1})^{s/2}g$. Therefore, if $\mathcal{R}_sf=0$, then necessarily $\mathcal{Q}_s f=0$ (for otherwise $G_\lambda$ would belong to $\mathcal{D}((-\Delta_\alpha)^{s/2})$ for $s\geqslant\frac{3}{2}$, which is forbidden by Proposition \ref{prop:eq:Dalph_on_Glam}), whence also $g=0$ and then $f=0$: $\mathcal{R}_s$ is injective.

The lack of surjectivity of $\mathcal{R}_{3/2}$ is a consequence of Proposition \ref{prop:high_s_in_Ds}(ii), as is evident from comparing the expressions \eqref{eq:Ds_for_s32} and \eqref{eq:doms_RQ} for $\mathcal{D}((\Delta_\alpha)^{3/4})$, taking into account that $\mathcal{D}^{(3/2)}_0\varsubsetneq H^\frac{3}{2}(\mathbb{R}^3)$.


When $s\in(\frac{3}{2},2)$ one can prove the invertibility of $\mathcal{R}_s=\mathbbm{1}-\mathcal{W}_s$ as a bijection on $H^s(\mathbb{R}^3)$ by means of the following 
argument. The bound \eqref{eq:Hsl_norm_w} found in Proposition \ref{lem:w_in_Hs} in the present notation reads
\[
 \|(p^2+\lambda)^{\frac{s}{2}}\,\widehat{\mathcal{W}_sf}\|_2\;\leqslant\;\frac{\sqrt{2}\,\sin\frac{s\pi}{2}}{\:\sin(\frac{s\pi}{2}-\frac{\pi}{4})}\,\|(p^2+\lambda)^{\frac{s}{2}}\widehat{f}\|_2\qquad\forall f\in H^s(\mathbb{R}^3)\,.
\]
Since
\[
s\;\longmapsto\; \frac{\sqrt{2}\,\sin\frac{s\pi}{2}}{\:\sin(\frac{s\pi}{2}-\frac{\pi}{4})}
\]
is continuous and strictly monotone decreasing, attaining the value 1 at $s=\frac{3}{2}$, then for $s\in(\frac{3}{2},2)$ the map $\mathcal{F}\mathcal{W}_s\mathcal{F}^{-1}$ (where $\mathcal{F}:L^2(\mathbb{R}^3,\ud x)\to L^2(\mathbb{R}^3,\ud p)$ is the Fourier transform, inherited also on $H^s(\mathbb{R}^3,\ud x)$) is bounded on the space $L^2(\mathbb{R}^3,(p^2+\lambda)^s\ud p)$ \emph{with norm strictly smaller than 1}. As a consequence, $\mathcal{F}\mathcal{R}_s\mathcal{F}^{-1}=\mathbbm{1}-\mathcal{F}\mathcal{W}_s\mathcal{F}^{-1}$ is a bijection on such space. Using an obvious isomorphism $L^2(\mathbb{R}^3,(p^2+\lambda)^s\ud p)\stackrel{\cong}{\longmapsto}L^2(\mathbb{R}^3,(p^2+1)^s\ud p)=\mathcal{F} H^s(\mathbb{R}^3,\ud x)$, one then concludes that the map $\mathcal{R}_s$ is a bijection on $H^s(\mathbb{R}^3,\ud x)$.
\end{proof}

%
%
%


\section{Proofs of the main results and transition behaviours}\label{sec:proofs}

\begin{proof}[Proof of Theorem \ref{thm:domain_s}]~

(i) Case $s\in(0,\frac{1}{2})$. Let $g\in\mathcal{D}((-\Delta_\alpha)^{s/2})$. Owing to Proposition  \ref{prop:canonical_decomp}, 
$g=f_g+h_g$ with $f_g\in H^s(\mathbb{R}^3)$ given by \eqref{eq:reg_comp_fg} and $h_g$ given by \eqref{eq:sing_comp_hg}. In Proposition \ref{pro:h_in_Hs} we established that $h_g\in H^s(\mathbb{R}^3)$ too, therefore $\mathcal{D}((-\Delta_\alpha)^{s/2})\subset H^s(\mathbb{R}^3)$. Conversely, in Proposition \ref{prop:Hs_in_domDalph}(i) we established that $\mathcal{D}((-\Delta_\alpha)^{s/2})\supset H^s(\mathbb{R}^3)$. The conclusion is the identity \eqref{eq:Ds_s0-1/2}.

(ii) Case $s\in(\frac{1}{2},\frac{3}{2})$. Again, owing to Proposition  \ref{prop:canonical_decomp}, a generic $g\in\mathcal{D}((-\Delta_\alpha)^{s/2})$ decomposes as
$g=f_g+h_g$ with $f_g\in H^s(\mathbb{R}^3)$ given by \eqref{eq:reg_comp_fg} and $h_g$ given by \eqref{eq:sing_comp_hg}. In Proposition \ref{lem:w_in_Hs} we established that $h_g=q_{f_g}G_\lambda+w_{f_g}$ for some $q_{f_g}\in\mathbb{C}$ and some $w_{f_g}\in H^s(\mathbb{R}^3)$. Therefore, $\mathcal{D}((-\Delta_\alpha)^{s/2})\subset H^s(\mathbb{R}^3)+\mathrm{span}\{G_\lambda\}$. Conversely, in Propositions \ref{prop:eq:Dalph_on_Glam} and  \ref{prop:Hs_in_domDalph}(i) we established the opposite inclusion \eqref{eq:Ds_supset_Hs_Gl}. The conclusion is the identity \eqref{eq:Ds_s1/2-3/2}.


(iii) Case $s\in(\frac{3}{2},2)$. Owing to Propositions \ref{prop:canonical_decomp} and \ref{lem:w_in_Hs}, $\mathcal{D}((-\Delta_\alpha)^{s/2})$ consists exactly of elements of the form $f+q_fG_\lambda+w_f$, obtained by letting $f$ span the whole $H^s(\mathbb{R}^3)$ and by taking $w_f$ and $q_f$ according to \eqref{eq:w} and \eqref{eq:qs}. (With the same argument as in the proof of part (ii), this  allows one to deduce again $\mathcal{D}((-\Delta_\alpha)^{s/2})\subset H^s(\mathbb{R}^3)\dotplus\mathrm{span}\{G_\lambda\}$, however the latter is now a strict inclusion, as established in Propositions \ref{prop:eq:Dalph_on_Glam} and \ref{prop:Hs_in_domDalph}(ii).) It follows from Proposition \ref{lem:w_in_Hs} that $F_f:=f+w_f\in H^s(\mathbb{R}^3)$ and it follows from Lemma \ref{lem:qf_controlledby_F0} that $F_\lambda$ is a continuous function on $\mathbb{R}^3$ with $F_f(0)=(\alpha+\frac{\sqrt{\lambda}}{4\pi})q_f$. Thus,
\[
\mathcal{D}((-\Delta_\alpha)^{s/2})\:\subset\:\Big\{F+\frac{F(0)}{\,\alpha+\frac{\sqrt{\lambda}}{4\pi}\,}G_\lambda\,\Big|\,F\in H^s(\mathbb{R}^3)\Big\}\,.
\]
Conversely, we established in Proposition \ref{prop:high_s_in_Ds} that
\[
 \mathcal{D}((-\Delta_\alpha)^{s/2})\;\supset\;\Big\{F+\frac{F(0)}{\,\alpha+\frac{\sqrt{\lambda}}{4\pi}\,}G_\lambda\,\Big|\,F\in H^s(\mathbb{R}^3)\Big\}\,,
\]
because in this regime of $s$ the space $\mathcal{D}_0^{(s)}$ used in Proposition \ref{prop:high_s_in_Ds} is the whole $H^s(\mathbb{R}^3)$ and $F^{(0)}=F(0)$. The conclusion is the identity \eqref{eq:Ds_s3/2-2}. Alternatively, in the equivalent language of the fractional maps introduced in Section \ref{sec:fractional_maps}, one argues as follows: according to \eqref{eq:doms_RQ},
\[
\mathcal{D}((-\Delta_\alpha)^{s/2})\;=\;\{\mathcal{R}_sf+(\mathcal{Q}_s f)G_\lambda\,|\,f\in H^s(\mathbb{R}^3)\}\,,
\]
Lemma \ref{lem:qf_controlledby_F0} reads
\[
\mathcal{Q}_s f\;=\;\frac{\:(\mathcal{R}_sf)(0)\:}{\,\alpha+\frac{\sqrt{\lambda}}{2\pi}\,}\,,
\]
and Proposition \ref{prop:RsQs}(v) establishes that $\mathcal{R}_s:
H^s(\mathbb{R}^3)\to H^s(\mathbb{R}^3)$ is a bijection, which all together gives precisely the representation \eqref{eq:Ds_s3/2-2} for $\mathcal{D}((-\Delta_\alpha)^{s/2})$.
\end{proof}

\begin{proof}[Proof of Theorem \ref{thm:equiv_of_norms}]~

(i) Case $s\in(0,\frac{1}{2})$. The bound $\|g\|_{H^s_\alpha}\lesssim\|g\|_{H^s}$ was proved in \eqref{eq:equiv_Hs_norm} of Proposition \eqref{prop:Hs_in_domDalph}(i). As for the opposite bound, Proposition \ref{prop:canonical_decomp} implies that $g=f_g+h_g$ and $\|g\|_{H^s_\alpha}=\|(-\Delta+\lambda\mathbbm{1})^{s/2}f_g\|_2\approx\|f_g\|_{H^s}$, Proposition \ref{pro:h_in_Hs} implies that $\|h_g\|_{H^s}\lesssim\|f_g\|_{H^s}$, therefore $\|g\|_{H^s}\lesssim \|f_g\|_{H^s}=\|g\|_{H^s_\alpha}$.

(ii) Case $s\in(\frac{1}{2},\frac{3}{2})$. By means of the decomposition of Propositions  \ref{prop:canonical_decomp} and \ref{lem:w_in_Hs}, as well as the surjectivity of the map $f\mapsto f+w_f$ on $H^s(\mathbb{R}^3)$ (Proposition \ref{prop:RsQs}(i)), one has $g=F_{f_g}+q_{f_g} G_\lambda$ with $F_{f_g}={f_g}+w_{f_g}$, and 
$\|g\|_{H^s_\alpha}=\|(-\Delta+\lambda\mathbbm{1})^{s/2}{f_g}\|_2\approx\|{f_g}\|_{H^s}$.
Combining this norm equivalence with the bounds $\|F_{f_g}\|_{H^s}\lesssim\|{f_g}\|_{H^s}$ and $(1+\alpha)|q_{f_g}|\lesssim\|{f_g}\|_{H^s}$ (Lemma \ref{lem:qf_scalar_product} and Proposition \ref{lem:w_in_Hs}, i.e., eq.~\eqref{eq:fractional_maps_are_bounded}) one has $\|F_f\|_{H^s}+(1+\alpha)|q_{f_g}|\lesssim\|F_f+q_{f_g} G_\lambda\|_{H^s_\alpha}$. For the opposite inequality we write $\|F_{f_g}+q_{f_g} G_\lambda\|_{H^s_\alpha}\leqslant\|F_{f_g}\|_{H^s_\alpha}+|q_{f_g}|\|G_\lambda\|_{H^s_\alpha}$ and we use $\|F_{f_g}\|_{H^s_\alpha}\lesssim\|F_{f_g}\|_{H^s}$ (eq.~\eqref{eq:equiv_Hs_norm} in Proposition \ref{prop:Hs_in_domDalph}) and $\|G_\lambda\|_{H^s_\alpha}\lesssim(1+\alpha)$ (eq.~\eqref{eq:Dalph_on_Glam2} in Proposition \ref{eq:Dalph_on_Glam}), whence the conclusion.

(iii) Case $s\in(\frac{3}{2},2)$. Arguing as in part (ii), for $F_\lambda+{\textstyle\frac{F_\lambda(0)}{\,\alpha+\frac{\sqrt{\lambda}}{4\pi}\,}}\,G_\lambda$ one has $F_\lambda=f+w_f=R_s f$, $\frac{F_\lambda(0)}{\,\alpha+\frac{\sqrt{\lambda}}{4\pi}\,}=Q_s f$, and  $\big\|F_\lambda+{\textstyle\frac{F_\lambda(0)}{\,\alpha+\frac{\sqrt{\lambda}}{4\pi}\,}}\,G_\lambda\big\|_{H^s_\alpha}\approx\|f\|_{H^s}$.
Since in the regime $s\in(\frac{3}{2},2)$ the map $R_s$ is invertible on $H^s(\mathbb{R}^3)$ (Proposition \ref{prop:RsQs}(v)), and hence also with bounded inverse, then $\|f\|_{H^s}\approx\|R_s f\|_{H^s}=\|F_\lambda\|_{H^s}$, which completes the proof.
\end{proof}

%
%

\begin{proof}[Proof of Theorem \ref{thm:fract_pow_formulas}]~

Part (i) was proved already in the end of Section \ref{sec:canonical-decomposition}. Part (ii) is entirely proved in Proposition \ref{prop:eq:Dalph_on_Glam} and Lemma \ref{lem:Jlambda}.
\end{proof}

The transition cases $s=\frac{1}{2}$ and $s=\frac{3}{2}$ are characterised as follows.

\begin{proposition}[Transition case $s=\frac{1}{2}$]\label{prop:domain_s_12}
Let $\alpha\geqslant 0$, $\lambda>0$, and $s=\frac{1}{2}$. Then
\begin{equation}\label{eq:Ds_s1/2}
\mathcal{D}((-\Delta_\alpha)^{1/4})\;=\;\{f+h_f\,|\,f\in H^{\frac{1}{2}}(\mathbb{R}^3)\}\,,
\end{equation}
where, for given $f$, $h_f$ is the $H^{\frac{1}{2}^-}$\!\!-function defined in \eqref{eq:cgeneral}-\eqref{eq:hgeneral} and discussed in Proposition \ref{pro:h_in_Hs}(ii). Moreover,
\begin{equation}\label{eq:control_domain_12}
 H^{\frac{1}{2}}(\mathbb{R}^3)\dotplus\mathrm{span}\{G_\lambda\}\;\varsubsetneq\;\mathcal{D}((-\Delta_\alpha)^{1/4})\;\subset\;H^{{\frac{1}{2}}^-}\!(\mathbb{R}^3)\,.
\end{equation}
\end{proposition}

\begin{proof}
The first statement is an immediate consequence of the canonical decomposition \eqref{eq:gdecomp} of Proposition \ref{prop:canonical_decomp} and of the definition \eqref{eq:cgeneral}-\eqref{eq:hgeneral}. The inclusion $\mathcal{D}((-\Delta_\alpha)^{1/4})\subset H^{{\frac{1}{2}}^-}\!(\mathbb{R}^3)$ of \eqref{eq:control_domain_12} follows at once from the decomposition \eqref{eq:gdecomp} and from Proposition \ref{pro:h_in_Hs}(ii), whereas the inclusion $H^{1/2}(\mathbb{R}^3)\dotplus\mathrm{span}\{G_\lambda\}\subset\mathcal{D}((-\Delta_\alpha)^{1/4})$ is precisely the inclusion \eqref{eq:Ds_supset_Hs_Gl} for $s=\frac{1}{2}$, which follows from Propositions \ref{prop:eq:Dalph_on_Glam} and \ref{prop:Hs_in_domDalph}(ii). Last, in order to see that the latter inclusion is strict, we observe that in course of the proof of Proposition \ref{pro:h_in_Hs}(ii) certain non-zero functions $f\in H^{1/2}(\mathbb{R}^3)$ were considered for which $\widehat{h_f}(p)\approx \langle p \rangle^{-2}\ln\langle p\rangle$ as $|p|\to +\infty$, which is logarithmically more singular than $G_\lambda$ and than an $H^{\frac{1}{2}}$-function.
\end{proof}

\begin{proposition}[Transition case $s=\frac{3}{2}$]\label{prop:domain_s_32}
Let $\alpha\geqslant 0$, $\lambda>0$, and $s=\frac{3}{2}$. Then
  \begin{equation}\label{eq:Ds_for_s32_Prop}
 \mathcal{D}((-\Delta_\alpha)^{3/4})\;=\;\Big\{F+\frac{F^{(0)}}{\,\alpha+\frac{\sqrt{\lambda}}{4\pi}\,}\,G_\lambda\,\Big|\,F\in \mathcal{D}_0^{(3/2)}\Big\}\,,
 \end{equation}
 where
 \begin{equation}
 \mathcal{D}_0^{(3/2)}\;=\;\left\{ F\in H^{\frac{3}{2}}(\mathbb{R}^3)\left|\!\!
 \begin{array}{c}
  F^{(0)}:=\displaystyle\frac{1}{\;(2\pi)^{\frac{3}{2}}}\int_{\mathbb{R}^3}\!\ud p\,\widehat{F}(p)<+\infty \\
  I_F+\frac{F^{(0)}}{\,\alpha+\frac{\sqrt{\lambda}}{4\pi}\,}\,J_\lambda\in L^2(\mathbb{R}^3)
 \end{array}
 \!\!\!\right.\right\}\varsubsetneq\;H^{\frac{3}{2}}(\mathbb{R}^3)\,,
\end{equation}
$I_F$ is the function defined by \eqref{eq:IF}-\eqref{eq:IF_kF} for given $F$ and $s=\frac{3}{2}$, and $J_\lambda$ is the function defined by \eqref{eq:defJ}-\eqref{eq:defJ-phi}.
%
\end{proposition}

\begin{proof}
 An immediate consequence of Proposition \ref{prop:high_s_in_Ds}.
\end{proof}

\appendix

\section{A general Schur bound}\label{app:Schur}

In this Appendix we establish, by means of the Schur test, the boundedness of an integral operator appearing frequently in our analysis.

We start with the following useful identities.

\begin{lemma}
For $a\in(0,1)$ and $R> 0$ one has
\begin{equation}\label{eq:useful_integral_repr}
R^{1-a}\!\int_0^{+\infty}\!\!\ud t\,\frac{\,t^{a-1}}{\:R+t\,}\;=\;\int_0^{+\infty}\!\!\ud t\,\frac{\,t^{a-1}}{\:1+t\,}\;=\;2\!\int_0^{+\infty}\!\!\ud t\, \frac{t^{2a-1}}{\,1+t^2}\;=\;\frac{\pi}{\sin a\pi}\,.
\end{equation}
\end{lemma}

\begin{proof}
Upon obvious changes of variables, it suffices to prove
\[
 \int_0^{+\infty}\!\!\ud t\,\frac{\,t^{a-1}}{\:1+t\,}\;=\;\frac{\pi}{\sin a\pi}\,.
\]
From the representation
\[
\lambda^{-a}\;=\;\frac{1}{\Gamma(a)} \int_0^{+\infty}\!\!\ud t\, t^{a-1} e^{-\lambda t}\,, \qquad \lambda >0\,,
\]
derived by the integral representation of the Gamma function
\[
\Gamma (a)=\int _{0}^{+\infty }\!\!\ud t\,t^{a-1}\,e^{-t}\,,
\]
and from
\[
\Gamma(a) \Gamma(1-a)\;=\; \frac{\pi}{\sin(\pi a)}\,,
\]
we find
\[
\begin{split}
 \int_0^{+\infty}\!\!\ud t\,\frac{\,t^{a-1}}{\:1+t\,}\;&=\;\int_0^{+\infty}\!\!\ud t\, t^{a-1}\! \int_0^{+\infty}\!\!\ud \lambda\, e^{-\lambda (1+t)} \;=\;\int_0^{+\infty}\!\!\ud\lambda\,e^{-\lambda}\!\int_0^{+\infty}\!\!\ud t\,t^{a-1}\,e^{-\lambda t} \\
 &=\;\Gamma(a)\!\int_0^{+\infty}\!\!\ud\lambda\,e^{-\lambda}\,\lambda^{-a}\;=\;\Gamma(a)\Gamma(1-a)\;=\;\frac{\pi}{\sin a\pi}\,,
\end{split}
\]
which completes the proof.
\end{proof}

Based on the above integral formula and the Schur boundedness criterion, we establish the following.

\begin{proposition}\label{prop:Tbeta_operators}
For a given constant $\beta\in(-\frac{1}{2},\frac{1}{2})$ and a measurable function $f$ on $\mathbb{R}^+$, let
\begin{equation}\label{eq:defKbeta}
K_\beta(\varrho,t)\;:=\;\frac{\,t^\beta \varrho^{-\beta}}{\varrho+t}\,,\qquad \varrho,t>0\,,
\end{equation}
and
\begin{equation}\label{eq:defTbeta}
(T_\beta f)(\varrho)\;:=\;\int_0^{+\infty}\!\!\ud t\,K_\beta(\varrho,t)\,f(t)\,.
\end{equation}
Then
\begin{equation}
\int_0^{+\infty}\!\!\ud\varrho\,|(T_\beta f)(\varrho)|^2\;\leqslant\;\Big(\frac{\pi}{\cos\beta\pi}\Big)^2\!\int_0^{+\infty}\!\!\ud t\,|f(t)|^2\,,
\end{equation}
therefore $T_\beta$ defines a bounded linear map on $L^2(\mathbb{R}^+)$ with norm
\begin{equation}\label{eq:Tbeta_boundedness}
\|T_\beta\|_{L^2(\mathbb{R}^+)\to L^2(\mathbb{R}^+)}\;\leqslant\;\frac{\pi}{\cos\beta\pi}\,.
\end{equation}
\end{proposition}

\begin{proof}
If we prove
\begin{equation}\label{eq:apply_schur}
\begin{split}
&\int_0^{+\infty}\!\!\ud \varrho\, K_\beta(\varrho,t)\,\varphi(\rho)\;\leqslant\;\frac{\pi}{\cos\beta\pi}\,\varphi(t) \\
&\qquad\quad\textrm{and }\qquad \int_0^{+\infty}\!\!\ud t\, K_\beta(\varrho,t)\,\varphi(t)\;\leqslant\;\frac{\pi}{\cos\beta\pi}\,\varphi(\varrho)
\end{split}
\end{equation}
for some measurable $\varphi$ on $\mathbb{R}^+$ with $\varphi(r)>0$, then \eqref{eq:Tbeta_boundedness} follows by a standard Schur test. With the choice $\varphi(r):=r^{-1/2}$ we find
\[
\begin{split}
\int_0^{+\infty}\!\!\ud \varrho\, K_\beta(\varrho,t)\,\varphi(\rho)\;&=\;t^\beta\!\int_0^{+\infty}\!\!\ud \varrho\,\frac{\, \varrho^{-\beta-\frac{1}{2}}}{\varrho+t}\;=\;t^\beta\,\frac{\pi}{\sin (\frac{1}{2}-\beta)\pi}\,\frac{1}{\,t^{\beta+\frac{1}{2}}}\;=\;\frac{\pi\varphi(t)}{\cos \beta\pi}\,,
\end{split}
\]
where we used \eqref{eq:useful_integral_repr} with $a=\frac{1}{2}-\beta$, and analogously
\[
\begin{split}
\int_0^{+\infty}\!\!\ud t\, K_\beta(\varrho,t)\,\varphi(t)\;&=\;\varrho^{-\beta}\!\int_0^{+\infty}\!\!\ud t\,\frac{\, t^{\beta-\frac{1}{2}}}{\varrho+t}\;=\;\frac{1}{\,\varrho^\beta}\,\frac{\pi}{\sin (\frac{1}{2}+\beta)\pi}\,\frac{1}{\,\varrho^{\frac{1}{2}-\beta}}\;=\;\frac{\pi\varphi(\varrho)}{\cos \beta\pi}\,,
\end{split}
\]
where we used \eqref{eq:useful_integral_repr} with $a=\frac{1}{2}+\beta$, therefore we obtain precisely \eqref{eq:apply_schur}.
\end{proof}

The form in which we actually apply Proposition \ref{prop:Tbeta_operators} in our analysis is given by the following Corollary.

\begin{corollary}\label{cor:schur}
For given constants $\beta\in(-\frac{1}{2},\frac{1}{2})$ and $\gamma,\delta> 0$,
and a measurable function $f$ on $\mathbb{R}^+$, let
\begin{equation}\label{eq:defQkernelSchur}
 Q_{\beta,\gamma,\delta}(u,v)\;:=\;\frac{\:u^{(\frac{1}{2}-\beta) \gamma-\frac{1}{2}}  \:v^{(\frac{1}{2}+\beta)\delta-\frac{1}{2}}}{u^\gamma+v^\delta}\,,\qquad u ,v>0 \,,
\end{equation}
and
\begin{equation}\label{eq:defQopSchur}
(Q_{\beta,\gamma,\delta} f)(u)\;:=\;\int_0^{+\infty}\!\!\ud v\,Q_{\beta,\gamma,\delta}(u,v)\,f(v)\,.
\end{equation}
Then $Q_{\beta,\gamma,\delta}$ defines a bounded linear map on $L^2(\mathbb{R}^+)$ with norm
\begin{equation}\label{eq:Q_boundedness}
\|Q_{\beta,\gamma,\delta}\|_{L^2(\mathbb{R}^+)\to L^2(\mathbb{R}^+)}\;\leqslant\;\frac{1}{\sqrt{\gamma\delta}\,}\,\frac{\pi}{\cos\beta\pi}\,.
\end{equation}
\end{corollary}

\begin{proof}
Setting
\[
 g(v)\;:=\;{\textstyle\frac{1}{\sqrt{d}\,}}\,v^{\frac{1-d}{2d}}\,f(v^{\frac{1}{d}})\,,
\]
the change of variable $v\mapsto v^\delta$ yields
\[
\int_{0}^{+\infty}\!\ud v\,|f(v)|^2\;=\;\int_{0}^{+\infty}\!\ud v\,|g(v)|^2
\]
and the change of variable $u\mapsto u^\gamma$, $v\mapsto v^\delta$ yields
\[
\int_{0}^{+\infty}\!\ud u\,|(Q_{\beta,\gamma,\delta}f)(u)|^2\;=\;\int_{0}^{+\infty}\!\ud v\,|{\textstyle \frac{1}{\sqrt{\gamma\delta}\,}}(T_\beta\, g)(v)|^2\,,
\]
where $T_\beta:L^2(\mathbb{R}^+)\to L^2(\mathbb{R}^+)$ is the operator defined in \eqref{eq:defKbeta}-\eqref{eq:defTbeta}. From this and from Proposition \ref{prop:Tbeta_operators},
\[
\begin{split}
 \|Q_{\beta,\gamma,\delta}f\|_{L^2(\mathbb{R}^+)}\;&=\;{\textstyle \frac{1}{\sqrt{\gamma\delta}\,}}\|T_\beta\, g\|_{L^2(\mathbb{R}^+)}\;\leqslant\;{\textstyle \frac{1}{\sqrt{\gamma\delta}\,}\,\frac{\pi}{\cos\beta\pi}}\,\|g\|_{L^2(\mathbb{R}^+)} \\
 &=\;{\textstyle \frac{1}{\sqrt{\gamma\delta}\,}\,\frac{\pi}{\cos\beta\pi}}\,\|f\|_{L^2(\mathbb{R}^+)}\,,
 \end{split}
\]
and \eqref{eq:Q_boundedness} follows.
\end{proof}

\section{An alternative proof for the transition $s=\frac{3}{2}$}\label{app:alternative_proofs}

We re-demonstrate in this Appendix, independently of the proof given for Proposition \ref{prop:high_s_in_Ds}, the following remarkable feature of the transition at $s=\frac{3}{2}$.

\begin{proposition}[Loss of genericity of the regular component for $s=\frac{3}{2}$]\label{cor:lemmaF0}
 Fix $\alpha\geqslant 0$, $\lambda>0$ and $s=\frac{3}{2}$. Suppose that $F+\kappa\,G_{\lambda}\in\mathcal{D}((-\Delta_\alpha)^{3/4})$ for some $F\in H^{3/2}(\R^3)$ and $\kappa\in\C$. Then $F^{(0)}$ defined in \eqref{eq:def_subspace_Ds} is finite.
\end{proposition}

\begin{proof}
Owing to the decompositions given by Propositions \ref{prop:canonical_decomp} and \ref{lem:w_in_Hs}, $F=f+w$ for some $f\in H^{3/2}(\mathbb{R}^3)$ and the corresponding $w$ given by \eqref{eq:cgeneral-Fourier}-\eqref{eq:w}.
Thus,
\[
  F^{(0)}\;=\;\lim_{R\to +\infty}\frac{1}{(2\pi)^{\frac32}}\int_{|p|<R}\widehat{F}(p)\,\ud p
\]
and
\[
 \begin{split}
  \int_{|p|<R}&\widehat{F}(p)\,\ud p\;=\int_{|p|<R}\!\ud p\left(\widehat{f}(p)-\frac{1}{\pi^{\frac32}(p^2+\lambda)}\int_0^{+\infty}\!\!\ud t\,\frac{t^{\frac{1}{4}}\,c(t)}{4\pi\alpha+\sqrt{\lambda+t}}\,\frac{1}{p^2+\lambda+t}\right) \\
  =&\;\int_{\mathbb{R}^3}\!\ud p\,\widehat{f}(p)\,(p^2+\lambda)^{\frac{3}{4}}\,\frac{\mathbf{1}_{\{|p|<R\}}}{\,(p^2+\lambda)^{\frac{3}{4}}}\:- \\
  -&\,\frac{1}{(2\pi^2)^{\frac{3}{2}}}\int_{|p|<R}\!\frac{\ud p}{\:p^2+\lambda}\int_0^{+\infty}\!\!\ud t\,\frac{t^{\frac{1}{4}}}{4\pi\alpha+\sqrt{\lambda+t}}\frac{1}{p^2+\lambda+t}\int_{\mathbb{R}^3}\!\ud q\,\frac{(q^2+\lambda)^{\frac{3}{4}}\widehat{f}(q)}{q^2+\lambda+t} \\
  \equiv&\;(\mathrm{I})-(\mathrm{II})\,.
 \end{split}
\]
Since the above expression is linear in $f$, it is not restrictive to assume $\widehat{f}(p)\geqslant 0$. The term (II) consists of an integration over a ball in the $p$ variable of the Fourier transform of a function that is surely in $H^{3/2}(\mathbb{R}^3)$ (since so is $w$, Proposition \ref{lem:w_in_Hs}): therefore, the $p$-integration is obviously finite and by dominated convergence the result equals the limit, as $K\to +\infty$, of the same expression when the integration over $q$ is truncated on $|q|<K$. The double integration in $p$ and $q$, both truncated to finite balls, can then be exchanged, and the limit $K\to +\infty$ can again be taken. Thus,
\[
\begin{split}
 (\mathrm{II})\;&=\;\frac{1}{(2\pi^2)^{\frac{3}{2}}}\int_{\mathbb{R}^3}\ud q\,\widehat{f}(q)\,(q^2+\lambda)^{\frac{3}{4}}\!\int_0^{+\infty}\!\!\ud t\,\frac{t^{\frac{1}{4}}}{4\pi\alpha+\sqrt{\lambda+t}}\frac{1}{q^2+\lambda+t}\;\times \\
 &\qquad\qquad\qquad\qquad \times \int_{|p|<R}\!\frac{\ud p}{\,(p^2+\lambda)(p^2+\lambda+t)\,}\,.
\end{split}
\]
Upon renaming the variables $p\leftrightarrow q$ in (II), one therefore obtains
\[
 \int_{|p|<R}\widehat{F}(p)\,\ud p\;=\;\int_{\mathbb{R}^3}\!\ud p\,\widehat{f}(p)\,(p^2+\lambda)^{\frac{3}{4}}\,\Big(\frac{\mathbf{1}_{\{|p|<R\}}}{\,(p^2+\lambda)^{\frac{3}{4}}}-\Omega_R(p)\Big)\,,
\]
where
\[
 \Omega_R(p)\;:=\;\frac{1}{(2\pi^2)^{\frac{3}{2}}}\int_{|q|<R}\frac{\ud q}{\,q^2+\lambda\,}\!\int_0^{+\infty}\!\!\ud t\,\frac{t^{\frac{1}{4}}}{\,(4\pi\alpha+\sqrt{\lambda+t})(p^2+\lambda+t)(q^2+\lambda+t)\, }\,.
\]
In the $p$-integration above the function $p\mapsto \widehat{f}(p)\,(p^2+\lambda)^{\frac{3}{4}}$ belongs to $L^2(\mathbb{R}^3)$, since $f\in H^{3/2}(\R^3)$.

Our strategy will be the following. First, we split
\[\tag{i}
  (p^2+\lambda)^{-\frac{3}{4}}\mathbf{1}_{\{|p|<R\}}-\Omega_R\;=\;-\Omega_R^{(2)}-\big(\Gamma_R^{(2)}+\Gamma_R^{(3)}+\Gamma_R^{(4)}\big)
\]
for four $L^2$-functions $\Omega_R^{(2)}$, $\Gamma_R^{(2)}$, $\Gamma_R^{(3)}$, and $\Gamma_R^{(4)}$, and then we prove that
\[\tag{ii}
 \begin{split}
  \Omega_R^{(2)}\;&\xrightarrow[]{\;R\to +\infty\;}\;\Omega^{(2)}\qquad\textrm{strongly in }L^2(\mathbb{R}^3) \\
  \Gamma_R^{(2)}\;&\xrightarrow[]{\;R\to +\infty\;}\;\Gamma^{(2)}\qquad\textrm{strongly in }L^2(\mathbb{R}^3) \\
  \Gamma_R^{(3)}\;&\xrightarrow[]{\;R\to +\infty\;}\;0\qquad\quad\:\textrm{weakly in }L^2(\mathbb{R}^3) \\
  \Gamma_R^{(4)}\;&\xrightarrow[]{\;R\to +\infty\;}\;0\qquad\quad\;\textrm{weakly in }L^2(\mathbb{R}^3)
 \end{split}
\]
for some $L^2$-functions $\Omega^{(2)}$ and $\Gamma^{(2)}$.
As a consequence of (i) and (ii) we then conclude the convergence of the quantity
\[
 \int_{|p|<R}\widehat{F}(p)\,\ud p\;=\;-\!\int_{\mathbb{R}^3}\!\ud p\,\widehat{f}(p)\,(p^2+\lambda)^{\frac{3}{4}}\,\Big(\Omega^{(2)}(p)+\Gamma_R^{(2)}(p)+\Gamma_R^{(3)}(p)+\Gamma_R^{(4)}(p)\Big)
\]
as $R\to +\infty$.

Let us start with re-writing
\[\tag{iii}
 \Omega_R(p)\;=\;\Omega_R^{(1)}(p)+\Omega_R^{(2)}(p)\,,
\]
where
\[
 \begin{split}
  \Omega_R^{(1)}(p)\;&:=\;\frac{1}{(2\pi^2)^{\frac{3}{2}}}\int_{|q|<R}\frac{\ud q}{\,q^2+\lambda\,}\!\int_0^{+\infty}\!\!\ud t\,\frac{t^{-\frac{1}{4}}}{\,(p^2+\lambda+t)(q^2+\lambda+t)\, } \\
  \Omega_R^{(2)}(p)\;&:=\;\frac{1}{(2\pi^2)^{\frac{3}{2}}}\int_{|q|<R}\frac{\ud q}{\,q^2+\lambda\,}\!\int_0^{+\infty}\!\!\ud t\,\frac{t^{-\frac{1}{4}} (\sqrt{t}-\sqrt{\lambda+t}-4\pi\alpha)}{\,(4\pi\alpha+\sqrt{\lambda+t})(p^2+\lambda+t)(q^2+\lambda+t)\, }\,.
 \end{split}
\]
Moreover, by means of the identity
\begin{equation*}
\begin{split}
\int_0^{+\infty}\!\!\ud t\,\frac{t^{-\frac{1}{4}}}{(C+t)(D+t)}\;&= 
\;\frac{\pi\sqrt{2}}{\,C^{\frac{1}{4}}(\sqrt{C}+\sqrt{D})(C^{\frac{1}{4}}+D^{\frac{1}{4}})\,D^{\frac{1}{4}}}\qquad (C,D\geqslant 0)
\end{split}
\end{equation*}
we re-write
\begin{equation*}
\begin{split}
\Omega_R^{(1)}(p)\;&=\;\frac{1}{\,2\pi^2(p^2+\lambda)^{\frac{1}{4}}}\;\times \\
&\qquad\times\int_{|q|<R}\frac{\ud q}{(\sqrt{p^2+\lambda}+\sqrt{q^2+\lambda})((p^2+\lambda)^{\frac{1}{4}}+(q^2+\lambda)^{\frac{1}{4}})(q^2+\lambda)^{\frac{5}{4}}} \\
&=\;A^{-\frac{1}{2}}I_{\lambda,R}(A)\,,
\end{split}
\end{equation*}
having set $A:=\sqrt{p^2+\lambda}$ and
\begin{equation*}
I_{\lambda,R}(A)\;:=\frac{1}{\,2\pi^2}\int_{|q|<R}\frac{\ud q}{(A+\sqrt{q^2+\lambda})(\sqrt{A}+(q^2+\lambda)^{\frac{1}{4}})(q^2+\lambda)^{\frac{5}{4}}}\,.
\end{equation*}
Let us then split
\[
 \begin{split}
  \Omega_R^{(1)}(p)\;&=\;A^{-\frac{1}{2}}I_{\lambda,R}(A) \\
  &=\;\mathbf{1}_{\{|p|<R\}}A^{-\frac{1}{2}}I_{0,+\infty}(A) \\
  &\qquad +\mathbf{1}_{\{|p|<R\}}A^{-\frac{1}{2}}\big(I_{\lambda,+\infty}(A)-I_{0,+\infty}(A)\big) \\
  &\qquad +\mathbf{1}_{\{|p|<R\}}A^{-\frac{1}{2}}\big(I_{\lambda,R}(A)-I_{\lambda,+\infty}(A)\big) \\
  &\qquad +\mathbf{1}_{\{|p|>R\}}A^{-\frac{1}{2}}I_{\lambda,R}(A) \\
  &\equiv\;\Gamma_R^{(1)}(p)+\Gamma_R^{(2)}(p)+\Gamma_R^{(3)}(p)+\Gamma_R^{(4)}(p)\,.
 \end{split}
\]
Since
\begin{align*}
I_{0,+\infty}(A)\;&=\;\frac{1}{2\pi^2}\int_{\mathbb{R}^3}\frac{\ud q}{(A+q)(\sqrt{A}+\sqrt{q})\,q^{\frac{5}{2}}}\;=\;\frac{2}{\pi}\int_0^{+\infty}\!\frac{\ud r}{(A+r)(\sqrt{A}+\sqrt{r})\,\sqrt{r\,}}\\
&=\;\frac{2}{\pi A}\int_0^{+\infty}\!\frac{\ud r}{(1+r)(1+\sqrt{r})\,\sqrt{r\,}}\;=\;\frac{1}{A}
\end{align*}
and hence
\[
 \Gamma_R^{(1)}(p)\;=\;\mathbf{1}_{\{|p|<R\}}A^{-\frac{1}{2}}I_{0,+\infty}(A)\;=\;\frac{\mathbf{1}_{\{|p|<R\}}}{\,(p^2+\lambda)^{\frac{3}{4}}}\,,
\]
then an exact cancellation yields
\[\tag{iv}
  (p^2+\lambda)^{-\frac{3}{4}}\mathbf{1}_{\{|p|<R\}}-\Omega_R^{(1)}\;=\;-\big(\Gamma_R^{(2)}(p)+\Gamma_R^{(3)}(p)+\Gamma_R^{(4)}(p)\big)\,.
\]
Thus, (iii) and (iv) prove (i).

Let us proceed proving the statements in (ii).
Since
\[
 |\sqrt{t\,}-\sqrt{\lambda+t\,}-4\pi\alpha|\;\leqslant\;\sqrt{\lambda}+4\pi\alpha
\]
and
\[
 q^2+\lambda+t\;\gtrsim\;(q^2+\lambda)^{\frac{5}{8}}\,t^{\frac{3}{8}}\,,
\]
then
\[
  |\Omega_R^{(2)}(p)|\;\lesssim\;\frac{1}{p^2+\lambda}\int_{|q|<R}\frac{\ud q}{\,(q^2+\lambda)^\frac{13}{8}}\!\int_0^{+\infty}\!\!\ud t\,\frac{\ud t}{\,t^{\frac{1}{4}}\,\sqrt{\lambda+t}\;t^{\frac{3}{8}} }\;\lesssim\;\frac{1}{p^2+\lambda}\,.
\]
This, and the positivity of $\Omega^{(2)}$ shows at once (by dominated convergence) that  $\Omega^{(2)}$ belongs to $L^2(\mathbb{R}^3)$ and converges strongly to a $L^2$-function as $R\to +\infty$. The first statement in (ii) is proved.

Next, we analyse $\Gamma_R^{(2)}$. We find
\[
 \begin{split}
  -\Gamma_R^{(2)}(p)\;&=\; \mathbf{1}_{\{|p|<R\}}A^{-\frac{1}{2}}\big(I_{0,+\infty}(A)-I_{\lambda,+\infty}(A)\big) \\
  &=\;\frac{\,\mathbf{1}_{\{|p|<R\}}\,}{2\pi^2}\int_{\mathbb{R}^3}\ud q\,\Big(\frac{1}{(A+q)(\sqrt{A}+\sqrt{q})\,q^{\frac{5}{2}}}\,- \\
  &\qquad\qquad\qquad\qquad  -\frac{1}{(A+\sqrt{q^2+\lambda})(\sqrt{A}+(q^2+\lambda)^{\frac{1}{4}})(q^2+\lambda)^{\frac{5}{4}}}\Big) \\
  &=\;\frac{\,\mathbf{1}_{\{|p|<R\}}\,}{2\pi^2}\int_{\mathbb{R}^3}\frac{N_1(A,q)+N_2(A,q)+N_3(A,q)+N_4(A,q)}{D(A,q)}\,\ud q\,,
 \end{split}
\]
where
\[
\begin{array}{lll}
 N_1(A,q)\;:=\;A^{\frac{3}{2}}\,((q^2+\lambda)^{\frac{5}{4}}-q^{\frac{5}{2}})\,, & & N_2(A,q)\;:=\;A((q^2+\lambda)^{\frac{3}{2}}-q^3)\,, \\
 N_3(A,q)\;:=\;A^{\frac{1}{2}}((q^2+\lambda)^{\frac{7}{4}}-q^{\frac{7}{2}})\,, & &
 N_4(A,q)\;:=\;(q^2+\lambda)^{2}-q^{4}\,,
\end{array}
\]
and
\[
 D(A,q)\;:=\;(A+\sqrt{q^2+\lambda})(\sqrt{A}+(q^2+\lambda)^{\frac{1}{4}})(A+q)(\sqrt{A}+\sqrt{q})\,(q^2+\lambda)^{\frac{5}{4}}q^{\frac{5}{2}}\,.
\]
Clearly, $D(A,q)\geqslant 0$ and $N_j(A,q)\geqslant 0$, $j\in\{1,2,3,4\}$.
From
\begin{equation*}
N_1(A,q)\lesssim A^{\frac{3}{2}}\langle q\rangle^{\frac{1}{2}},\quad N_2(A,q)\lesssim A\langle q\rangle,\quad N_3(A,q)\lesssim A^{\frac{1}{2}}\langle q\rangle^{\frac{3}{2}},\quad D(A,q)\gtrsim A^{3}q^{\frac{5}{2}}\langle q\rangle^{\frac{5}{2}}
\end{equation*}
we deduce
\[
\begin{split}
\int_{\mathbb{R}^3}\frac{N_1(A,q)+N_2(A,q)+N_3(A,q)}{D(A,q)}\,\ud q\;&\lesssim\;\int_{\mathbb{R}^3}\frac{A^{\frac{3}{2}}\langle q\rangle^{\frac{1}{2}}+A\langle q\rangle+ A^{\frac{1}{2}}\langle q\rangle^{\frac{3}{2}}}{A^{3}q^{\frac{5}{2}}\langle q\rangle^{\frac{5}{2}}}\,\ud q\\
&\lesssim\; A^{-\frac{3}{2}}\int_{\mathbb{R}^3}\frac{\ud q}{\,q^{\frac{5}{2}}\langle q\rangle}\;\lesssim\; A^{-\frac{3}{2}}\,,
\end{split}
\]
and from
\begin{equation*}
N_4(A,q)\lesssim \langle q\rangle^2,\quad D(A,q)\geqslant A^2q^{\frac{5}{2}}\langle q\rangle^{\frac{7}{2}}
\end{equation*}
we also deduce
\begin{equation*}
\int_{\mathbb{R}^3}\frac{N_4(A,q)}{D(A,q)}\,\ud q\;\lesssim\; A^{-2}\!\int_{\mathbb{R}^3}\frac{\ud q}{\,q^{\frac{5}{2}}\langle q\rangle^{\frac{3}{2}}}\;\lesssim\; A^{-2}\,.
\end{equation*}
Therefore,
\begin{equation*}
\Gamma_R^{(2)}(p)\;\lesssim\; \mathbf{1}_{\{|p|<R\}}\,A^{-\frac{1}{2}}(A^{-\frac{3}{2}}+A^{-2})\;\lesssim\; \mathbf{1}_{\{|p|<R\}}\,(p^2+\lambda)^{-1}\,.
\end{equation*}
This bound shows, again by dominated convergence, that $\Gamma_R^{(2)}$ belongs to $L^2(\mathbb{R}^3)$ and converges strongly to a $L^2$-function as $R\to +\infty$, thus proving the second statement in (ii).

Let us now analyse $\Gamma_R^{(3)}$. One has
\[
\begin{split}
 &|\Gamma_R^{(3)}(p)|\;=\;\mathbf{1}_{\{|p|<R\}}A^{-\frac{1}{2}}\big(I_{\lambda,+\infty}(A)-I_{\lambda,R}(A)\big) \\
 &=\;\frac{\,\mathbf{1}_{\{|p|<R\}}\,}{\,(p^2+\lambda)^{\frac{1}{4}}}\;\frac{2}{\pi}\int_{R}^{+\infty}\!\!\frac{r^2\,\ud r}{(\sqrt{p^2+\lambda}+\sqrt{r^2+\lambda})((p^2+\lambda)^{\frac{1}{4}}+(r^2+\lambda)^{\frac{1}{4}})(r^2+\lambda)^{\frac{5}{4}}} \\
 &\leqslant\;\frac{\,\mathbf{1}_{\{|p|<R\}}\,}{\,(p^2+\lambda)^{\frac{3}{4}}}\;\frac{2}{\pi}\int_{\frac{R}{\sqrt{p^2+\lambda}}}^{+\infty}\frac{\ud r}{(1+r)(1+\sqrt{r})\sqrt{r}\,} \\
 &=\;\frac{\,\mathbf{1}_{\{|p|<R\}}\,}{\,(p^2+\lambda)^{\frac{3}{4}}}\;L\big({\textstyle \frac{R}{\sqrt{p^2+\lambda}}}\big)\,,
\end{split}
\]
having set
\[
 L(s)\;:=\;\frac{2}{\pi}\int_{s}^{+\infty}\!\!\frac{\ud r}{(1+r)(1+\sqrt{r})\sqrt{r}\,}\;=\;\frac{2}{\pi}\arctan\frac{1}{\sqrt{s}\,}-\frac{1}{\pi}\ln\,\frac{(1+\sqrt{s})^2}{1+s}\,.
\]
It is straightforward to see that $s\mapsto L(s)$ is strictly monotone for $s\in[0,+\infty)$, with $L(0)=0$ and with the asymptotics
\[
 L(s)\;=\;{\textstyle\frac{2}{\pi}}s^{-1}-{\textstyle\frac{4}{3\pi}}\,s^{-\frac{3}{2}}+O(s^{-3})\qquad\textrm{ as }s\to +\infty\,.
\]
Therefore
\[
 L(s)\;\lesssim\;\frac{1}{s+1}\,,\qquad s\in[0,+\infty)\,,
\]
and plugging the latter bound into the above estimate for $|\Gamma_R^{(3)}(p)|$ yields
\[
 |\Gamma_R^{(3)}(p)|\;\lesssim\;\frac{\mathbf{1}_{\{|p|<R\}}}{\,(p^2+\lambda)^{\frac{1}{4}}\,(R+\sqrt{p^2+\lambda})\,}\,.
\]
This implies
\[
 \begin{split}
  \int_{\mathbb{R}^3}|\Gamma_R^{(3)}(p)|^2\,\ud p\;&=\int_{|p|<R}\frac{\ud p}{\,\sqrt{p^2+\lambda}\,(R+\sqrt{p^2+\lambda})^2\,} \\
  &\leqslant\;4\pi\int_0^R\!\frac{r}{(R+r)^2}\,\ud r\;=\;4\pi(\ln 2-{\textstyle\frac{1}{2}})\,,
 \end{split}
\]
which shows that $\Gamma_R^{(3)}$ belongs to $L^2(\mathbb{R}^3)$ with norm uniformly bounded in $R$. Moreover, $\Gamma_R^{(3)}(p)\to 0$ point-wise s $R\to +\infty$. These two properties together imply that $\Gamma_R^{(3)}\to 0$ weakly in $L^2(\mathbb{R}^3)$. The third statement in (iii) is proved.

Last, let us analyse $\Gamma_R^{(4)}$. By definition $\Gamma_R^{(4)}(p)=\mathbf{1}_{\{|p|>R\}}A^{-\frac{1}{2}}I_{\lambda,R}(A)\to 0$ point-wise as $R\to +\infty$. Moreover,
\[
\begin{split}
 &\Gamma_R^{(4)}(p)\;=\;\mathbf{1}_{\{|p|>R\}}A^{-\frac{1}{2}}I_{\lambda,R}(A) \\
 &=\;\frac{\,\mathbf{1}_{\{|p|>R\}}\,}{\,(p^2+\lambda)^{\frac{1}{4}}}\;\frac{2}{\pi}\int_{0}^{R}\frac{r^2\,\ud r}{(\sqrt{p^2+\lambda}+\sqrt{r^2+\lambda})((p^2+\lambda)^{\frac{1}{4}}+(r^2+\lambda)^{\frac{1}{4}})(r^2+\lambda)^{\frac{5}{4}}} \\
 &\leqslant\;\frac{\,\mathbf{1}_{\{|p|>R\}}\,}{\,(p^2+\lambda)^{\frac{3}{4}}}\;\frac{2}{\pi}\int_{0}^{\frac{R}{\sqrt{p^2+\lambda}}}\frac{\ud r}{(1+r)(1+\sqrt{r})\sqrt{r}\,} \\
 &=\;\frac{\,\mathbf{1}_{\{|p|>R\}}\,}{\,(p^2+\lambda)^{\frac{3}{4}}}\;\big(1-L\big({\textstyle \frac{R}{\sqrt{p^2+\lambda}}}\big)\big)\,.
\end{split}
\]
From
\[
 1-L(s)\;=\;\frac{2}{\pi}\arctan\sqrt{s}+\frac{1}{\pi}\ln\,\frac{(1+\sqrt{s})^2}{1+s}
\]
it is straightforward to see that $s\mapsto 1-L(s)$ is strictly monotone increasing for $s\in[0,+\infty)$, with
asymptotics $1-L(s)\to 1$ as $s\to +\infty$ and
\[
 1-L(s)\;=\;{\textstyle\frac{4}{\pi}}\,s^{\frac{1}{2}}-{\textstyle\frac{2}{\pi}}\,s+{\textstyle\frac{4}{5\pi}}\,s^{\frac{5}{2}}+o(s^{\frac{5}{2}})\qquad\textrm{ as }s\downarrow 0\,.
\]
In particular,
\[
 1-L(s)\;\leqslant\;\frac{4}{\pi}\,\sqrt{s}\,,\qquad 0\leqslant s\leqslant 2\,.
\]
Now, since $\Gamma_R^{(4)}$ has support in $|p|\geqslant R>\sqrt{\frac{R^2}{4}-\lambda}$ (having taken $R$ eventually large enough), whence $s=\frac{R}{\sqrt{p^2+\lambda}\,}\leqslant 2$, the above lower bound for $1-L(s)$ reads
\[
 1-L\big({\textstyle \frac{R}{\sqrt{p^2+\lambda}}}\big)\;\leqslant\;\frac{4\sqrt{R}}{\,\pi(p^2+\lambda)^{\frac{1}{4}}}\,,\qquad |p|>R\,,
\]
which implies that the estimate on $\Gamma_R^{(4)}$ can now be completed as
\[
 \Gamma_R^{(4)}(p)\;\leqslant\;\frac{\,\mathbf{1}_{\{|p|>R\}}\,}{\,(p^2+\lambda)^{\frac{3}{4}}}\;\big(1-L\big({\textstyle \frac{R}{\sqrt{p^2+\lambda}}}\big)\big)\;\lesssim\;\sqrt{R}\;\frac{\,\mathbf{1}_{\{|p|>R\}}\,}{\,p^2+\lambda\,}\,.
\]
Therefore,
\[
\begin{split}
 \int_{\mathbb{R}^3}|\Gamma_R^{(4)}(p)|^2\,\ud p\;&\lesssim\;R\int_{|p|>R}\frac{\ud p}{\,(p^2+\lambda)^2}\;\lesssim\;R\int_R^{+\infty}\!\frac{\ud r}{\,r^2}\;=\;1\,.
\end{split}
\]
The latter estimate shows that $\Gamma_R^{(4)}$ belongs to $L^2(\mathbb{R}^3)$ with norm bounded  uniformly in $R$. This, together with the fact that $\Gamma_R^{(4)}\to 0$ point-wise in $p$, implies that  $\Gamma_R^{(4)}$ has a unique weak limit in $L^2(\mathbb{R}^3)$ and this limit is $0$.
The fourth statement in (ii) and hence the whole (ii) is then established, and the proof is concluded.
\end{proof}

%
%
%



\def\cprime{$'$}

\end{document}